\documentclass[a4paper,12pt]{article}
\usepackage{amsthm,amsfonts,amssymb,euscript}
\usepackage{latexsym, multicol, fancybox}
\usepackage{graphicx}
\usepackage{color}
\usepackage{amsmath, amsthm, amssymb, bm}
\usepackage{epstopdf}
\usepackage{caption}
\usepackage{psfrag}

\DeclareFontFamily{U}{mathx}{\hyphenchar\font45}
\DeclareFontShape{U}{mathx}{m}{n}{
      <5> <6> <7> <8> <9> <10>
      <10.95> <12> <14.4> <17.28> <20.74> <24.88>
      mathx10
      }{}
\DeclareSymbolFont{mathx}{U}{mathx}{m}{n}
\DeclareFontSubstitution{U}{mathx}{m}{n}
\DeclareMathAccent{\widecheck}{0}{mathx}{"71}

\def\ombc{\widecheck{\omb}}

\newtheorem{theorem}{Theorem}[section]
\newtheorem{lemma}[theorem]{Lemma}
\newtheorem{proposition}[theorem]{Proposition}
\newtheorem{corollary}[theorem]{Corollary}
\newtheorem{definition}[theorem]{Definition}
\newtheorem{remark}[theorem]{Remark}

\numberwithin{equation}{section}

\newcommand{\bea}{\begin{eqnarray}}
\newcommand{\eea}{\end{eqnarray}}
\def\beaa{\begin{eqnarray*}}
\def\eeaa{\end{eqnarray*}}
\def\ba{\begin{array}}
\def\ea{\end{array}}
\def\be#1{\begin{equation} \label{#1}}
\def \eeq{\end{equation}}
\def\bsplit{\begin{split}}
\newcommand{\nn}{\nonumber}

\def\lab{\label}
\def\les{\lesssim}

\def\c{\cdot}

\def\dual{{\,^\star \mkern-2mu}}

\def\tr{\mbox{tr}}

\def\ov{\overline}

\renewcommand{\div}{\mbox{div}\,}
\newcommand{\curl}{\mbox{curl}\,}
\def\nab{\nabla}
\def\lap{\Delta}
\def\pr{\partial}

\def\dkb{ \, \mathfrak{d}     \mkern-9mu /}
\def\dk{\mathfrak{d}}

\def\rhod{ \,^\star  \hspace{-2.2pt} \rho}

\def\ovu{\overset{\circ}{ u}}

\def\ovs{\overset{\circ}{ s}}

\def\ovr{\overset{\circ}{ r}}
\def\ovla{\protect\overset{\circ}{ \la}\,}

\def\ovPhi{\protect\overset{\circ}{\Phi}\,}
\def\ovphi{\protect\overset{\circ}{\phi}\,}

\def\ovS{\overset{\circ}{ \S}}
\def\ovJ{\overset{\,\,\circ}{J}}

\def\epg{\protect\overset{\circ}{\ep}}  

\def\ug{\overset{\circ}{u}}   
\def\sg{\overset{\circ}{s}}               
\def\rg{\overset{\circ}{r}}   
 
\def\mg{\overset{\circ}{m}}   

\def\ovu{\overset{\circ}{ u}}

\def\ovs{\overset{\circ}{ s}}

\def\ovr{{\overset{\circ}{ r}\,}}
\def\ovm{{\overset{\circ}{ m}\,}}
\def\ovS{\overset{\circ}{ S}}

\def\ovg{\overset{\circ}{ g}}

  \def\Un{U^{(n)}}

\def\trchbS{{\trchb^\S}}

\def\a{\alpha}
\def\b{\beta}
\def\ga{\gamma}
\def\Ga{\Gamma}
\def\de{\delta}
\def\De{\Delta}
\def\ep{\epsilon}
\def\la{\lambda}
\def\La{\Lambda}

\def\om{\omega}

\def\rhod{\dual \rho}
\def\th{{\theta}}
\def\Th{{\Theta}}
\def\ka{\kappa}
\def\ze{\zeta}
\def\Up{\Upsilon}

\def\Th{\Theta}

\def\vphi{{\varphi}}

\def\vsi{\varsigma}

\renewcommand{\aa}{\protect\underline{\a}}
\newcommand{\bb}{\protect\underline{\b}}
\def\omb{\protect\underline{\om}}

\def\Omb{\underline{\Omega}}

\newcommand{\chib}{\protect\underline{\chi}}
\newcommand{\xib}{\protect\underline{\xi}}
\newcommand{\etab}{\protect\underline{\eta}}

\def\kab{\protect\underline{\kappa}}

\def\DD{{\mathcal D}}

\def\GG{{\mathcal G}}
\def\HH{{\mathcal H}}
\def\II{{\mathcal I}}

\def\KK{{\mathcal K}}

\def\MM{{\mathcal M}}
\def\NN{{\mathcal N}}

\def\QQ{{\mathcal Q}}
\def\RR{{\mathcal R}}
\def\SS{{\mathcal S}}

\def\UU{{\mathcal U}}

\def\D{{\bf D}}

\def\J{{\bf J}}

\def\M{{\bf M}}

\def\R{{\bf R}}

\def\S{{\bf S}}

\def\g{{\bf g}}

\def\fb{\protect\underline{f}}

\def\Cb{{\underline{C}}}

\def\CCC{{\Bbb C}}

\def\RRR{{\Bbb R}}
\def\SSS{{\Bbb S}}

\def\hk{\mathfrak{h}}

\def\rhoc{\check \rho}

\def\kac{\widecheck \ka}
\def\kabc{\widecheck{\underline{\ka}}}

\def\ombc{\underline{\check \omega}}

\def\chih{\widehat{\chi}}
\def\chibh{\widehat{\chib}}
\def\trch{\tr \chi}
\def\trchb{\tr \chib}

\def\c{\cdot}
\def \f12{\frac 1 2 }
\def\ov{\overline}

\def\Lab{\underline{\La}}

\def\epg{\protect\overset{\circ}{\ep}}  
\def\dg{\overset{\circ}{\de}}

\def\undB{\underline{B}}

\def\Un{U^{(n)}}
\def\Sn{{S^{(n)}}}
\def\fn{f^{(n)}}
\def\fbn{\fb^{(n)}}
\def\fnn{f^{(n+1)}}
\def\fbnn{\fb^{(n+1)}}
\def\ovlan{\ovla^{(n)}}

\def\Fn{F^{(n)}}

\def\Sn{ S^{(n)}}

\def\gn{g^{\S(n)}}
\def\DDn{\DD^{(n)}}

\def\divzero{{\overset{\circ}{ \div}}}

\def\lapzero{{\overset{\circ}{ \lap}}}

\def\Jp{J^{(p)}}
\def\JpS{{J^{(p, \S)}}} 
\def\JJpS{\J^{(p,\S)} }

\def\Cbp{\Cb^{(p)}} 
\def\Mp{M^{(p)}}

\def\CbpS{\Cb^{(\S, p)}} 
\def\MpS{M^{(\S, p)}}

\def\kadot{\dot{\ka}}
\def\kabdot{\dot{\kab}}
\def\mudot{\dot{\mu}}

\def\Cbdot{\dot{\Cb}}
   \def\Mdot{\dot{M}}

   \def\Cbn{\Cb^{(n)}}
   \def\Cbpn{\Cb^{(n), p}}
   \def\Mpn{M^{(n), p}}
   \def\Mn{M^{(n)} }
   
     \def\Cbndot{\Cbdot^{(n)}}
   \def\Cbpndot{\Cbdot^{(n), p}}
   \def\Mpndot{\Mdot^{(n), p}}
   \def\Mndot{\Mdot^{(n)} }

\def\gn{g^{(n)}}

\def\Psih{\widehat{\Psi}}
\def\kabcS{ \widecheck{\kab^\S}}

\def\trchbc{\widecheck{\trchb} }

\newdimen\cdsep
\def\cdstrut{\vrule height .6\cdsep width 0pt depth .4\cdsep}
\def\@cdstrut{{\advance\cdsep by 2em\cdstrut}}

\def\arrow#1#2{
  \ifx d#1
    \llap{$\scriptstyle#2$}\left\downarrow\cdstrut\right.\@cdstrut\fi
  \ifx u#1
    \llap{$\scriptstyle#2$}\left\uparrow\cdstrut\right.\@cdstrut\fi
  \ifx r#1
    \mathop{\hbox to \cdsep{\rightarrowfill}}\limits^{#2}\fi
  \ifx l#1
    \mathop{\hbox to \cdsep{\leftarrowfill}}\limits^{#2}\fi
}

\def\JpSS{\J^{(p, \S)}}
\def\Jpov{\ovJ\,^{(p)}}

\def\Psih{\widehat{\Psi}}

\begin{document}

\title{Effective results on uniformization  and intrinsic  GCM spheres in perturbations of Kerr}
\author{Sergiu Klainerman, J\'er\'emie Szeftel}

\maketitle

{\bf Abstract.} \textit{This  is a follow-up of our paper  \cite{KS-Kerr1}      on the construction of  general covariant   modulated (GCM) spheres in  perturbations of Kerr, which we expect to  play a  central  role in establishing  their  nonlinear stability.  We  reformulate the  main results of that paper  using a    canonical definition of  $\ell=1$ modes on a   $2$-sphere embedded in a  $1+3$ vacuum manifold.  This   is based on a new, effective,   version of the classical uniformization  theorem which  allows us   to   define such    modes and  prove their stability    for   spheres  with   comparable metrics.   The reformulation  allows us   to prove a second, intrinsic,  existence theorem for  GCM spheres, expressed purely in terms of geometric quantities  defined  on   it.  A natural definition of angular momentum  for such GCM  spheres is also introduced,  which we  expect  to play a key  role  in  determining  the final  angular momentum  for   general perturbations of Kerr.}

 \tableofcontents

    
  \section{Introduction}
  

  This  is a follow-up of our paper  \cite{KS-Kerr1}      on the construction of general covariant   modulated (GCM)  spheres in perturbations of Kerr.  
  We  reformulate the  main results of that paper  using a    canonical definition of  $\ell=1$ modes on a   $2$-sphere embedded in a  $1+3$ vacuum manifold.  This   is based on a new, effective,   version of the classical uniformization  theorem which  allows us   to   define such    modes and  prove their stability    for   spheres  with   comparable metrics.   The reformulation  allows us 
  to prove a second, intrinsic,   existence theorem for  GCM spheres, expressed purely in terms of geometric quantities  defined  on  it.
   A natural definition of angular momentum  for such GCM  spheres is also introduced,  which we  expect  to play a key   role  in the  definition of angular momentum  for   general perturbations of Kerr. 
  
   We refer to the introduction  in \cite{KS-Kerr1} for a  short   description of the role of GCM spheres      in the proof of the nonlinear stability of Schwarzschild in \cite{KS-Schw}  and the crucial role we expect them to play in  extending that  result  to  Kerr.

  
  \subsection{Review of the main results of \cite{KS-Kerr1}}
  

  
   \subsubsection{Background space}
  

   As in  \cite{KS-Kerr1}    we consider spacetime regions $\RR$ foliated by a geodesic foliation $S(u, s)$        induced by   an outgoing  optical function\footnote{As explained  in Remark  1.5 of  \cite{KS-Kerr1},   this is not strictly necessary.  Any other foliation satisfying comparable  asymptotic assumptions would also work.}      $u$   with $s$   a properly normalized  affine parameter along the null geodesic generators  of   $L=-\g^{\a\b}\pr_\b  u\pr_\a $. We denote by $r=r(u, s)$ the area radius of $S(u,s)$ and  let  $(e_3, e_4, e_1, e_2)$ be   an adapted null frame with $e_4$   proportional  to   $L$ and $e_1, e_2 $ tangent to  spheres $S=S(u, s)$, see section \ref{sec:backgroundspacetime}. 
   The main assumptions made in   \cite{KS-Kerr1}   were that    the Ricci   and curvature coefficients,  relative to the  adapted   null frame,  have the  same asymptotics in powers of $r$ as   in  Schwarzschild space.    Note that these  assumptions  hold true in the far region of Kerr and is  expected to hold true for  the far region of realistic perturbations of Kerr.   The actual size of the perturbation from   Kerr   is measured with respect to a small parameter $\epg>0$, see  sections  \ref{subsubsect:regionRR1} and \ref{subsubsect:regionRR2} for  precise definitions.

  
  \subsubsection{Definition of GCM spheres}  
  
  
  These are topological spheres $\S$  embedded in $\RR$  endowed  with a null frame  $e_3^\S, e_4^\S, e_1^\S, e_2^\S$  adapted to $\S$  (i.e.   $e_1^\S, e_2^\S$ tangent to $\S$),   relative
  to which     the  null expansions   $\ka^\S=\trch^\S$, $\kab^\S=\trchb^\S$ and  mass aspect function
  $\mu^\S$ take Schwarzschildian values, i.e.
  \bea
  \lab{Introd:GCMspheres1}
 \ka^\S -\frac{2}{r^\S} =0, \qquad  \kab^\S+\frac{2\Up^\S}{r^\S} =0, \qquad  \mu^\S - \frac{2m^\S}{(r^\S)^3} =0,
  \eea 
  where $r^\S, m^\S$ denote the area  radius and Hawking mass of $\S$ and $\Up^\S=1-\frac{2m^\S}{r^\S}$, see section \ref{sec:backgroundspacetime} for  precise definitions.

  As   explained in the introduction to \cite{KS-Kerr1}, these conditions have to be relaxed  with respect to the $\ell=0, 1$ modes  on $\S$.  To make sense of  this statement   requires  us to fix a basis of $\ell=1$  modes on $\S$, which generalize the $\ell=1$ spherical harmonics of 
   the standard sphere\footnote{Recall that on the standard sphere $\SSS^2$, in spherical coordinates $(\th, \vphi)$,  these are 
   $J^{(0, \SSS^2)}=\cos\th$, $J^{(+,\SSS^2)}=\sin\th\cos\vphi$, $J^{(-,\SSS^2)}=\sin\th\sin\vphi$.}. 
   Assuming  the existence of  such a basis $\JpS,   p\in\big\{ -, 0, +\big\}$,  we   define, for a scalar function $h$,
    \bea
  ( h)^\S_{\ell=1} &:=&\left\{\int_{\S} h \JpS, \quad  p=-, 0, +\right\}. 
  \eea
 A scalar function  $h$ is said to be  supported  on  $\ell\le 1$ modes, i.e.  $(f)^\S_{\ell\ge 2}=0$,  if   there exist  constants  
  $A_0, B_{-}, B_0, B_{+} $ such that
   \bea
   h= A_0+  B_{-} J^{( -, \S) }+B_0 J^{( 0, \S)} + B_{+}  J^{(+, \S)}. 
  \eea
  With this definition,  \eqref{Introd:GCMspheres1}  is replaced by
  \bea
  \lab{Introd:GCMspheres2}
 \ka^\S -\frac{2}{r^\S} =0, \qquad \left( \kab^\S+\frac{2\Up^\S}{r^\S}\right)_{\ell\ge 2 }=0, \qquad \left( \mu^\S - \frac{2m^\S}{(r^\S)^3} \right)_{\ell\ge 2}=0.
  \eea

    
\subsubsection{Deformations of  spheres and frame transformations}


The construction of   GCM spheres
in \cite{KS-Kerr1} was obtained    by deforming a given  sphere  $ \ovS=S(\ovu, \ovs)$ of the  background foliation of $\RR$. An   $O(\dg)$  deformation of   $\ovS$   is    defined by a  
map $\Psi:\ovS\to \S $   of the form
   \bea\lab{eq:definitionofthedeformationmapPsiinintroduction}
 \Psi(\ovu, \ovs,  y^1, y^2)=\left( \ovu+ U(y^1, y^2 ), \, \ovs+S(y^1, y^2 ), y^1, y^2  \right)
 \eea
  with $(U, S)$ smooth functions on $\ovS$, vanishing at a fixed point of $\ovS$,  of size  proportional to the small constant  $\dg$.
  Given such a deformation we identify, at any point on $\S$,   two  important  null frames.
  \begin{enumerate}
  \item  The null frame $( e_3, e_4, e_1, e_2)$ of the background foliation  of $\RR$.
  \item A  null frame  $( e^\S_3, e^\S_4, e^\S_1, e^\S_2)$ adapted to the sphere $\S$. 
  \end{enumerate}
     In general, two  null frames   $(e_3, e_4, e_1, e_2)$ and $(e_3', e_4', e_1', e_2') $ are related by  a  frame transformation of the form, see Lemma \ref{Lemma:Generalframetransf}, 
 \bea
 \lab{eq:Generalframetransf-intro}
 \bsplit
  e_4'&=\la\left(e_4 + f^b  e_b +\frac 1 4 |f|^2  e_3\right),\\
  e_a'&= \left(\de_{ab} +\frac{1}{2}\fb_af_b\right) e_b +\frac 1 2  \fb_a  e_4 +\left(\frac 1 2 f_a +\frac{1}{8}|f|^2\fb_a\right)   e_3,\\
 e_3'&=\la^{-1}\left( \left(1+\frac{1}{2}f\c\fb  +\frac{1}{16} |f|^2  |\fb|^2\right) e_3 + \left(\fb^b+\frac 1 4 |\fb|^2f^b\right) e_b  + \frac 1 4 |\fb|^2 e_4 \right),
 \end{split}
 \eea
  where the scalar $\la$    and the 1-forms $f$ and $\fb$      are    called the transition coefficients of the transformation.  As explained in the introduction to \cite{KS-Kerr1},  one can   relate all  Ricci and curvature  coefficients  of the primed frame in terms of  the Ricci and curvature coefficients of the  un-primed one.    In particular, the GCM  conditions \eqref{Introd:GCMspheres2}    can be expressed in terms of differential  conditions for    the transition coefficients $(f, \fb, \la) $.    The condition that  the horizontal  part of the frame $(e^\S_1, e_2^\S)$  is tangent to $\S$  also leads to a relation  between the gradients of  the scalar functions  $U, S$, introduced   in \eqref{eq:definitionofthedeformationmapPsiinintroduction},  and $(f, \fb) $.  Roughly,  we thus expect   to derive a  coupled system\footnote{See the introduction in \cite{KS-Kerr1} for further explanations.}  between Laplace equations for $U, S$, on $\ovS$    and    an elliptic Hodge  system for $F=(f, \fb, \la-1)$ on $\S$.

    
\subsubsection{GCM spheres with non canonical $\ell=1$ modes in \cite{KS-Kerr1}}

   
      Here is a short version of our main result in  \cite{KS-Kerr1}.
 \begin{theorem}[Existence of GCM spheres in \cite{KS-Kerr1}]
\lab{Theorem:ExistenceGCMS1-intro}
Let  $\RR$   be  fixed  spacetime region, endowed with an  outgoing geodesic   foliation  $S(u, s)$,   verifying 
 specific asymptotic assumptions\footnote{Compatible with small perturbations of Kerr.}   expressed  in terms  of two parameters $0<\dg\leq \epg$.  In particular we assume that
 the  GCM quantities  of the background spheres in $\RR$, i.e. 
   \bea
  \lab{Introd:GCMspheres3}
 \ka-\frac 2 r , \qquad \left( \kab+\frac{2\Up}{r}\right)_{\ell\ge 2 }, \qquad \left( \mu- \frac{2m}{r^3} \right)_{\ell\ge 2},
  \eea 
   are small with respect to the parameter $\dg$.
 Let  $\ovS=S(\ovu, \ovs)$   be  a fixed    sphere of the foliation with  $\rg$ and $\mg$ denoting respectively its area radius and  Hawking mass, with $\rg$ sufficiently large. 
 Then, 
for any fixed triplets   $\La, \Lab \in \RRR^3$  verifying
\bea
|\La|,\,  |\Lab|  &\les & \dg,
\eea
 there 
exists a unique  GCM sphere $\S=\S^{(\La, \Lab)}$, which is a deformation of $\ovS$, 
such that 
 \bea
  \lab{Introd:GCM spheres2-again}
 \ka^\S -\frac{2}{r^\S} =0, \qquad \left( \kab^\S+\frac{2\Up^\S}{r^\S}\right)_{\ell\ge 2 }=0, \qquad \left( \mu^\S - \frac{2m^\S}{(r^\S)^3} \right)_{\ell\ge 2}=0,
  \eea 
  and
  \bea
  \lab{Introd:GCM spheres-LaLab}
  (\div^\S f)_{\ell=1}=\La, \qquad (\div^\S\fb)_{\ell=1}=\Lab,
  \eea
  where $(f, \fb, \la)$ denote  the transition coefficients of the transformation \eqref{eq:Generalframetransf-intro}  from the background frame of $\RR$ to the frame adapted to $\S$.  
    \end{theorem}
    
    \begin{remark}
    \eqref{Introd:GCMspheres3},  \eqref{Introd:GCM spheres2-again} and \eqref{Introd:GCM spheres-LaLab} depend  on the  definition of $\ell=1$ modes  
    on $\ovS$ and $\S$.  In \cite{KS-Kerr1}   the choice $J^{(p, \S)}$  of $\ell=1$ modes for $\S$   was tied to   the modes 
    $\Jp$  of $\ovS$ via  the formula,  
   \bea
   \JpS=\Jp\circ\Psi^{-1},
   \eea 
   with $\Psi$ given by \eqref{eq:definitionofthedeformationmapPsiinintroduction}.    The  $\ell=1$   modes $\Jp$  on $\ovS$  were chosen such that  they verify certain natural conditions, see \eqref{eq:Jpsphericalharmonics}, but where  not uniquely defined and thus not canonical.
     \end{remark}


\subsection{Main  new  results on uniformization  of  spheres}


  One of the main  goals  of this paper is to  remove the arbitrariness   in the definition  of the $\ell=1$ modes  on the deformed  GCM spheres $\S$.  We do that by   appealing to  two new results  concerning  uniformization of nearly round spheres.
   The first  result,     based on   works\footnote{ We thank A. Chang   for bringing this works to our attention. The precise  formulation given in   Theorem  \ref {theorem-Eff.uniformalization-appendix}  does not seem  to appear in the literature we are aware of. } by   Onofri   \cite{Onofri} and   Chang-Yang     \cite{CY1}, \cite{CY2},   given  in  Theorem  \ref {theorem-Eff.uniformalization-appendix},  provides  an effective   version of  the classical  uniformization theorem.   The second result, given in Theorem   \ref{Thm:unifor-twometrics},  based   on  Theorem  \ref {theorem-Eff.uniformalization-appendix}   and previous  work\footnote{We thank C. De Lellis for bringing this paper to our attention.} of   Frisecke James  and  M\"uller \cite{FJM},  allows us to  formulate and prove a stability result  for effective uniformization
    of nearby spheres. Together with a notion of calibration, which we introduce  in section \ref{subsection:calibration},  this result allows us to define  a canonical definition of $\ell=1$ modes for  deformed spheres $\S$.


 \subsubsection{Classical uniformization}   
 
 
 According to the classical uniformization theorem, if $S$ is a closed, oriented and connected surface of genus 0 and $g^S$  is a  Riemannian metric on $S$, then,  there exists  a smooth  diffeomorphism  $\Phi:\SSS^2\to S$  and a smooth  conformal factor  
$u$ on $\SSS^2$ such that
\bea
\lab{standard-unifS-intro}
\Phi^\#( g^S) =  (r^S)^2  e^{2u}  \ga_0
\eea
where $\ga_0$ is the   canonical metric on the standard sphere $\SSS^2$ and $r^S$ the area radius  of $S$.   Unfortunately the   conformal factor $u$ is not unique and   it is thus difficult to control its size and, more importantly for our applications to GCM spheres, it does not allow one to compare the the conformal factors for two nearby spheres. We state below two results  which overcome these difficulties. Both results are applicable   to almost round spheres, i.e.  closed  $2$-surfaces $S$ with Gauss curvature $K^S$ sufficiently close to 
    that of a round sphere. More precisely, we have for some sufficiently small $\ep>0$
    \bea
    \lab{eqAlmostroound-Intro} 
    \left|K^S-\frac{1}{( r^S)^2} \right|\le \frac{\ep}{( r^S)^2}. 
    \eea

     In what follows we state, in a simplified version,
    our main new results on  effective uniformization.  


    \subsubsection{Effective uniformization}
    

    \begin{theorem}[Effective uniformization 1]
    \lab{Thm:effectiveU1-Intro}
    Given an   almost round  sphere  $(S, g^S)$ as above  there exists, up to isometries\footnote{i.e.  all  the solutions are of the form $(\Phi\circ O, u\circ O)$ for $O\in O(3)$.} of $\SSS^2$,  a  unique   diffeomorphism   $\Phi:\SSS^2\to S$  and a  unique  conformal factor $u$ such that  
   \bea
   \bsplit
    \Phi^\#( g^S)& =  (r^S)^2e^{2u}  \ga_{\SSS^2},\\
     \int_{\SSS^2} e^{2u} x^i &=0, \qquad i=1,2,3.
    \end{split}
    \eea    
    Moreover  the  size of the conformal factor $u$ 
    is small with respect to the parameter $\ep$, i.e.
    \beaa
     \|  u\|_{L^\infty  (\SSS^2) } \les \ep.
    \eeaa 
    \end{theorem}
      The result is restated in   Theorem \ref{theorem-Eff.uniformalization-appendix},  in the case  of radius 1 spheres,   and  in general  in Corollary \ref{proposition:effective-uniformisation}.

    We   use this effective uniformization to define  a basis of $\ell=1$ modes on $S$ as follows
\bea
\lab{intro-canmodes-intro}
 J^{(p,S)}&=&J^{(p, \SSS^2)}\circ \Phi^{-1}
\eea
where  $J^{(p, \SSS^2)}$ are  the standard  $\ell=1$ spherical harmonics of   the round sphere $\SSS^2$.
We note however  that  this definition is still ambiguous with  respect to arbitrary  rotations of $\SSS^2$.  We can remove  this arbitrariness   for  deformations  $\S$ of  $\ovS$ by  calibrating  the effective uniformization  of $\S$ with that of $\ovS$, see the discussion in   section \ref{subsection:calibration}. This requires however our stable  uniformization  result which we  discuss below.


 \subsubsection{Stability of   the effective uniformization}
 

  \begin{theorem}[Effective uniformization 2]
  \lab{Thm:effectiveU2-Intro}
  Consider   two almost round spheres   $(S_1, g^{S_1} )$ and $(S_2, g^{S_2} )$ and their respective 
  canonical   uniformizations  $(\Phi_1, u_1)$, $(\Phi_2, u_2)$ provided by   the effective uniformization result mentioned above, and a smooth map $\Psi:S_1\to S_2$. 
   Assume also    that the metrics $g^{S_1} $ and $\Psi^\#( g^{S_2} )$
  are close to each other in  $S^1$, 
   \bea
   \lab{eq:intro.unifor-twometricsB}
 \big( r^{S_1}\big)^{-2}  \| g^{S_1}-  \Psi^\#( g^{S_2} )\|_{L^\infty(S^1)} \le  \de
  \eea
  relative to a  small parameter $\de>0$.
Then,   there exists $O\in O(3)$ such that  the map  $\Psih:=(\Phi_2)^{-1}\circ  \Psi\circ \Phi_1$  verifies,
 \bea
 \lab{eq:intro.Psih-Id}
  \| \Psih-O\|_{L^\infty(\SSS^2)}&\les &\de.
 \eea
 Moreover  the conformal factors $u_1, u_2$ verify
  \bea
 \big \|u_1- \Psih^\# u_2\big \|_{L^\infty(\SSS^2)} \les \de.
  \eea
  \end{theorem}
  The result is restated in   Theorem    \ref{Thm:unifor-twometrics}.  Its proof is   based on  a   result   due to G. Friesecke, R. James  and  S. M\"uller \cite{FJM} which improves on a classical result of F. John  \cite{John}.
  
To   avoid the indeterminacy with respect to $O\in O(3)$, we introduce a notion of calibration between $(\Phi_1, v_1) $ and $(\Phi_2, v_2)$  with respect to the map $\Psi$,  see   section \ref{subsection:calibration}.  Under this notion of calibration we  can replace \eqref{eq:intro.Psih-Id} by,
\bea
\| \Psih-I\|_{L^\infty(\SSS^2)}&\les &\de.
\eea


 \subsubsection{Stability of  the   canonical $\ell=1$ modes} 
 
 
 Consider   two almost round spheres   $(S_1, g^{S_1} )$ and $(S_2, g^{S_2} )$ and their respect 
  canonical   uniformizations  $(\Phi_1, u_1)$, $(\Phi_2, u_2)$ provided by   the effective uniformization result mentioned above, and a smooth map $\Psi:S_1\to S_2$. 
   Assume also    that the metrics $g^{S_1} $ and $\Psi^\#( g^{S_2} )$
  are close to each other  in  $S^1$, i.e.,
   \bea
   \lab{eq:intro.unifor-twometricsB:bis}
  \| g^{S_1}-  \Psi^\#( g^{S_2} )\|_{H^{2}(S^1)} \le \de.
  \eea
  Assume also that the uniformization maps of $S_1, S_2$ are calibrated,  as in  section \ref{subsection:calibration}
   and let 
  \beaa
  J^i=J^{S_i}= J\circ \Phi_i^{-1}, \qquad i=1,2,
  \eeaa
  be the $\ell=1$     canonical modes of  
   $S_1, S_2 $. Then
   \bea
  \Big| J^1- J^2\circ \Psi  \Big| &\les &\de.
   \eea
   The result, stated in  Proposition \ref{Prop:comparison.canJ},    follows from    Corollary \ref{Cor:unifor-twometrics-calibrated}  which is an immediate consequence 
    of    Theorem \ref{Thm:unifor-twometrics}   and    calibration.


  \subsection{Construction of GCM spheres with canonical $\ell=1$ modes}
  

  
  \subsubsection{Canonical version of  Theorem \ref{Theorem:ExistenceGCMS1-intro}}
  
  
  \begin{theorem}[GCM spheres with canonical $\ell=1$ modes]
\lab{Theorem:ExistenceGCMS1-can-intro}  
Under the same assumptions as  in Theorem \ref{Theorem:ExistenceGCMS1-intro}  
 there exists a unique\footnote{Up to a rotation of $\SSS^2$.}  GCM sphere $\S=\S^{(\La, \Lab)}$, which is a deformation of $\ovS$, 
such that 
 \bea
  \lab{Introd:GCM spheres2-again'}
 \ka^\S -\frac{2}{r^\S} =0, \qquad \left( \kab^\S+\frac{2\Up^\S}{r^\S}\right)_{\ell\ge 2 }=0, \qquad \left( \mu^\S - \frac{2m^\S}{(r^\S)^3} \right)_{\ell\ge 2}=0,
  \eea 
  and
  \bea
  \lab{Introd:GCM spheres-LaLab'}
  (\div^\S f)_{\ell=1}=\La, \qquad (\div^\S\fb)_{\ell=1}=\Lab,
  \eea
  with canonically defined $\ell=1$ 
  modes, i.e.  $ \J^{(p,S)}=J^{(p,\SSS^2)}\circ \Phi^{-1}$, see   \eqref{intro-canmodes-intro}.  
  \end{theorem} 
  The result is restated in  Theorem  \ref{Theorem:ExistenceGCMS1-2}.

  
  \subsubsection{Intrinsic GCM spheres}
 
  
  Using  Theorem \ref{Theorem:ExistenceGCMS1-can-intro}      we  prove an intrinsic\footnote{Without reference  to the background foliation.}   result  on the    existence  and uniqueness of     GCM  surfaces  where we replace  the $\ell=1$ conditions \eqref{Introd:GCM spheres-LaLab'}  on $\div^\S(f)$ and  $\div^\S(\fb)$ by the  vanishing of the $\ell=1$ modes of $\b^\S$ and $ \widecheck{\trchbS}$.
   
  \begin{theorem}[Intrinsic GCM spheres with canonical $\ell=1$ modes]  
  \lab{theorem:ExistenceGCMS2-Intro}
  Under  slightly stronger assumptions  on the background foliation of $\RR$  there 
   exists a  unique\footnote{Up to a rotation of $\SSS^2$.}  GCM deformation of $\ovS$  verifying, in addition to   \eqref{Introd:GCM spheres2-again'}, 
    \bea
  ( \div^\S \b^\S)_{\ell=1}=0, \qquad  \widecheck{\trchbS}_{\ell=1}=0,
   \eea
   relative to the canonical $\ell=1$ modes of $\S$. 
  \end{theorem} 
  The result is restated in  Theorem  \ref{theorem:ExistenceGCMS2}. 
 
 \begin{remark}
The assumptions on the spacetime region $\RR$ in Theorem \ref{theorem:ExistenceGCMS2-Intro} are in particular satisfied in Kerr for $r$ sufficiency large, see Lemma \ref{lemma:controlfarspacetimeregionKerrassumptionRR}. We can thus apply Theorem \ref{theorem:ExistenceGCMS2-Intro} in that context, and obtain the existence of intrinsic GCM spheres $\S_{Kerr}$ in Kerr for $r$ sufficiency large, see Corollary \ref{cor:ExistenceGCMS1inKerr}. The intrinsic GCM spheres  $\S$ of Theorem \ref{theorem:ExistenceGCMS2-Intro} thus correspond to the analog of $\S_{Kerr}$ in perturbations of Kerr for $r$ sufficiency large.
 \end{remark}

  
  \subsubsection{Definition of angular momentum}
  
  
  The result  of Theorem \ref{theorem:ExistenceGCMS2-Intro}  is unique up to a rotation 
  of $\SSS^2$ in the definition of the canonical  $\ell=1$ modes on $\S$, see \eqref{intro-canmodes-intro}.  We can remove this final  ambiguity\footnote{The only  case where the ambiguity cannot be removed is when \eqref{eq:removalofambiguityrotationthankstocurlbeta-intro} holds  for any  one choice of canonical modes, i.e. $(\curl\b^\S)_{\ell=1}=0$. In that  case,  we set $a^\S=0$ which is consistent with \eqref{eq:definitionofangularmomentumaSonthesphereS-intro}.}
  by    adjusting  the choice of canonical  modes  $\J^{(p, \S)} $  on $\S$,  performing a suitable rotation in $\SSS^2$,   such that 
   \bea\lab{eq:removalofambiguityrotationthankstocurlbeta-intro}
   \int_\S  \curl^\S \b^\S\, \J^{(\pm, \S)}=0.
   \eea
   This allows us to define on $\S$   the angular   parameter   $a^\S$  by the  formula\footnote{Note that in a Kerr space  $\KK(a, m)$, relative to  a geodesic foliation normalized on $\II^+$,   we have   $  \int_\S  \curl^\S  \b^\S  \J^{(\pm, \S)}=0$ and    $  \int_\S  \curl^\S  \b^\S \J^{(0, \S)} =\frac{8\pi a m}{  (r^\S)^3}+O(\frac{ma^2}{(r^\S)^4})$.} 
   \bea\lab{eq:definitionofangularmomentumaSonthesphereS-intro}
 a^\S:=\frac{(r^\S)^3}{8\pi m^\S}   \int_\S  \curl^\S  \b^\S \J^{(0, \S)}.  
   \eea
   The  result is stated more precisely in Corollary \ref{Corr:ExistenceGCMS2}.
   
\begin{remark}
 Given its relation with GCM spheres,      we expect  that the    definition given above will  play a key   role  in determining the final  angular momentum  for   general perturbations of Kerr. 
We also note that other definitions of angular momentum have been  proposed in the literature, see 
\cite{Sz} for a comprehensive  review, and    \cite{Rizzi} and \cite{Chen} for recent  interesting proposals.  
\end{remark}

  
      \subsection{Structure of the paper}


The structure of the paper is as follows.
\begin{itemize}
\item In section 2,  we review some of the classic results on  the uniformization theorem for  spheres.
\item In section 3,  we  provide a proof of  our first effective uniformization  result   stated in Theorem \ref{Thm:effectiveU1-Intro} and  use it to define  canonical basis of $\ell=1$ modes.
\item 
In Section 4,  we prove    our  main   stability result   for effective uniformization stated in Theorem \ref{Thm:effectiveU2-Intro}.
\item  In Section 5,  we review the geometric set-up  we need to state our GCM results.
\item In Section 6, we   review   the  main result  of \cite{KS-Kerr1}, stated in Theorem   \ref{Theorem:ExistenceGCMS1-intro},  and  prove its new version, stated in Theorem \ref{Theorem:ExistenceGCMS1-can-intro}, using  the notion of  canonical ${\ell=1}  $  basis introduced in section 3. 
\item In section 7,  we prove our intrinsic  GCM result stated in Theorem \ref{theorem:ExistenceGCMS2-Intro} and use it to define  our notion of angular momentum. We also illustrate   the result      in Corollary \ref{cor:ExistenceGCMS1inKerr}    applied  to  a  far region  of  a Kerr spacetime.
\end{itemize}

  
      \subsection{Acknowledgments}


      The authors are grateful to  A. Chang and  C. De Lellis  for very helpful suggestions and references in connection to  our results on effective    uniformization in sections 3 and 4.

The first author is supported   by  the  NSF grant  DMS 180841 as well as by  the Simons grant  10011738. He would like to thank the Laboratoire Jacques-Louis Lions  of Sorbonne Universit\'e  and   IHES   for their  hospitality during  his many visits.  The second author is supported by ERC grant  ERC-2016 CoG 725589 EPGR.


  \section{Review of uniformization results for the sphere}
  \lab{sec:reviewstandarduniformization}


In this section, we  review  some well-known results concerning the 
uniformization theorem for the sphere that will be used in  section \ref{sec:effectiveuniformization}.

   
\subsection{Uniformization for  metrics on $\SSS^2$}  


We start with  the following well known  calculation.
  \begin{lemma} 
 \lab{Le:transf-Gauss} 
 Let $S$ a surface and let  $g^S$  be a  Riemannian metric on $S$.  For  a scalar function $u$ on $S$,    the Gauss curvature of    $ e^{2u} g^S$ is connected to that of  $g^S$ by the formula
 \bea
 K(  e^{2u} g^S)=e^{-2u} \big(K^S-\lap_{S} u \big)
 \eea
 where $K^S$ denotes the Gauss curvature  of $S$ and  $\lap_S $  the Laplace-Beltrami on $S$. 
\end{lemma}

   According to the classical uniformization theorem, if $S$ is a closed, oriented and connected surface of genus 0 and $g^S$  is a  Riemannian metric on $S$, then,  there exists  a smooth  diffeomorphism  $\Phi:\SSS^2\to S$  and a smooth  conformal factor  
$u$ on $\SSS^2$ such that
\bea
\lab{standard-unifS}
\Phi^\#( g^S) =  e^{2u}  \ga_0
\eea
where $\ga_0$ is the   canonical metric on the standard sphere $\SSS^2$.      In view of  Lemma \ref{Le:transf-Gauss} above, if we denote by $g$   the  metric on $\SSS^2$,  $g=\Phi^\# (g^S) = e^{2u} \ga_0$,  we  derive
  \bea
 \lab{eq:Euler-K}
 \lap_0 u +K(g) e^{2u} =1
 \eea
 where we have used  the notation
 \beaa
 \De_0 &:=& \De_{\ga_0}.
 \eeaa
 We thus have the following corollary of Lemma \ref{Le:transf-Gauss}.
  \begin{corollary}
  \lab{Corr:transf-Gauss}
  If $\Phi^\# g^S= e^{2u} \ga_0 $, then $u $ verifies the equation
  \bea
  \lab{eq:Euler-K-Phi}
  \lap_0 u + (K^S\circ \Phi)  e^{2u}  =1,
  \eea
  where $K^S$ is the Gauss curvature of $g^S$ on $S$. 
  \end{corollary}
  \begin{proof}
  The proof follows from \eqref{eq:Euler-K}  in view  of the  fact that  $K\big(\Phi^\# g^S\big) = K(g^S)\circ \Phi$.
  \end{proof}
  
  \begin{definition}
  Let $\M$ denote the group of conformal transformations of  $\SSS^2$, i.e. the set of diffeomorphisms $\Phi$ of $\SSS^2$ such that $\Phi^\# \ga_0= e^{2u}  \ga_0$ for some scalar function $u$ on $\SSS^2$. 
  \end{definition}
  
  \begin{remark}\lab{rem:remarkonequationEuler-K=1}
  Let $\Phi\in\M$ so that $\Phi^\# \ga_0= e^{2u}  \ga_0$. Then, $u$ satisfies\footnote{This follows immediately from writing $\Phi^\# \ga_0= e^{2u}  \ga_0$ in matrix form, by evaluating on an orthonormal frame, and then taking the absolute value of the determinant on both sides. Also, recall that $|\det d\Phi|$ is an intrinsic scalar on $\SSS^2$, i.e. it does not depend on the particular choice of orthonormal frame.}
  \beaa 
  u= \frac 12 \log |\det d\Phi|.
  \eeaa
Also, in view of Corollary \ref{Corr:transf-Gauss},  we have
  \bea
  \lab{eq:Euler-K=1}
  \lap_0 u+ e^{2u} =1.
  \eea
 \end{remark}

   \begin{lemma} 
 All solutions of \eqref{eq:Euler-K=1}  are of the form $u=\frac 1 2 \log | \det d\Phi|$ with $\Phi\in \M$.
 \end{lemma}
 
 \begin{proof} 
 We have already checked that for  every $\Phi\in \M$,  $u=\frac 12 \log |\det d\Phi| $ is a solution of  \eqref{eq:Euler-K=1}, see Remark \ref{rem:remarkonequationEuler-K=1}. 
 On the other hand, if $u$ is a solution  of that equation, then  $ K(e^{2u}\ga_0) =1$ by Lemma \ref{Le:transf-Gauss} which implies that there exist a diffeomorphism 
 $\Phi:\SSS^2\to \SSS^2 $ such that  $\Phi^\# \ga_0=e^{2u}\ga_0$,  i.e.  $\Phi\in \M $ and $u=\frac 1 2 \log|\det  d \Phi| $.
 \end{proof}
 
 \begin{corollary}
  If  $u$ is a solution of \eqref{eq:Euler-K}  with $K=K(g)$, and if $\Phi\in \M$,  then 
  \beaa
  u_\Phi &:=& u\circ \Phi+ \frac 1 2 \log \big|\det d\Phi \big|
  \eeaa
    is also a solution of \eqref{eq:Euler-K} with $K$ replaced by $K\circ \Phi$.
  \end{corollary}
  
 \begin{proof}
 If  $\Phi\in \M$  then, since $\Phi^\# \ga_0=|\det(d\Phi)|  \ga_0$ and $g=e^{2u}\gamma_0$,
 \beaa
\Phi^\# g= e^{2 u\circ \Phi}   \Phi^\# \ga_0=e^{2u\circ \Phi}  |\det (d\Phi)|  \ga_0= e^{2 v}\ga_0, \qquad v=u\circ \Phi+ \frac 1 2 \log\big| \det d\Phi\big|.
\eeaa
Hence if $u$ is a solution of \eqref{eq:Euler-K},  with $K=K(g) $,  then so   is     $ u_\Phi=u\circ \Phi+ \frac 1 2 \log \big|\det d\Phi \big| $  with $K$ replaced by $K\circ \Phi$, as $K\big(\Phi^\# g\big) = K(g)\circ \Phi$.
\end{proof}

 
 \subsection{Conformal isometries of $\SSS^2$ and  the M\"obius group}
 \lab{sec:conformalisometrymoebiusgroup}
 
 
 We represent the standard sphere $\SSS^2$ as $\big\{ x\in\RRR^3, |x|^2=1\big\}$. Let  $N=(0,0, 1)$  denote the north pole of $\SSS^2$.
 Through the stereographic projection  from the North pole to the   equatorial plane   plane $(x^1, x^2)$  we consider the complex coordinate
 \bea
 z=\frac{x^1+i x^2}{1-x^3}
 \eea
 with the inverse transformation
 \bea\lab{eq:inversestereographicprojectionformula}
 x^1= \frac{2}{1+|z|^2} \Re z,\quad    x^2 =\frac{2}{1+|z|^2} \Im z, \quad x^3 =\frac{|z|^2-1}{|z|^2+1}
 \eea
  and pull-back of the standard    metric $\ga_0$ on $\SSS^2$
  \bea
  \lab{standard sphere-z} 
  4 \big(1+ |z|^2\big)^{-2} |dz|^2. 
  \eea
  
 The   conformal isometry group $\M$ of $\SSS^2$ consists in fact of M\"obius transforms and conjugation of M\"obius transforms, see for example Theorem 18.10.4 and section 18.10.2.4 in \cite{Berger}, and can thus be identified with $SL(2,\CCC)$,
 \bea
 z\to \frac{az+b}{cz+d},\qquad z\to \frac{a\ov{z}+b}{c\ov{z}+d},\qquad ad-bc=1, \quad a,b, c, d\in \CCC.
 \eea
The particular case of M\"obius transforms where $d=a^{-1}>0$ and $b=c=0$ will play an important role. Given $t>0$ and a point $p\in \SSS^2$,  we can choose coordinates  such that $p$ is at  the north pole and obtain scale transformations defined by
 \bea\lab{eq:defintionscaletranformations}
 \Phi_{p, t} z= tz.
 \eea

 
 \subsection{Onofri  functional and inequality}


 Following \cite{Onofri},  as well as \cite{CY1} \cite{CY2},   we introduce  the  functional
 \bea
 \bsplit
 S[u]&=\frac{1}{4\pi} \left( \int_{\SSS^2}  |\nab u|^2+ 2 \int_{\SSS^2} u\right).
 \end{split}
 \eea

\begin{lemma}
\lab{eq:Gageforu}
The functional $S$     is invariant with respect to the transformations $ u\to u_\Phi$, i.e. 
\bea
S[u_\Phi]=S[u].
\eea
\end{lemma}

\begin{proof}
See \cite{Onofri}.
\end{proof}

\begin{definition}
We define the center of mass of  $e^{2u}$  to be
\bea
CM[e^{2u}]&=&\frac{\int_{\SSS^2}  x  e^{2u} }{\int_{\SSS^2} e^{2u}} 
\eea
where $x=(x^1, x^2, x^3)$ on the sphere $\SSS^2$.  
Also, we define  the spaces of functions
\bea
\bsplit
\SS:=&\Big\{ u\in H^1(\SSS^2)\,\,\textrm{ such that }\,\,  CM[e^{2u}]=0\Big\}, \\
\SS_0:=&\left\{ u\in \SS \,\,\textrm{ such that }\,\,  \int_{\SSS^2} e^{2u}=4\pi\right\}. 
\end{split}
\eea
\end{definition}

\begin{remark}\lab{rem:ifuiscenteredthensoisucricO}
If $u$ belongs to $\SS$ (resp. $\SS_0$), then $u\circ O$ also belongs to $\SS$ (resp. $\SS_0$) for any $O\in O(3)$. 
\end{remark}

\begin{lemma}
\lab{le:renormalization}
Given $u$ a smooth function on $\SSS^2$ there exists  a   conformal transformation  $\Phi$ such that  $u_\Phi\in \SS$, i.e. $\int_{\SSS^2}  e^{2 u_\Phi} x =0$.
\end{lemma}

\begin{proof}
See \cite{Onofri}, as well as Proposition 2.2. in \cite{CY1} and Lemma 2 in \cite{Chang}.
\end{proof}

\begin{remark}
In fact, $\Phi$ in Lemma \ref{le:renormalization} is unique up to isometries of $\SSS^2$, see Lemma  \ref{lemma:improved-renormalization:bis}.
\end{remark}

\begin{proposition}[Onofri]
\lab{Thm:Onofri}
Given $f\in H^1(\SSS^2) $ we have,
\bea
\lab{eq:Thm.Onofri}
\log\left(\int_{\SSS^2} e^{2u}\right) &\le S[u]
\eea
with equality iff  $u=\frac 1 2 \log |\det \Phi| $ for some conformal $\Phi\in \M$.
\end{proposition}

\begin{proof}
See \cite{Onofri}.
\end{proof}


\subsection{Improved Onofri inequality and applications}


The following is Proposition B in \cite{CY1} which improves  the Onofri inequality  for  $u\in \SS$.
\begin{proposition}
\lab{PropChang-:propB}
There exists an $a<1$ such that for all  $u\in \SS$
\bea
\frac{1}{4\pi}\int e^{2u}\le \exp{\left(a\frac{1}{4\pi}\int |\nab u|^2+ 2 \frac{1}{4\pi}\int u \right)}.
\eea
\end{proposition}

\begin{corollary}
\lab{cor:Chang1}
Let $a<1$ as in Proposition \ref{PropChang-:propB}. If $u\in \SS_0$ then,
\bea
\frac{1}{4\pi} \int|\nab u|^2 \le (1-a)^{-1}   S[u].
\eea
\end{corollary}

\begin{proof}
Since $u\in \SS_0$, we have in view of Proposition \ref{PropChang-:propB}
\bea
1=\frac{1}{4\pi}\int e^{2u}\le \exp{\left(a\frac{1}{4\pi}\int |\nab u|^2+ 2 \frac{1}{4\pi}\int u \right)}
\eea
from which we  deduce  
\beaa
a\frac{1}{4\pi}\int|\nab u|^2+ 2 \frac{1}{4\pi}\int u &\ge& 0.
\eeaa 
By definition of $S[u]$  we infer
\beaa
(1-a)\frac{1}{4\pi}\int |\nab u|^2 =S[u]-\left( a \frac{1}{4\pi}\int |\nab u|^2 +2 \frac{1}{4\pi}\int u\right)\le S[u]
\eeaa
as desired.
\end{proof}

The following is Proposition   4.1 in \cite{CY2}.
\begin{proposition}
\lab{prop:Chang}
Let $K>0$ be  the Gauss curvature of the metric  $e^{2w} \ga_0$ on $\SSS^2$, so that  $w$ verifies \eqref{eq:Euler-K}, i.e. $\lap w +K e^{2 w} =1$.
Then $S[w] \leq C(K)$ for some constant  $C$ depending only on $\max_{\SSS^2} K$ and $\min_{\SSS^2} K$.
\end{proposition} 

\begin{proof}
In view of  Lemma \ref{le:renormalization}  and Lemma \ref{eq:Gageforu},  we may assume that $w$ is centered, i.e. $w\in \SS$. The proof then follows by exploiting that $w$ satisfies $\lap w +K e^{2 w} =1$, and relying on an  estimate satisfied by functions in $\SS$. See  \cite{CY2} for the details. 
\end{proof}

The following is  Corollary 4.3. in \cite{CY2}.
\begin{corollary} 
\lab{corr:Alice}
Given any  $\de>0$, there exists $\ep>0$ sufficiently small such that if  $\|K-1\|_{\infty}  \leq \ep$ then  $S[w]  \le \de$  for all  $w, K$  which verify  equation   \eqref{eq:Euler-K}.
\end{corollary}

\begin{proof}
The proof,  based on a contradiction argument, relies on  Lemma  \ref{eq:Gageforu}, Proposition \ref{prop:Chang} and the fact that  any  centered solution of $\lap_0 w +e^{2w}=1$ must vanish identically. See  \cite{CY2} for the details.
\end{proof}

We restate the result in the following.
\begin{corollary}
\lab{corr:Alice2}
Let $g$  a metric  on $\SSS^2$   such that  $ g= e^{2u} \ga_0 $ with $u\in \SS_0$. 
 Then,  given any  $\de>0$,  there exists  $\ep>0$  sufficiently small  such that if 
 $ \| K(g)-1\|_{L^\infty} \leq \ep$,  then 
 \bea
 \| u\|_{H^1} &\leq& \de.
 \eea
\end{corollary}

\begin{proof} 
 According to  Corollary \ref{corr:Alice}, given $\de >0$   sufficiently small, 
we can find  $\ep>0$  such that  if   $ \| K(g)-1\|_{L^\infty} \leq \ep$,  then $S[u]\le \de^2$. In view of Corollary \ref{cor:Chang1},  we deduce
\beaa
\frac{1}{4\pi}\int |\nab u|^2 &\leq& (1-a)^{-1}\de^2.
\eeaa
Together with the bound $S[u]\leq \de^2$, we infer
\beaa
\|\nabla u\|_{L^2}+\left|\int u\right| &\les& \de.
\eeaa
Together with the Poincar\'e inequality for scalars on $\SSS^2$, we infer 
\beaa
\|u\|_{H^1} &\les &  \de
\eeaa
as desired.
\end{proof}


  \section{Effective uniformization for nearly round spheres}
  \lab{sec:effectiveuniformization}


Let $(S, g^S)$ be a fixed sphere of  area radius $1$, i.e. $|S| =4\pi$.  The goal of this section is to obtain the following improvement of  the results reviewed in section \ref{sec:reviewstandarduniformization}. 

\begin{theorem}[Effective uniformization]
\lab{theorem-Eff.uniformalization-appendix}
Let $(S, g^S)$ be a fixed sphere with $|S| =4\pi$. There exists, up to isometries\footnote{i.e.  all  the solutions are of the form $(\Phi\circ O, u\circ O)$ for $O\in O(3)$. In particular, recall that if $u$ is centered, then so is $u\circ O$ for $O\in O(3)$, see Remark  \ref{rem:ifuiscenteredthensoisucricO}.} of $\SSS^2$,  a  unique   diffeomorphism   $\Phi:\SSS^2\to S$  and a  unique centered conformal factor $u$, i.e. $u\in\SS$, such that  $ \Phi^\#( g^S) =  e^{2u}  \ga_0$. Moreover,  under the almost round condition
\bea
\lab{eq:Almostround}
\|K^S-1\|_{L^\infty} \le \ep
\eea
where  $K^S=K(g^S)$,  the following properties are verified for sufficiently small $\ep>0$.
\begin{enumerate}
\item We have
\bea
\lab{eq:Thm.-Eff.uniformalization1-appendix}
 \| u\circ\Phi^{-1} \|_{L^\infty  (S) } \les \ep.
 \eea
 \item If in addition $\|K^S -1\|_{H^s(S)} \leq \ep$ for some $s\geq 0$, then 
 \bea
 \lab{eq:Thm.-Eff.uniformalization1-appendix:bis}
 \| u\circ \Phi^{-1} \|_{H^{2+s}(S) } \les \ep.
 \eea
\end{enumerate}
\end{theorem}

\begin{remark}
 To the best of our knowledge the estimates   \eqref{eq:Thm.-Eff.uniformalization1-appendix},  \eqref{eq:Thm.-Eff.uniformalization1-appendix:bis}  have not been stated as such in the literature.    The  uniqueness statement  also appears to be new.
\end{remark}

\begin{remark}
One can easily adapt the statement of Theorem \ref{theorem-Eff.uniformalization-appendix}  to the case $|S|\neq 4\pi$, see Corollary \ref{proposition:effective-uniformisation}.  
\end{remark}

Theorem \ref{theorem-Eff.uniformalization-appendix} will be proved in section \ref{sec:proofof-theorem-Eff.uniformalization-appendix}. We first provide improvements of Lemma \ref{le:renormalization} and Corollary \ref{corr:Alice2}.


\subsection{Uniqueness for Lemma \ref{le:renormalization}}


In this section, we provide a uniqueness statement for Lemma \ref{le:renormalization}, see  Lemma  \ref{lemma:improved-renormalization:bis}.  We also strengthen the conclusions of Lemma \ref{le:renormalization} for  nearly round spheres, see Lemma  \ref{lemma:improved-renormalization}. To this end, we start with the following decomposition of conformal isometries of $\SSS^2$.
\begin{lemma}\lab{lemma:decompositionofhomographies}
Any $\Phi\in \M$ admits the following decomposition
\beaa
\Phi=O_1\circ\Phi_{N, t}\circ O_2
\eeaa
where $O_1, O_2\in O(3)$, $N=(0,0,1)$, $t>0$, and the scale transformation $\Phi_{N, t}$ has been introduced in \eqref{eq:defintionscaletranformations}. 
\end{lemma}

\begin{proof}
Upon multiplying $\Phi$ by a reflexion, which belongs to $O(3)$, we may assume that $\Phi$ corresponds, in view of section \ref{sec:conformalisometrymoebiusgroup}, to a M\"obius transformation, and hence to a matrix in $SL(2,\CCC)$. Then, using that, see for example \cite{JoSi},
\begin{enumerate}
\item composition of M\"obius transformations corresponds to matrix multiplication in $SL(2,\CCC)$, 

\item $SO(3)$ corresponds to M\"obius transformations with matrices in $SU(2)$, 

\item scale transformations $\Phi_{N, t}$ with $t>0$ correspond to the following matrix  in $SL(2,\CCC)$
\beaa
\left(\ba{cc}
\sqrt{t} & 0\\
0 & \frac{1}{\sqrt{t}}
\ea\right),
\eeaa
\end{enumerate}
 the lemma reduces to proving the following decomposition for a matrix $A\in SL(2,\CCC)$
\bea\lab{eq:reductionofthedecompositionMtoSL2C}
A=U_1DU_2, \qquad U_1, U_2\in SU(2), \qquad D=\left(\ba{cc}
\sqrt{t} & 0\\
0 & \frac{1}{\sqrt{t}}
\ea\right).
\eea
To prove the claim \eqref{eq:reductionofthedecompositionMtoSL2C}, we apply the polar decomposition to $A$
\beaa
A=RU_3, \qquad R=R^*, \qquad R>0, \qquad \det(R)=1, \qquad U_3\in SU(2),
\eeaa
where $\det(R)=1$ and $U_3\in SU(2)$ since $\det(A)=1$. Also, $R$ being hermitian semidefinite positive, we may  diagonalize it as follows 
\beaa
R=U_1\left(\ba{cc}
\la & 0\\
0 & \la^{-1}
\ea\right)U_1^*, \qquad U_1\in SU(2), \qquad \la>0,
\eeaa
where we have used the fact that  $\det(R)=1$. Finally, we have
\beaa
A=U_1\left(\ba{cc}
\sqrt{t} & 0\\
0 & \frac{1}{\sqrt{t}}
\ea\right)U_2, \textrm{ where }U_2=U_1^*U_3, \qquad t=\la^2>0, 
\eeaa
which is \eqref{eq:reductionofthedecompositionMtoSL2C} as desired.
\end{proof}

\begin{lemma}
\lab{lemma:improved-renormalization:bis}
Let $u$ be a smooth function on $\SSS^2$  and let  $\Phi$ be a conformal transformation. Assume that both $u$ and $u_\Phi$ belong to $\SS$. Then,  we have\footnote{Recall from Remark \ref{rem:ifuiscenteredthensoisucricO} that if $u$ belongs to $\SS$, then $u\circ O$ also belongs to $\SS$ for any $O\in O(3)$.} $\Phi\in O(3)$.
\end{lemma}

\begin{proof}
Decomposing $\Phi$ using Lemma \ref{lemma:decompositionofhomographies}, it suffices in fact to prove that if  $u$ and $u_{\Phi_{N,t}}$ are in $\SS$, where $t>0$, then $\Phi_{N,t}=I$, i.e. $t=1$. Thus, from now on, we assume that $u$ and $u_{\Phi_{N,t}}$ are in $\SS$ with $t>0$.

Since $u$ and $u_{\Phi_{N,t}}$ are in $\SS$, we have
\bea\lab{eq:consequenceofuanduPhitzarebothinSS}
\nn 0 &=& \int_{\SSS^2}e^{2u_{\Phi_{N,t}}}x -\int_{\SSS^2}e^{2u}x =\int_{\SSS^2}e^{2u\circ\Phi_{N,t}}|\det(d\Phi_{N,t})|x -\int_{\SSS^2}e^{2u}x\\
& =& \int_{\SSS^2}e^{2u}x\circ\Phi_{N,t}^{-1} -\int_{\SSS^2}e^{2u}x  = \int_{\SSS^2}e^{2u}\Big(x\circ\Phi_{N,\frac{1}{t}}  -x\Big)
\eea
where we used the definition of $u_{\Phi_{N,t}}$, the change of variable formula, and the fact that $\Phi_{N,t}^{-1}=\Phi_{N,t^{-1}}$. Now, in view of the formula for $\Phi_{N,t}$, we have, using the formula for stereographic coordinates, see \eqref{eq:inversestereographicprojectionformula},  
\beaa
 x^1\circ\Phi_{N,t^{-1}}=\frac{2t^{-1}}{1+t^{-2}|z|^2} \Re z,\quad    x^2\circ\Phi_{N,t^{-1}} =\frac{2t^{-1}}{1+t^{-2}|z|^2} \Im z, \quad x^3\circ\Phi_{N,t^{-1}} =\frac{t^{-2}|z|^2-1}{t^{-2}|z|^2+1}.
 \eeaa
Since
\beaa
|z|^2=\frac{1+x^3}{1-x^3},
\eeaa
we infer, using   \eqref{eq:inversestereographicprojectionformula} again,
\bea\lab{eq:xcomposedwithscaletransformation}
\bsplit
 x^1\circ\Phi_{N,-t} &=  \frac{2t^{-1}}{t^{-2}(1+x^3)+(1-x^3)}x^1,\\
 x^2\circ\Phi_{N,-t} &=  \frac{2t^{-1}}{t^{-2}(1+x^3)+(1-x^3)}x^2, \\ 
 x^3\circ\Phi_{N,-t} &= \frac{t^{-2}(1+x^3)-(1-x^3)}{t^{-2}(1+x^3)+(1-x^3)}.
 \end{split}
 \eea
 In particular, the last identity of \eqref{eq:xcomposedwithscaletransformation} yields
\beaa
x^3\circ\Phi_{N,-t} -x^3 &=& (t^{-2}-1)\frac{1-(x^3)^2}{t^{-2}(1+x^3)+(1-x^3)}
\eeaa
and thus, 
\beaa
 \int_{\SSS^2}e^{2u}\Big(x^3\circ\Phi_{N,\frac{1}{t}}  -x^3\Big) &=& (t^{-2}-1)\int_{\SSS^2}e^{2u}\frac{1-(x^3)^2}{t^{-2}(1+x^3)+(1-x^3)}
\eeaa
so that we have, in view of \eqref{eq:consequenceofuanduPhitzarebothinSS},  
\beaa
(t^{-2}-1)\int_{\SSS^2}e^{2u}\frac{1-(x^3)^2}{t^{-2}(1+x^3)+(1-x^3)} &=& 0.
\eeaa
Since the integral is strictly positive and $t>0$, we infer $t=1$ as desired. 
\end{proof}

 \begin{lemma}
\lab{lemma:improved-renormalization}
Let $\ep$ and $\de$ two constants such that $0<\de\leq\ep$. Given $u$ a smooth function on $\SSS^2$ such that 
\beaa
\left|\int_{\SSS^2}e^{2u}x\right|\leq\de, \qquad \|u\|_{L^\infty}\leq \ep.
\eeaa
Then, for $\ep>0$ sufficiently small,  there exists  a   conformal transformation  $\Phi$ such that  $u_\Phi\in \SS$ and satisfying in addition  $\big| \Phi-I  \big|\les \de$. 
\end{lemma}

\begin{proof}
Let $p_1=(1,0,0)$, $p_2=(0,1,0)$ and $p_3=(0,0,1)$, and let us consider the map
\beaa
\Theta: \RRR^3\to\RRR^3, \qquad \Theta(t)=\int_{\SSS^2}e^{2u_{\Phi_{p_1, t_1}\circ\Phi_{p_2, t_2}\circ\Phi_{p_3, t_3}}}x.
\eeaa
To prove the lemma, it suffices to exhibit $t\in\RRR^3$ such that $\Theta(t)=0$ and $|t-(1,1,1)|\les\de$. 

Using the definition of $u_{\Phi_{p_i,t_i}}$, the change of variable formula, and the fact that $\Phi_{p_i,t_i}^{-1}=\Phi_{p_i,t_i^{-1}}$, we have
\beaa
\Theta(t) &=& \int_{\SSS^2}e^{2u}x\circ\Phi_{p_3, t_3^{-1}}\circ\Phi_{p_2, t_2^{-1}}\circ\Phi_{p_1, t_1^{-1}}.
\eeaa
Since $|\int_{\SSS^2}e^{2u}x|\leq\de$, and since $\Phi_{p_i,1}$ is the identity, we have
\bea\lab{eq:estimateofThetaforapplyingtheBanachfixedpointtheoremforzerosofTheta}
\Theta(1,1,1) &=& \int_{\SSS^2}e^{2u}x=O(\de). 
\eea

Next, we compute the differential of $\Theta$ at $t=(1,1,1)$. To this end, note that we have in view of \eqref{eq:xcomposedwithscaletransformation}
\beaa
\Theta(1,1,t_3) &=& \int_{\SSS^2}e^{2u}x\circ \Phi_{p_3, \frac{1}{t_3}}=\Theta(1,1,1)+\int_{\SSS^2}e^{2u}\Big(x\circ \Phi_{N, \frac{1}{t_3}}-x\Big)\\
&=& \Theta(1,1,1)+\int_{\SSS^2}e^{2u}
\left(\ba{c}
\frac{2t^{-1}}{t^{-2}(1+x^3)+(1-x^3)}x^1\\
\frac{2t^{-1}}{t^{-2}(1+x^3)+(1-x^3)}x^2\\
\frac{t^{-2}(1+x^3)-(1-x^3)}{t^{-2}(1+x^3)+(1-x^3)}
\ea\right).
\eeaa
We infer, using also $\|u\|_{L^\infty}\leq \ep$, 
\beaa
\pr_{t_3}\Theta(1,1,1) &=& \int_{\SSS^2}e^{2u}\left(\ba{c}
2x^1x^3\\
2x^2x^3\\
-(1-(x^3)^2)
\ea\right)=O(\ep)+\int_{\SSS^2}\left(\ba{c}
2x^1x^3\\
2x^2x^3\\
-(1-(x^3)^2)
\ea\right)
\eeaa
and hence
\beaa
\pr_{t_3}\Theta(1,1,1)=-\frac{8\pi}{3}\left(\ba{c}
0\\
0\\
1
\ea\right)+O(\ep).
\eeaa
One easily derives corresponding identities for $\pr_{t_1}\Theta(1,1,1)$ and $\pr_{t_2}\Theta(1,1,1)$ which yields
\bea\lab{eq:estimateofThetaforapplyingtheBanachfixedpointtheoremforzerosofTheta:bis}
d\Theta_{|_{(1,1,1)}}=-\frac{8\pi}{3}I+O(\ep).
\eea
Thus, finding $t$ such that $\Theta(t)=0$ amounts to solving the following fixed point 
\beaa
t-(1,1,1) &=& -(d\Theta_{|_{(1,1,1)}})^{-1}\left[\Theta(1,1,1) +\left(\Theta(t) -\Theta(1,1,1)-d\Theta_{|_{(1,1,1)}}\big(t-(1,1,1)\big)\right)\right]
\eeaa
whose existence, together with the desired estimate $|t-(1,1,1)|\les\de$, follows, in view of \eqref{eq:estimateofThetaforapplyingtheBanachfixedpointtheoremforzerosofTheta} and \eqref{eq:estimateofThetaforapplyingtheBanachfixedpointtheoremforzerosofTheta:bis}, from the Banach fixed point theorem.
\end{proof}


\subsection{Improvement of Corollary \ref{corr:Alice2} for nearly round spheres}


In the proposition below, we improve the results  of Corollary \ref{corr:Alice2} for nearly round spheres.
\begin{proposition}
\lab{proposition:effective-uniformisation-Appendix}
Assume that  the metric $g= e^{2u} \ga_0$  on $\SSS^2$, with $u\in \SS_0$, is     such that   $ \| K(g)-1\|_{L^\infty} \leq \ep$. 
Then, for $\ep>0$ sufficiently small,  we have 
\bea\lab{eq:boundin:proposition:effective-uniformisation-Appendix}
\| u\|_{L^\infty(\SSS^2)} \les \ep.
\eea 
\end{proposition}

\begin{proof}
Let $\de>0$ a sufficiently small universal constant which will be chosen later. Then,  according to  Corollary \ref{corr:Alice2},  if $\ep$, such that $\|K-1\|_{L^\infty}\le \ep$ holds, is sufficiently small, i.e. $0<\ep\leq \ep(\de)\ll 1$, we have
\beaa
\| u\|_{H^1} &\le&  \de.
\eeaa
Our goal is to improve this estimate by showing that the stronger bound \eqref{eq:boundin:proposition:effective-uniformisation-Appendix} holds.

Since $g= e^{2u} \ga_0$,   the scalar function $u$ verifies  equation \eqref{eq:Euler-K}, i.e.
\beaa
\lap_0 u+ K e^{2u} =1,
\eeaa
where $K=K(g)$, which we rewrite in the form,
\beaa
\lap_0 u+ (e^{2u}-1)= -(K-1 ) e^{2u} 
\eeaa
or, 
\beaa
\lap_0 u+ 2 u &=& - (K-1 ) e^{2u}  - f(u), \quad\textrm{ where }\quad f(u):= e^{2u}-1-2u=\sum_{n\ge 2 } \frac{2^n}{n!}  u^n.
\eeaa
Also, since $u\in\SS_0$, we have $\int e^{2u} x =0$. Together with  $\int x=0$,  we deduce,
\beaa
0 &=& \int e^{2u}x=\int  \big( 1+2u+ f(u) \big)x=2 \int  u x +\int f(u) x 
\eeaa
and hence
\beaa
\left|\int  u\, x \right| \les   \int    |f(u)|\les \|f(u)\|_{L^2}.
\eeaa
Note that $f(u)\le   4  u^2 e^{ 2u} $.
By standard elliptic estimates we have\footnote{Note that the kernel of $\lap_0 + 2$ is given by the $\ell=1$ spherical harmonics, and that $\int  u x$ corresponds to the projection on these spherical harmonics.}
\beaa
\| u\|_{H^2}&\les&  \| (\lap_0 + 2)u \|_{L^2}+ \left|\int  u\, x \right| 
\eeaa
and hence
\beaa
\| u\|_{H^2}&\les& \|(K-1) e^{2u} \|_{L^2}  +\| f(u) \|_{L^2} + \left|\int  u\, x \right| \\
&\les&\ep   \| e^{2u} \|_{L^2}+  \| f(u) \|_{L^2}\\
&\les&  \ep   \| e^{2u} \|_{L^2} + 4 \| u\|_{L^\infty}^2   \| e^{2u} \|_{L^2}.
\eeaa
In view of Onofri inequality \eqref{eq:Thm.Onofri},  we have 
\beaa
\log \int e^{4 u} &\le&  S[2u]  \les  \| u\|_{H^1}^2 +\| u\|_{H^1}   \les  \de   \le 1 
\eeaa
where we have used $\| u\|_{H^1} \le\de$ and chosen $\de>0$ sufficiently small. We deduce,
\beaa
\| u\|_{H^2}&\les&   \ep +  \|u\|_{L^\infty}^2.
\eeaa
By a calculus inequality
\beaa
\| u\|_{L^\infty}\les  \|u\|_{L^2}^{\frac{1}{2}} \| u\|_{H^2}^{\frac{1}{2}}. 
\eeaa
Hence,
\beaa
\| u\|_{H^2}&\les&     \ep+    \|u\|_{L^2} \| u\|_{H^2}.
\eeaa
Thus, since $\|u\|_{H^1}\le \de$, we infer that, in fact, for $\de>0$ small
\beaa
\| u\|_{H^2}&\les& \ep.
\eeaa
 In other words, if $\|K-1\|_{L^\infty} \le \ep $, for $\ep>0$ sufficiently small, then  $\|u\|_{H^2}\les \ep$.  In particular, by Sobolev inequality,  we have $\| u\|_{L^\infty}\les \ep$ as stated.
\end{proof}


\subsection{Proof of  Theorem \ref{theorem-Eff.uniformalization-appendix}}
\lab{sec:proofof-theorem-Eff.uniformalization-appendix}



\subsubsection{Existence  part}


According to the standard uniformization theorem,  there exists a scalar function $\tilde{u}$ and a map $\widetilde{\Phi}:\SSS^2\to S$ such that      we have $\widetilde{\Phi}^\# (g^S)= e^{2\tilde{u}} \ga_0$ with $ \int_{\SSS^2}  e^{2\tilde{u}}=4\pi$.  In view of  Lemma \ref{le:renormalization}, we can find     $\Psi\in \M$  such that    $u=\tilde{u}_\Psi\in \SS_0$, and  then,  $g=\Phi^\# (g^S)=e^{2u} \ga_0$ with $\Phi=\widetilde{\Phi}\circ\Psi$.


\subsubsection{Uniqueness  part}


 We address below the issue of  uniqueness,  up to isometries of $\SSS^2$, of $(\Phi, u)$. Assume that we have
 \bea
( \Phi_i)^\# g^S= e^{2u_i} \ga_0, \qquad i=1,2,
 \eea
  with $\Phi_i:\SSS^2\to S$ diffeomorphisms and   $u_i$ centered conformal factors.  Let $\Psi=\Phi_2^{-1} \circ \Phi_1$ so that $\Psi:\SSS^2\to \SSS^2$. In view of the above, we have 
  \beaa
  \Psi^\# \ga_0&=& \Phi_1^\#(\Phi_2^{-1})^\#\ga_0=\Phi_1^\#(e^{-2u_2\circ\Phi_2^{-1}}g^S) =e^{-2u_2\circ\Phi_2^{-1}\circ\Phi_1}e^{2u_1}\ga_0
  \eeaa
  and hence
  \beaa
  \Psi^\# \ga_0&=& e^{2(u_1- u_2\circ \Psi)} \ga_0.
\eeaa
Therefore  $\Psi$ is a conformal isometry of $\SSS^2$, and we  deduce in view of Remark \ref{rem:remarkonequationEuler-K=1}
\beaa
u_1- u_2\circ \Psi=\frac 1 2 \log |\det\Psi|,
\eeaa
i.e.,
\beaa
u_1=( u_2)_\Psi.
\eeaa
Since both $u_1$ and $u_2$ are centered  we deduce, in view of Lemma \ref{lemma:improved-renormalization:bis},   $\Psi=O$ with $O\in O(3)$ and therefore $\Phi_1=\Phi_2\circ O$ and $u_1=u_2\circ O$ as stated.


\subsubsection{Estimates}


From now on, we assume in addition the almost round condition \eqref{eq:Almostround}. In view of  \eqref{eq:Almostround}, and the fact that $g=e^{2u} \ga_0$ with $u\in \SS_0$, we deduce by Proposition \ref{proposition:effective-uniformisation-Appendix} that $\| u\|_{L^\infty(\SSS^2) } \les \ep $. This implies $\|u\circ\Phi^{-1} \|_{L^\infty(S)} \les \ep $ as stated in \eqref{eq:Thm.-Eff.uniformalization1-appendix}. To prove  the higher derivative estimates we  observe  that
\bea
(\Phi^{-1})^\# \ga_0 &=& e^{-2 v}g^S, \qquad v:= u\circ\Phi^{-1}, 
\eea
which together with Lemma \ref{Le:transf-Gauss} implies
 \beaa
 1=K((\Phi^{-1})^\# \ga_0)=K(  e^{-2v} g^S)=e^{2v} \big(K^S+\lap_{S} v \big)
 \eeaa
so that $v$ verifies
\beaa
\lap_S v  = \big(e^{-2v}-1\big) -\big(K^S-1\big), \qquad \|v\|_{L^\infty(S)} \les \ep.
\eeaa
In view of the  almost round condition \eqref{eq:Almostround}, we easily deduce  that  in fact $\| v\|_{H^2(S)} \les \ep $. 
Moreover, if  we have in addition $\|K^S-1\|_{H^s} \le \ep$, then, by  standard elliptic regularity, we obtain  $\| v\|_{H^{2+s} (S)} \les \ep$ as stated.


\subsection{Effective uniformization for nearly round spheres of arbitrary area}


Let $(S, g^S)$ be a fixed sphere, and let $r^S$ denote its area radius, i.e. $r^S$ satisfies 
\beaa
|S|=4\pi (r^S)^2.
\eeaa
Given a positive integer $s$, we introduce  the following norm on $S$
   \bea
   \lab{definition:spaceH^k(Sbutnotbold)}
\| f\|_{\hk_s(S)}:&=&\sum_{i=0}^s \|( r^S \nab^S )^i f\|_{L^2(S)}.
\eea
The goal of the following corollary is to extend Theorem \ref{theorem-Eff.uniformalization-appendix} to the case  $r^S\neq 1$.

   \begin{corollary}
   \lab{proposition:effective-uniformisation} 
 Let $(S, g^S)$ be a fixed sphere.    There exists, up to isometries\footnote{i.e.  all  the solutions are of the form $(\Phi\circ O, u\circ O)$ for $O\in O(3)$.} of $\SSS^2$,  a  unique   diffeomorphism   $\Phi:\SSS^2\to S$  and a  unique centered conformal factor $u$, i.e. $u\in\SS$, such that  
   \beaa
    \Phi^\#( g^S) =  (r^S)^2e^{2u}  \ga_0.
    \eeaa  
    Moreover,  under the almost round condition
    \bea\lab{eq:Almostround-S} 
   \left\| K^S-\frac{1}{ (r^S) ^2} \right\|_{L^\infty(S)} &\leq& \frac{\ep}{(r^S)^2},
   \eea
     the following properties are verified for sufficiently small $\ep>0$. 
   
\begin{enumerate}
\item  We have
\bea
\lab{eq:Thm.-Eff.uniformalization1}
 \| u\circ\Phi^{-1} \|_{L^\infty  (S) } &\les& \ep.
 \eea
 
 \item If in addition 
 \bea\lab{eq:Almostround-S-k}
\left \|K^S -\frac{1}{(r^S)^2}\right\|_{\hk_s(S)} &\leq& \frac{\ep}{r^S}
 \eea
  for some $s\geq 0$, then
  \bea
\left \|u\circ\Phi^{-1}\right\|_{\hk_{s+2}(S)} &\les&  \ep r^\S.
 \eea

 \end{enumerate}
\end{corollary}

\begin{proof}
Consider the metric $\tilde{g}^S$ on $S$ given by
\beaa
\tilde{g}^S &:=& \frac{1}{(r^S)^2}g^S.
\eeaa
Then, $(S, \tilde{g}^S)$ has area radius 1, and $K(\tilde{g}^S)=(r^S)^{2}K(g^S)$ so that, in view of \eqref{eq:Almostround-S}, 
the almost round condition \eqref{eq:Almostround} holds for $\tilde{g}^S$. Thus, we may apply Theorem \ref{theorem-Eff.uniformalization-appendix} to $(S, \tilde{g}^S)$ which then  implies the first two conclusions   of Corollary \ref{proposition:effective-uniformisation} for $(S, g^S)$.
\end{proof}


\subsection{Canonical basis of $\ell=1$ modes  on $S$}


Let $S$ be an almost round sphere, i.e. verifying \eqref{eq:Almostround-S}. The goal of this section is to define  on $S$ a canonical generalization of the $\ell=1$ spherical harmonics. 

Recall that  the $\ell=1$ spherical harmonics $J^{\SSS^2}=(J^{(-, \SSS^2)}, J^{(0, \SSS^2)}, J^{(+, \SSS^2)} ) $ are given by the restriction of  $x^1, x^2,  x^3$  to $\SSS^2$.  More precisely, in polar coordinates,
\bea\lab{eq:defofstandardsphericalharmonics}
J^{(0, \SSS^2)} =x^3= \cos\th, \qquad J^{(+, \SSS^2)} =  x^1=\sin\th \cos \vphi, \qquad  J^{(-, \SSS^2)} = x^2=\sin\th \sin \vphi.
\eea

\begin{lemma}
We have, for $p, q\in\{-, 0, +\}$,
\bea
\bsplit
\lap_0 J^{(p, \SSS^2)} &= - 2 J^{(p, \SSS^2)},\\
\int_{\SSS^2 }  J^{(p, \SSS^2)} J^{(q, \SSS^2)} da_{\ga_0} &= \frac{4\pi}{3} \de_{pq}, \\
\int_{\SSS^2}  J^{(p, \SSS^2)} da_{\ga_0} &=0. 
\end{split}
\eea
\end{lemma}

\begin{proof}
Straightforward verification.
\end{proof}

 \begin{definition}[Basis of canonical $\ell=1$ modes  on $S$]
 \lab{definition:ell=1mpdesonS}
 Let $(S,g^S)$ be an almost round sphere, i.e. verifying \eqref{eq:Almostround-S}. Let   $(\Phi, u)$  the unique, up to isometries of $\SSS^2$,  uniformization pair  given by Corollary \ref{proposition:effective-uniformisation},
i.e.,       
\beaa
\Phi:\SSS^2\longrightarrow S, \qquad \Phi^\# (g^S) = (r^S)^2e^{2u}\ga_0,\qquad u\in\SS.
\eeaa
We define the basis of canonical $\ell=1$ modes on $S$ by
\bea
J^S &:=& J^{\SSS^2}\circ \Phi^{-1}, 
\eea
where $J^{\SSS^2}$ denotes the $\ell=1$ spherical harmonics, see \eqref{eq:defofstandardsphericalharmonics}.  
 \end{definition}
 \begin{remark}
 Note that  the canonical basis  is unique up to a rotation on $\SSS^2$.
 \end{remark}
 \begin{lemma}
 \lab{lemma:basicpropertiesofJforcanonicalell=1basis}
  Consider  $(S,g^S)$   a sphere of  area radius $r^S$  verifying  the   almost round condition \eqref{eq:Almostround-S}. Let   $(\Phi, u)$  the unique, up to isometries of $\SSS^2$,  uniformization pair  given by Corollary \ref{proposition:effective-uniformisation}. 
Let $J^S$ denote the basis of canonical $\ell=1$ modes on $S$ of Definition \ref{definition:ell=1mpdesonS}. Then,  we have 
 \bea
 \bsplit
 \lap_S J^{(p,S)}  &=-\frac{2}{(r^S)^2} J^{(p,S)}  +\frac{2}{(r^S)^2}\big(1- e^{-2v}\big)J^{(p,S)}, \\
 \int_S J^{(p,S)} J^{(q,S)} da_g &= \frac{4\pi}{3}(r^S)^2 \de_{pq}+ \int_S J^{(p,S)} J^{(q,S)} \big(1- e^{-2v} \big)   da_{g^S},\\
 \int_S J^{(p,S)} da_g &=0,
 \end{split}
 \eea
 with $\lap^S $ the Laplace-Beltrami of the metric $g^S$ and with $v:= u\circ\Phi^{-1}$. 
 Moreover  we have,
  \bea
  \lab{eq:PropertiesofJ^S}
 \bsplit
 \lap_S J^{(p,S)}  &=\left(-\frac{2}{(r^S)^2} +O\left(\frac{\ep}{(r^S)^2}\right) \right)J^{(p,S)}, \\
 \int_S J^{(p,S)}J^{(q,S)} da_g &= \frac{4\pi}{3}(r^S)^2\de_{pq}+ O(\ep(r^S)^2),
 \end{split}
 \eea
 where $\ep>0$ is the smallness constant appearing in the almost round condition \eqref{eq:Almostround-S}. 
 \end{lemma}
 
 \begin{proof}
 Since $(\Phi, u)$ is the unique, up to isometries of $\SSS^2$,  uniformization pair  given by Corollary \ref{proposition:effective-uniformisation}, we have
\beaa
\Phi:\SSS^2\longrightarrow S, \qquad \Phi^\# (g^S) = (r^S)^2e^{2u}\ga_0,\qquad u\in\SS.
\eeaa
We rewrite this as follows 
\bea
g^S=(r^S)^2e^{2v} g_0, \qquad g_0:=\big(\Phi^{-1}\big)^\# \ga_0, \qquad v:= u\circ\Phi^{-1}. 
\eea
Since $\lap_0J^{\SSS^2}=-2J^{\SSS^2}$ and $g_0= \big(\Phi^{-1}\big)^\# \ga_0$, we have, in view of the definition of $J^S$, i.e. $J^S=J^{\SSS^2}\circ \Phi^{-1}=\big(\Phi^{-1}\big)^\# J^{\SSS^2}$, 
 \beaa
 \lap_{g_0}   J^S =-2 J^S.
 \eeaa
On the other hand, since $g^S=(r^2)^2e^{2v} g_0$, we have   in view of the conformal invariance of the Laplacian, $\lap_S = (r^S)^{-2}e^{-2v} \lap_{g_0} $. We deduce
 \beaa
 \lap_S J^S= (r^S)^{-2}e^{-2v} \lap_{g_0}  J^S=-2(r^S)^{-2}  e^{-2v}  J^S=-\frac{2}{(r^S)^2} J^S+\frac{2}{(r^S)^2}\Big(- e^{-2v}+1  \Big)J^S.
 \eeaa
 Also,
\beaa
\int_S J^{(p,S)} J^{(q,S)} da_{g^S}&=&\int_S J^{(p,S)} J^{(q,S)} (r^S)^2 e^{2v}  da_{g_0}\\
&=& (r^S)^2\int_S\big(\Phi^{-1}\big)^\#J^{(p,\SSS^2)}\big(\Phi^{-1}\big)^\#J^{(q,\SSS^2)}da_{\big(\Phi^{-1}\big)^\#\ga_0}\\
&&+(r^S)^2 \int_S J^{(p,S)}  J^{(q,S)}  \big( e^{2v}-1\big)   da_{g_0}\\
&=& (r^S)^2\int_{\SSS^2}J^{(p,\SSS^2)}J^{(q,\SSS^2)}da_{\ga_0}+ \int_S J^{(p,S)} J^{(q,S)}  \big( e^{2v}-1\big)  e^{-2v}  da_{g^S} \\
&=& \frac{4\pi}{3}(r^S)^2 \de_{pq}+ \int_SJ^{(p,S)}  J^{(q,S)}    \big(1- e^{-2v} \big)  da_{g^S}.
\eeaa
Also, in view of the  centeredness of $u$,
\beaa
\int_S J^S  da_{g^S}&=&\int_S J^{\SSS^2}\circ \Phi^{-1} (r^S)^2e^{2v}   da_{g_0}=(r^S)^2\int_{S}\big(\Phi^{-1}\big)^\#(e^{2u}J^{\SSS^2}) da_{\big(\Phi^{-1}\big)^\#\ga_0}\\
&=& (r^S)^2\int_{\SSS^2} e^{2u}x  da_{\ga_0} =0
\eeaa
as stated.
The proof of \eqref{eq:PropertiesofJ^S} follows immediately in view of the fact that $v=u\circ\Phi$  satisfies $\| v\|_{L^\infty}\les \ep$, see \eqref{eq:Thm.-Eff.uniformalization1}.
 \end{proof}

\begin{corollary}
  Let $(S,g^S)$ verifying \eqref{eq:Almostround-S}. Let   $(\Phi, u)$  the unique, up to isometries of $\SSS^2$,  uniformization pair  given by Corollary \ref{proposition:effective-uniformisation}. 
Let $J^S$ denote the basis of canonical $\ell=1$ modes on $S$ of Definition \ref{definition:ell=1mpdesonS}. Then,  for sufficiently small $\ep>0$, the following holds\footnote{Note that, a priori, one would expect the right-hand side of \eqref{eq:specialcancellationell=1modeofGausscurvature} to be $O(\ep)$. The fact that it is actually $O(\ep^2)$ is an application of Corollary \ref{proposition:effective-uniformisation} and \eqref{eq:PropertiesofJ^S}, see the proof below.} 
\bea\lab{eq:specialcancellationell=1modeofGausscurvature}
\int_S\left(K^S-\frac{1}{(r^S)^2}\right)J^{(p,S)} &=& O(\ep^2),\qquad p=0,+,-,
\eea
where $K^S$ and $r^S$ denote respectively the Gauss curvature and the area radius of $S$.
\end{corollary}

\begin{proof}
Since $\Phi^*(g^S)=(r^S)^2e^{2u}\ga_0$, we have $K((r^S)^{-2}e^{-2u}\Phi^*(g^S))=1$ and hence 
\beaa
1 &=& K((r^S)^{-2}e^{-2u}\Phi^*(g^S))=K((r^S)^{-2}\Phi^*(e^{-2v}g^S))\\
&=&K((r^S)^{-2}e^{-2v}g^S)\circ\Phi, \qquad\quad v:=u\circ\Phi^{-1},
\eeaa
which together with Lemma \ref{Le:transf-Gauss}  yields
 \beaa
1= K(  (r^S)^{-2}e^{-2v} g^S)=(r^S)^2 e^{2v} \big(K^S+\lap_{S} v \big).
 \eeaa 
We infer
\bea
\lab{eq:K^S}
K^S -\frac{1}{(r^S)^2} &=& -\left(\Delta_S+\frac{2}{(r^S)^2}\right)v+\frac{f(-v)}{(r^S)^2}, \quad\qquad f(v):=e^{2v}-1-2v. 
\eea
Multiplying by $J^{(p,S)}$, integrating on $S$ and using integration by parts, we deduce
\beaa
\int_S\left(K^S-\frac{1}{(r^S)^2}\right)J^{(p,S)} &=& -\int_S J^{(p,S)}\left(\Delta_S+\frac{2}{(r^S)^2}\right)v+\int_S\frac{f(-v)}{(r^S)^2}J^{(p,S)}\\
&=& -\int_Sv\left(\Delta_S+\frac{2}{(r^S)^2}\right)J^{(p,S)}+\int_S\frac{f(-v)}{(r^S)^2}J^{(p,S)}.
\eeaa
Together with \eqref{eq:PropertiesofJ^S}, the control $|v|\les\ep$ provided by \eqref{eq:Thm.-Eff.uniformalization1}, and the definition of $f$,  we infer
\beaa
\left|\int_S\left(K^S-\frac{1}{(r^S)^2}\right)J^{(p,S)}\right| &\les& \int_S\frac{\ep}{(r^S)^2}|v|+\int_S\frac{|v|^2}{(r^S)^2}\les \ep^2
\eeaa
as desired.
\end{proof}

    
    \section{Stability of uniformization for nearby  spheres}
    
    
  Consider two almost round spheres spheres $(S_1, g^{S_1} )$ and $(S_2, g^{S_2} )$, i.e. verifying \eqref{eq:Almostround-S},  and their respective  uniformization pairs $(\Phi_1, u_1)$, $(\Phi_2, u_2)$, i.e. 
  \bea
  \lab{eq:unifor-twometricsA}
  \bsplit
 & \Phi_1:\SSS^2 \longrightarrow S_1, \qquad   g_1:=\Phi_1^\#  (g^{S_1})= (r^{S_1})^2e^{2u_1} \ga_0,\\
  & \Phi_2:\SSS^2 \longrightarrow S_2, \qquad   g_2:=\Phi_2^\#(  g^{S_2})= (r^{S_2})^2e^{2u_2} \ga_0,
  \end{split}
  \eea
  and $u_1, u_2$ defined on $\SSS^2$   verifying the  conclusions of Corollary \ref{proposition:effective-uniformisation}. 
  We assume in addition given a   smooth  diffeomorphism   $\Psi:S_1\to S_2$ such  that the metrics $g^{S_1} $ and $\Psi^\#( g^{S_2} )$ are close to each other  in  $S^1$ with respect to   the coordinate chart provided by $\Phi_1$, i.e. for some $0<\de\le \ep$, 
 \bea
   \lab{eq:unifor-twometricsB}
         \left\| g^{S_1}-  \Psi^\#( g^{S_2}) \right\|_{L^\infty(S_1)} +     \frac{1}{(r^{S_1})}  \left\| g^{S_1}-  \Psi^\#( g^{S_2}) \right\|_{\hk_{4}(S^1)} \le (r^{S_1})^2  \de.
  \eea 
  The goal of this section is to     show  the existence of  a    canonical  diffeomorphism  $\Psih:\SSS^2\to \SSS^2$   which  relates the two uniformization maps. 
  More precisely  we prove the following.
    
  \begin{theorem}
  \lab{Thm:unifor-twometrics}
  Under the assumptions above, let $\Psih:\SSS^2\to \SSS^2$ be the  unique smooth diffeomorphism    such that $\Psi\circ \Phi_1=\Phi_2\circ \Psih$. Then,  the following holds true.
  \begin{enumerate}
  \item The diffeomorphism    $\Psih$ is smooth and  there exists $O\in O(3)$ such that
 \bea
  \lab{eq:unifor-twometrics1}
  \| \Psih-O\|_{L^\infty(\SSS^2)}+ \| \Psih-O\|_{H^1(\SSS^2)}&\les &\de.
 \eea
 \item The conformal factors $u_1, u_2$ verify
  \bea
  \lab{eq:unifor-twometrics2}
 \big \|u_1- \Psih^\# u_2\big \|_{L^\infty(\SSS^2)} \les \de.
  \eea
  \end{enumerate}
  \end{theorem}

\begin{remark}
Let us note the following concerning assumption  \eqref{eq:unifor-twometricsB}. 
\begin{itemize}
\item It is clearly not sharp in terms of regularity. Sharper results could be obtained by working in H\"older spaces. On the other hand,  in view of our applications, $\hk_s(S)$ are the natural spaces.

\item It is coordinate dependent. Though it is sufficient for our applications, it would be nice find a coordinate independent condition sufficient to recover the conclusions of Theorem \ref{Thm:unifor-twometrics}. 
\end{itemize}
\end{remark}

The proof of Theorem \ref{Thm:unifor-twometrics} is postponed to section  \ref{sec:proofofThm:unifor-twometrics}. 
It will rely in particular on the control of almost isometries of $\SSS^2$ discussed in the next section.

 
 \subsection{Almost isometries of $\SSS^2$}
 

   \begin{proposition} 
   \lab{prop:diffeomorphismTh}
   Let $\Th :\SSS^2\to \SSS^2$  be a $C^2$ diffeomorphism  such that, 
     \beaa
   \| \Th^\#\ga_0  -\ga_0\|_{L^\infty(\SSS^2)} \le \de.
   \eeaa
   Then, there exists  $O\in O(3)$ such that
   \bea
    \lab{eq:unifor-twometrics6}
\|\Theta - O\|_{L^\infty(\SSS^2)}+\|\Theta - O\|_{H^1(\SSS^2)} &\les& \de.
\eea
   \end{proposition}

   \begin{proof}
  We first extend the map $\Th$ to  a map  $\widetilde{\Th} :\RRR^3\setminus \{0\} \to \RRR^3\setminus \{0\} $  such that  for every $ x=r\om, \om\in \SSS^2$, $r>0$,
 \bea\lab{eq:definitionofwidetildeThetafortheproofofdiffeomorphismTh}
 \widetilde{\Th} (r\om)=r \Th(\om).
 \eea
 It is easy to check that, denoting by $e$ the euclidian metric,
 \bea\lab{eq:themapwidetildeThetaisanalmostisometryoftheEuclideanspace}
 \| \widetilde{\Th} ^\# e - e\|_{L^\infty(D) } \les \de, \qquad D:=\left\{ x\in \RRR^3\,\,\Big/\,\, \frac 1 2 \le |x|\le 2\right\}.
 \eea
 We deduce
 \beaa
 \|\nab \widetilde{\Th}  \big(\nab \widetilde{\Th} \big)^t -I \|_{L^\infty(D)} \les \de, 
 \eeaa
 i.e.  
 \beaa
 \|A A^t -I \|_{L^\infty(D)} \les \de, \qquad A:= \nab  \widetilde{\Th}, 
 \eeaa
  where $A$ is a function on $D$ taking values in 3 by 3 matrices.
 
 By the polar decomposition of  $A$,  there exists, at every point $p\in D$, a unique orthogonal matrix $O\in O(3)$  such  that
 \beaa
 A= (AA^t)^{\frac{1}{2}}  O.
 \eeaa
 We deduce, at every $p\in D$,
 \beaa
 |A-O |&=& \big|(AA^t)^{\frac{1}{2}} -I\big|\les    \de. 
 \eeaa
 Therefore,
 \beaa
 \sup_{D}\min_{ O\in O(3) } \Big| \nab \widetilde{\Th} -O\Big|&\les \de
 \eeaa
 i.e.,
 \beaa
\| \textrm{dist}(\nab \widetilde{\Th}, O(3))\|_{L^\infty(D)}\les \de.
 \eeaa
In particular, the determinant of $\nab \widetilde{\Th}$ cannot vanish on $D$, and hence, since $D$ is connected, is either strictly positive or strictly negative everywhere on $D$. From now on, we assume for simplicity that the determinant of $\nab \widetilde{\Th}$ is strictly positive everywhere on $D$, the other case being similar. We infer
 \bea\lab{eq:disctanceofnabwidetildeThtoOof3issmall}
\| \textrm{dist}(\nab \widetilde{\Th}, SO(3))\|_{L^\infty(D)}\les \de.
 \eea

 We rely on the following  result\footnote{The result in \cite{FJM} actually  covers only  $p=2$. The case $1<p<+\infty$ is in section 2.4 of S. Conti and B. Schweizer \cite{CoSc}.} of  G. Friesecke, R. James and S. M\"uller in \cite{FJM} pointed to us by  Camillo De Lellis. 
  \begin{theorem}
  \lab{Thm:FJM1}
 Let $ D\subset \RRR^n $  a bounded Lipschitz domain, $n\ge 2$ and  $1<p<+\infty$. There exists  a constant $C>0$  depending only on $D$  and $p$  with the property:
 
 For  every    $\psi \in  W^{1,p} (D, \RRR^n)$ there exists a rotation $O_0\in SO(n) $ such that,
 \bea
 \|\nab \psi - O_0\|_{L^p(D)} &\les&  C\| \textrm{dist}(\nab \psi, SO(n))\|_{L^p(D)}. 
 \eea
  \end{theorem} 
  
  \begin{remark}\lab{rem:FirtzJohnBMOestimate}
  The estimate of Theorem \ref{Thm:FJM1}  fails  in $L^\infty$ but still holds in BMO from  a  well known  estimate of F. John \cite{John}.
  \end{remark}

 We infer from \eqref{eq:disctanceofnabwidetildeThtoOof3issmall} and Theorem \ref{Thm:FJM1} the existence of a rotation 
 $O_0\in SO(3) $ such that 
\beaa
\|\nab  \widetilde{\Th} - O_0\|_{L^3(D)} + \|\nab  \widetilde{\Th} - O_0\|_{L^2(D)} &\les& \de.
\eeaa
Restricting  to the sphere, we infer,
\beaa
\|\nab \Th- O_0\|_{L^3(\SSS^2)} +\|\nab \Th- O_0\|_{L^2(\SSS^2)}  &\les& \de.
\eeaa
Together with Poincar\'e inequality and the Sobolev embedding, this yields
\beaa
\|\Theta - O_0\|_{L^\infty(\SSS^2)}+\|\nab (\Theta - O_0)\|_{L^2(\SSS^2)} &\les& \de
\eeaa
 as stated. This concludes the proof of Proposition \ref{prop:diffeomorphismTh}.
\end{proof}

    
    \subsection{Proof of Theorem \ref{Thm:unifor-twometrics}}
    \lab{sec:proofofThm:unifor-twometrics}
    
  
We note that  the following  is an  immediate consequence of  
 \eqref{eq:unifor-twometricsB} 
 \bea
 \lab{comp:rS1-to-rS2}
 \left|\frac{r^{S_1}}{  r^{S_2} }-1\right| &\les& \de.
 \eea
   Denote 
   \beaa
   g_1= \Phi^\#_1( g^{S_1}), \qquad  g_2= \big(\Psi\circ\Phi_1\big)^\#  ( g^{S_2}).
   \eeaa
    In view of  our  assumptions, we have
   \bea
    \lab{eq:unifor-twometrics3}
   g_1= (r^{S_1})^2e^{2u_1} \ga_0, \qquad \int_{\SSS^2} e^{2u_1} x =0, \qquad \|u_1\|_{H^4(\SSS^2)} \le \ep,
   \eea
and
 \bea
    \lab{eq:unifor-twometrics4}
   \|g_1-g_2\|_{L^\infty(\SSS^2)} +  \|g_1-g_2\|_{H^{4}(\SSS^2)} \les    (r^{S_1})^2   \de.
   \eea
  We deduce,
 \beaa
  \left\| e^{2 u_1} \ga_0 -    \frac{1}{(r^{S_1})^2 }   g_2 \right\|_{L^\infty(\SSS^2)} + \left\| e^{2 u_1} \ga_0 -    \frac{1}{(r^{S_1})^2 }   g_2 \right\|_{H^{4}(\SSS^2)} \les   \de, 
 \eeaa
 or, in view of the control of $u_1$ as well as \eqref{comp:rS1-to-rS2}
 \bea
  \| \ga_0 -   e^{-2 u_1}     \widetilde{g}_2  \|_{L^\infty(\SSS^2)}+   \| \ga_0 -   e^{-2 u_1}      \widetilde{g}_2 \|_{H^{4}(\SSS^2)} \les \de, \qquad  \widetilde{g}_2:=  \frac{1}{(r^{S_2})^2 }   g_2. 
 \eea
Consequently
 \beaa
 \| K\big( e^{-2u_1}  \widetilde{g}_2\big)-1\|_{H^2(\SSS^2)} \les \de. 
 \eeaa
We can thus apply Corollary \ref{proposition:effective-uniformisation} to the metric $e^{-2u_1} \widetilde{g}_2 $ and deduce that there exists, up to an isometry of $\SSS^2$,  a unique centered    conformal  $v$
 and a unique smooth  diffeomorphism $\Th:\SSS^2\to \SSS^2$  such that
\bea
  \lab{eq:unifor-twometrics5}
\Th^\#\Big(  e^{-2u_1}  \widetilde{g}_2 \Big)=(r_{1, 2})^2e^{2v} \ga_0, \qquad \|v \|_{H^{4}(\SSS^2)} \les \de.
\eea
 Hence,
\beaa
e^{-2 u_1}  \widetilde{g}_2 =(r_{1, 2})^2 e^{2 v\circ \Th^{-1} }  (\Th^{-1}) ^\# \ga_0.
\eeaa
Since $ \| \ga_0 -   e^{-2 u_1}    \widetilde{g}_2 \|_{H^{4}} \les \de$ it follows that,
\beaa
\| \ga_0 -  (r_{1, 2})^2e^{2 v\circ \Th^{-1} }  (\Th^{-1}) ^\# \ga_0\|_{H^{4}(\SSS^2)} \les \de, \qquad |r_{1, 2}-1|\les\de,
\eeaa
where we also used $\| v\|_{L^\infty(\SSS^2)} \les \de$ for the second inequality.  
Since $\| v\|_{H^{4}(\SSS^2)} \les \de$ and  $\Th$ is smooth  we deduce,
\bea
  \lab{eq:unifor-twometrics5'}
\| \ga_0 - \Th ^\# \ga_0\|_{H^{4}(\SSS^2)} \les \de.
\eea
  We now appeal to Proposition \ref{prop:diffeomorphismTh} which yields the existence of $O\in O(3)$ such that 
   \bea\lab{eq:consequenceof:unifor-twometrics5'}
\|\Theta - O\|_{L^\infty(\SSS^2)}+\|\Theta - O\|_{H^1(\SSS^2)} &\les& \de.
\eea

We rewrite \eqref{eq:unifor-twometrics5} in the form
\beaa
\Th^\# \widetilde{g}_2= (r_{1, 2})^2e^{2v+2u_1\circ \Th} \ga_0= e^{2\tilde{u}_2 }\ga_0, \qquad \tilde{u}_2:= v+ u_1\circ  \Th+\log(r_{1, 2}).
\eeaa
Note that   $\tilde{u}_2$ is a priori not   centered. On the other hand, we have
\bea\lab{eq:Linftyboundfortildeu2onSSS2}
\|\tilde{u}_2\|_{L^\infty} &\leq& \|v\|_{L^\infty}+\|u_1\|_{L^\infty}+|\log(r_{1, 2})|\les\ep.
\eea
Also, since $|r_{1, 2}-1|\les\de$, $\| v\|_{H^4}\les \de$  and  $\|\Th - O\|_{L^\infty}\les \de$,
\beaa
\int_{\SSS^2}  e^{\tilde{u}_2}  x da_{\ga_0} &=& \int_{\SSS^2}  e^{2v+2u_1\circ \Th+2\log(r_{1, 2})}  x da_{\ga_0}\\
 &=& \int_{\SSS^2}  e^{2u_1\circ \Th}    x da_{\ga_0} +  \int_{\SSS^2}  e^{2u_1\circ \Th} \big( (r_{1, 2})^2e^{2v}-1\big)   x da_{\ga_0}   \\
&= &\int_{\SSS^2}  e^{2u_1}   (x\circ \Th^{-1} ) \frac{1}{|\det  d \Th|} d a_{\ga_0}+O(\de)\\
&= &\int_{\SSS^2}  e^{2u_1}    x\circ O^{-1} d a_{\ga_0}    +O(\de).
\eeaa
We deduce, remembering that $O$ is a matrix in $O(3)$ and $x\in \mathbb{R}^3$,  
\beaa
\int_{\SSS^2}  e^{\tilde{u}_2}  x da_{\ga_0} &= &\int_{\SSS^2}  e^{2u_1}  O^{-1} x d a_{\ga_0}+O(\de)\\
 &=& O^{-1}\left(\int_{\SSS^2}  e^{2u_1}  x d a_{\ga_0}\right)+O(\de).
\eeaa
Since  $u_1$ is centered, we infer
\bea
  \lab{eq:unifor-twometrics7}
\int_{\SSS^2}  e^{2\tilde{u}_2} x &= & O(\de).
\eea
In view of \eqref{eq:Linftyboundfortildeu2onSSS2} and \eqref{eq:unifor-twometrics7}, Lemma \ref{lemma:improved-renormalization} applies so that there exists   $\psi\in \M$ such that\footnote{The statement of Lemma \ref{lemma:improved-renormalization} actually yields $|\psi-I|\les\de$. Here we use the fact that $\SS$ is invariant by $O(3)$, see Remark \ref{rem:ifuiscenteredthensoisucricO}, so that we may indeed assume, by composing with $O^{-1}\in O(3)$, that  $|\psi-O^{-1}|\les \de$.}
 \bea
  u_2 ':= (\widetilde{u}_2)_\psi = \tilde{u}_2\circ \psi+\frac 1 2 \log|\det \psi|  \in \SS,    \qquad \big\|\psi-O^{-1} \big\|_{L^\infty(\SSS^2)}\les \de.
  \eea
 We deduce
  \beaa
\psi^\#  \Th^\# \widetilde{g}_2 = e^{\tilde{u}_2 \circ \psi} \psi^\# \ga_0= e^{2\tilde{u}_2 \circ \psi}    | \det d\psi |\ga_0 =e^{2\tilde{u}_2 \circ \psi +\log|\det d\psi| } \ga_0= e^{2 u'_2} \ga_0.
\eeaa
Also, in view of the control on $\Th$ and $\psi$, we have
  \beaa
 \| \Th \circ \psi   -I\|_{L^\infty(\SSS^2)}+   \|   \Th \circ \psi  -I\|_{H^1(\SSS^2)}  \les \de.
  \eeaa
  Therefore we have found  a diffeomorphism 
  \bea
  \Psih' :=(\Th\circ \psi)^{-1}, \qquad \Psih':\SSS^2\longrightarrow\SSS^2,
  \eea
  with the properties
  \bea
  (\Psih'^{-1})^\# \widetilde{g_2} &=& e^{2 u'_2} \ga_0, \qquad u_2' \in \SS, 
  \eea
 and
\bea\lab{eq:almosttherightestimateetoconcludestabilityunifromizationproofupmissingrotation:0}
   \| \Psih'   -I\|_{L^\infty(\SSS^2)}+   \|  \Psih'  -I\|_{H^1(\SSS^2)}  \les \de.
 \eea
 Also, recall that  
 \beaa
u_2' &=&  \tilde{u}_2\circ \psi+\frac 1 2 \log|\det \psi|  = \big( v+ u_1\circ  \Th +2\log(r_{1, 2})\big)\circ\psi +\frac 1 2 \log|\det \psi|\\
&=&  u_1\circ \Psih'^{-1} + v_\psi +2\log(r_{1, 2}).
  \eeaa
 We infer
  \beaa
  \Psih'^\# u'_2 -  u_1 &=&   \Psih'^\#\big( v_\psi\big)  +2\log(r_{1, 2}),
   \eeaa
 and thus, since $|r_{1, 2}-1|\les\de$, $\| v\|_{H^{4}}  \le \de$,  $ \| \Psih'   -I\|_{L^\infty}\les\de$   and  $\|\psi-O^{-1}\|_{L^\infty} \les \de$, we deduce,
\bea\lab{eq:almosttherightestimateetoconcludestabilityunifromizationproofupmissingrotation:1}
  \|u_1 - \Psih'^\#(u'_2)  \|_{L^\infty(\SSS^2)} \les \de.
  \eea
  
 Next, let
 \beaa
 \Phi_2':=\Psi\circ \Phi_1\circ \Psih'^{-1}.
 \eeaa
 
Then, by        the definition of  $g_2$ and $\widetilde{g}_2$,
  \beaa
  \widetilde{g}_2=\frac{1}{(r^{S_2})^2 }   g_2= \frac{1}{(r^{S_2})^2 }(\Psi\circ \Phi_1)^\# g^{S_2},  \eeaa
   and      the       construction of $\Psih'$  we have
 \beaa
( \Phi_2')^\#  g^{S_2}&=&   (\Psih'^{-1})^\#  \big( \Psi\circ \Phi_1\big)^\# g^{S_2} =   (r^{S_2})^2 (\Psih'^{-1})^\#  \widetilde{g}_2=    (r^{S_2})^2   e^{2u_2'} \ga_0
 \eeaa
 and by the definition of $\Phi_2, u_2$ we also have,
 \beaa
 ( \Phi_2)^\#  g^{S_2}= (r^{S_2})^2   e^{2u_2} \ga_0.
 \eeaa
 Therefore, since both $u_2$ and $u_2'$ are centered, from the uniqueness statement
 of Corollary \ref{proposition:effective-uniformisation}, there existe $O\in O(3)$ such that,   
 \bea\lab{eq:almosttherightestimateetoconcludestabilityunifromizationproofupmissingrotation:2}
 u_2'=u_2\circ O, \qquad \Psi\circ \Phi_1=\Phi_2\circ O\circ \Psih',
 \eea
where the last identity corresponds to $\Phi_2'=\Phi_2\circ O$.   We now define
 \beaa
 \Psih &:=& O\circ \Psih'.
 \eeaa
 Then, together with \eqref{eq:almosttherightestimateetoconcludestabilityunifromizationproofupmissingrotation:0}, \eqref{eq:almosttherightestimateetoconcludestabilityunifromizationproofupmissingrotation:1} and \eqref{eq:almosttherightestimateetoconcludestabilityunifromizationproofupmissingrotation:2}, we infer
 \beaa
 \Psi\circ \Phi_1=\Phi_2\circ \Psih, \qquad  \| \Psih   -O\|_{L^\infty}+   \|  \Psih -O\|_{H^1}  \les \de, \qquad \|u_1 - \Psih^\#(u_2) \|_{L^\infty(\SSS^2)} \les \de,
\eeaa
as stated.  This concludes the proof of Theorem \ref{Thm:unifor-twometrics}.

 
 \subsection{Higher regularity estimates for  Theorem  \ref{Thm:unifor-twometrics}} 
 

 Though \eqref{eq:unifor-twometrics1} \eqref{eq:unifor-twometrics2}  suffice for  our main application to GCM spheres,  we provide here, for the sake of completeness,  a discussion of corresponding higher regularity estimates. This relies on a recent personal communication of Camillo De Lellis and  Stefan M\"uller. More precisely  they can prove the following.
 \begin{theorem}[C. De Lellis, S. M\"uller \cite{DeLellis}]
 \lab{th:DeLellisMullerhigherregularity}
Let $n\geq 2$. Consider  the balls  $B_1\subset B_2 \subset \RRR^n$ and $e$ the euclidian metric on $\RRR^n$. 
 Consider a map   $u\in H^1( B_2, \RRR^n)$  verifying $\det \nab u>0$ and,
 \bea
\big \| u^\#  e - e\big \|_{L^\infty(B_1)} \le \frac 1 2. 
 \eea
 Then, there exists a rotation $O\in SO(n)$  such that, relative to the H\"older norms $C^\a$, $0<\a<1$,
 \beaa
 \big\| \nab u- O\big  \|_{C^\a( B_1)} \les \big \| u^\#  e - e\big \|_{C^\a(B_2)}.
 \eeaa
 \end{theorem}
 
 The goal of this section is to prove the   following proposition that extends the estimates \eqref{eq:unifor-twometrics1},  \eqref{eq:unifor-twometrics2}  of Theorem \ref{Thm:unifor-twometrics} to corresponding higher regularity estimates.
  \begin{proposition}
  \lab{Thm:unifor-twometrics:higherregularity}
  Assume, in addition to  the assumptions of Theorem \ref{Thm:unifor-twometrics}, that 
    \bea
   \lab{eq:unifor-twometricsB:higherregularityassumption}
  \left\| g^{S_1}-  \Psi^\#( g^{S_2}) \right\|_{\hk_{4+s}(S^1)} \le \de
  \eea
  for some $s\geq 0$. Then, the following higher regularity analogs of  \eqref{eq:unifor-twometrics1}, \eqref{eq:unifor-twometrics2}  hold true.
    \begin{enumerate}
  \item The diffeomorphism    $\Psih$ is smooth and  there exists $O\in O(3)$ such that
 \bea
  \lab{eq:unifor-twometrics1:higherregularity}
  \| \Psih-O\|_{H^{5+s}(\SSS^2)}&\les &\de.
 \eea
 \item The conformal factors $u_1, u_2$ verify
  \bea
  \lab{eq:unifor-twometrics2:higherregularity}
 \big \|u_1- \Psih^\# u_2\big \|_{H^{4+s}(\SSS^2)} \les \de.
  \eea
  \end{enumerate}
  \end{proposition}
 
 The only part of the proof of Theorem \ref{Thm:unifor-twometrics} which does not immediately extend to higher regularity is when applying Proposition \ref{prop:diffeomorphismTh} to infer \eqref{eq:consequenceof:unifor-twometrics5'} from \eqref{eq:unifor-twometrics5'}. Thus, to complete the proof of Proposition \ref{Thm:unifor-twometrics:higherregularity}, it suffices to prove the following extension of Proposition \ref{prop:diffeomorphismTh} to higher regularity. 
 
   \begin{proposition} 
   \lab{prop:diffeomorphismTh:higherregularity}
   Let $\Th :\SSS^2\to \SSS^2$  be a $C^2$ diffeomorphism  such that, 
     \bea\lab{eq:themapThetaisanalmostisometryofthestandardspherehigherreg}
   \| \Th^\#\ga_0  -\ga_0\|_{H^{s+2}(\SSS^2)} \le \de
   \eea
  for some $s\geq 0$. Then, there exists  $O\in O(3)$ such that
   \bea
    \lab{eq:unifor-twometrics6:higherregularityesitmate}
\|\Theta - O\|_{H^{s+3}(\SSS^2)} &\les& \de.
\eea
   \end{proposition} 
 
 \begin{proof}
 Recall that the proof of Proposition \ref{prop:diffeomorphismTh} proceeds by extending the map $\Theta$ to a  map  $\widetilde{\Th} :\RRR^3\setminus \{0\} \to \RRR^3\setminus \{0\} $  such that  for every $ x=r\om, \om\in \SSS^2$, $r>0$,  
 \beaa
 \widetilde{\Th} (r\om)=r \Th(\om),
 \eeaa
see \eqref{eq:definitionofwidetildeThetafortheproofofdiffeomorphismTh}. Under the assumption \eqref{eq:themapThetaisanalmostisometryofthestandardspherehigherreg}, we obtain the following higher regularity analog of the estimate \eqref{eq:themapwidetildeThetaisanalmostisometryoftheEuclideanspace}
\bea\lab{eq:esaimteofwidetildeThetainviewofthelinearizationofitseequation}
 \| \widetilde{\Th}^\# e - e\|_{H^{s+2}(D) } \les \de, \qquad D=\left\{ x\in \RRR^3\,\,\Big/\,\, \frac 1 2 \le |x|\le 2\right\}.
 \eea
In view of Theorem \ref{th:DeLellisMullerhigherregularity},  covering $\SSS^2$ by finitely many euclidean balls\footnote{More precisely,   $B_1$ and $B_2$ of Theorem \ref{th:DeLellisMullerhigherregularity} are chosen to be euclidean balls centered on points of   $\SSS^2$ of radius given respectively  by $1/4$ and $1/2$ so that both $B_1$ and $B_2$ are included in $D$ and the union of the $B_1$'s includes $D_1$.}, and using that $H^{2}(D)$ embeds in $C^{1/2}(D)$  by Sobolev, we infer the existence of $O\in O(3)$ such that 
   \beaa
\|\nabla\widetilde{\Th} - O\|_{C^{\frac{1}{2}}(D_1)} \les \de, \qquad D_1:=\left\{ x\in \RRR^3\,\,\Big/\,\, \frac{3}{4} \le |x|\le \frac{5}{4}\right\}.
\eeaa
We deduce the following decomposition for $\widetilde{\Th}$
 \bea\lab{eq:decompostionofwidetildeThetainviewofthelinearizationofitseequation}
\widetilde{\Th}=(I+\widetilde{\Th}')O, \qquad \|\widetilde{\Th}'\|_{L^\infty(D_1)}+\|\nabla\widetilde{\Th}'\|_{L^\infty(D_1)}\les\de.
\eea

Next, we rewrite \eqref{eq:esaimteofwidetildeThetainviewofthelinearizationofitseequation} in matrix form 
\bea\lab{eq:linearizationofitseequationforwidetildeTheta:notlinearizedyet}
\nabla\widetilde{\Th}(\nabla\widetilde{\Th})^t =I+ E, \qquad \|E\|_{H^{s+2}(D) } \les \de.
 \eea
Linearizing this equation based on the decomposition of $\widetilde{\Th}$ provided by \eqref{eq:decompostionofwidetildeThetainviewofthelinearizationofitseequation}, we infer on $D_1$
\bea\lab{eq:linearizationofitseequationforwidetildeTheta}
S(\widetilde{\Th}') &=& E-\nabla\widetilde{\Th}'(\nabla\widetilde{\Th}')^t,  \qquad S(\widetilde{\Th}'):=\nabla\widetilde{\Th}'+(\nabla\widetilde{\Th}')^t.
\eea
Note that $S$ satisfies the well-known   Korn inequality, see for example \cite{Cia},
 \bea\lab{eq:KorninequalityonD1}
\|f\|_{H^1(D_1)}\leq C_{D_1}\Big(\|f\|_{L^2(D_1)}+\|S(f)\|_{L^2(D_1)}\Big), \qquad f:D_1\to \RRR^3. 
 \eea
 where $S(f)$ denotes  the  symmetric part of the gradient of $f$.
Differentiating  \eqref{eq:linearizationofitseequationforwidetildeTheta}, relying on the estimate \eqref{eq:linearizationofitseequationforwidetildeTheta:notlinearizedyet} for $E$ and the a priori bound \eqref{eq:decompostionofwidetildeThetainviewofthelinearizationofitseequation} for $\widetilde{\Th}'$ on $D_1$, and making use of Korn inequality \eqref{eq:KorninequalityonD1}, we infer
 \beaa
 \|\widetilde{\Th}'\|_{H^{s+3}(D_1)}\les\de
\eeaa
and hence
 \beaa
 \|\widetilde{\Th}-O\|_{H^{s+3}(D_1)}\les\de.
\eeaa
Restricting to the sphere, we obtain 
\beaa    
    \|\Theta - O\|_{H^{s+3}(\SSS^2)} &\les& \de
\eeaa 
 as desired.
 \end{proof}

 
  \subsection{Calibration  of  uniformization maps  between spheres}
  \lab{subsection:calibration}
  

 In order to eliminate the  arbitrariness with respect to isometries of $\SSS^2$  in Theorem  \ref{Thm:unifor-twometrics}, see \eqref{eq:unifor-twometrics1},  we calibrate 
  the effective   uniformization maps\footnote{Given by Corollary \ref{proposition:effective-uniformisation}.}  $\Phi_1:\SSS^2 \to S_1 $, $\Phi_2:\SSS_2  \to S_2  $, for  given  diffeomorphism  $\Psi :S_1 \to S_2  $,   as follows.
   \begin{definition}
   \lab{def:calibration}
    On $\SSS^2$ we  fix\footnote{In particular, one can choose $N=(0,0,1)$ and $v=(1,0,0)$.} a point $N$ and a unit vector $v$ in the tangent space $T_N\SSS^2$. Given $\Psi:S_1\to S_2$,  we say that the  effective uniformization maps  $\Phi_1:\SSS^2 \to S_1 $, $\Phi_2:\SSS^2  \to S_2  $  are calibrated relative to $\Psi$  if  the map  $\Psih:=(\Phi_2)^{-1}\circ \Psi\circ \Phi_1:\SSS^2\to \SSS^2$ is such that
\begin{enumerate}
\item  The map $\Psih$ fixes the point $N$, i.e. $\Psih(N)=N$.

\item    The tangent map    $\Psih_\#$ fixes the direction of $v$, i.e.    $\Psih_\#(v)=a_{1,2}v$ where $a_{1,2}>0$.

\item  The tangent map $\Psih_\#$ preserves the orientation of $T_N\SSS^2$.
 \end{enumerate}
 \end{definition}
 
 \begin{lemma}\lab{lemma:existenceofcalibrationforagivenmapPhi1}
  Given $\Psi:S_1\to S_2$ and a fixed  effective uniformization map $\Phi_1:\SSS^2 \to S_1 $ for $S_1$. Then,  there exists a unique,  effective,   uniformization  for $S_2$  calibrated   with   that of $S_1$ relative to $\Psi$.
 \end{lemma} 
 
 \begin{proof}
 Let  $\Phi_2'$ be an effective uniformization  map  for $S_2$  and let $\Psih':=(\Phi'_2)^{-1}\circ \Psi\circ \Phi_1$ so that $\Psih':\SSS^2\to \SSS^2$. Recall that for any $O\in O(3)$,  $\Phi_2$ given by $\Phi_2=\Phi_2'\circ O^{-1}$ is also an effective uniformization  map  for $S_2$.  We have in this case $\Psih=O\Psih'$, and $\Phi_1$ and $\Phi_2$  are  calibrated  relative to $\Psi$ if and only if $O$ satisfies the following.
\begin{enumerate}
\item $O(\Psih'(N))=N$.

\item $\Psih'_\#(Ov)=av$ where $a>0$.

\item $\Psih'_\#\circ O$ preserves the orientation of $T_N\SSS^2$.
 \end{enumerate} 
 Since $\Psih'$ is $O(\de)$ close to an element of $O(3)$, see  \eqref{eq:unifor-twometrics1}, there exists a unique $O\in O(3)$ satisfying the conditions above.
    \end{proof}

   We now state the following corollary to Theorem  \ref{Thm:unifor-twometrics}.   
  \begin{corollary}
  \lab{Cor:unifor-twometrics-calibrated}
  In addition to the assumptions    of Theorem  \ref{Thm:unifor-twometrics},  assume  that the 
  maps $\Phi_1, \Phi_2$ are calibrated  relative to $\Psi$ according to  Definition \ref{def:calibration}.  Then $\Psih$ verifies
  \bea
  \lab{eq:unifor-twometrics1-calibr}
  \| \Psih-I\|_{L^\infty(\SSS^2)}+ \| \Psih-I\|_{H^1(\SSS^2)}&\les &\de.
 \eea
  The conformal factors $u_1, u_2$ verify
  \bea
  \lab{eq:unifor-twometrics2-calibr}
 \big \|u_1- \Psih^\# u_2\big \|_{L^\infty(\SSS^2)} \les \de.
  \eea
  \end{corollary}
 
  \begin{proof}
In view of Theorem  \ref{Thm:unifor-twometrics}, there exists $O\in O(3)$ such that $\Psih$ satisfies 
\bea
  \lab{eq:unifor-twometrics1:bisfortheproof}
  \| \Psih-O\|_{L^\infty(\SSS^2)}+ \| \Psih-O\|_{H^1(\SSS^2)}&\les &\de
 \eea
and \eqref{eq:unifor-twometrics2-calibr} holds.

It remains to check   \eqref{eq:unifor-twometrics1-calibr}. In view of \eqref{eq:unifor-twometrics1:bisfortheproof} and the  fact that $\Phi_1, \Phi_2$ are calibrated  relative to $\Psi$ according to  Definition \ref{def:calibration}, we infer that
\begin{enumerate}
\item $O(N)=N+O(\de)$,

\item $O(v)=av+O(\de)$ with $a>0$.

\item $O$ preserves the orientation of $T_N\SSS^2$.
 \end{enumerate}
These conditions imply immediately that $|O-I|\les\de$ which together with \eqref{eq:unifor-twometrics1:bisfortheproof}  yields \eqref{eq:unifor-twometrics1-calibr} as desired.
  \end{proof}

 \begin{lemma}[Transitivity of calibrations]
 \lab{Le:Transitivitycalibrations}
 Let $\Phi_1:\SSS^2 \to S_1 $, $\Phi_2:\SSS^2  \to S_2  $ and $\Phi_3:\SSS^2  \to S_3  $ three  effective uniformization maps. Let $\Psi_{12}: S_1\to S_2$  and $\Psi_{13}:S_1\to S_3$ satisfying \eqref{eq:unifor-twometricsB}
 and assume that $\Phi_1$, $\Phi_2$  are calibrated relative to $\Psi_{12}$, while $\Phi_1$, $\Phi_3$  are calibrated relative to $\Psi_{13}$. Then, $\Phi_2$, $\Phi_3$  are calibrated relative to $\Psi_{23}:=\Psi_{13}\circ\Psi_{12}^{-1}$. 
\end{lemma}

\begin{proof}
Since  $\Phi_1$, $\Phi_2$  are calibrated relative to $\Psi_{12}$, and $\Phi_1$, $\Phi_3$  are calibrated relative to $\Psi_{13}$, $\Psih_{12}:=(\Phi_2)^{-1}\circ \Psi_{12}\circ \Phi_1$ and $\Psih_{13}:=(\Phi_3)^{-1}\circ \Psi_{13}\circ \Phi_1$ satisfy the three properties of Definition \ref{def:calibration}. Then, introducing
\beaa
\Psi_{23}=\Psi_{13}\circ\Psi_{12}^{-1}, \qquad \Psih_{23}:=(\Phi_3)^{-1}\circ \Psi_{23}\circ \Phi_2,
\eeaa
we have $\Psih_{23}=\Psih_{13}\circ\Psih_{12}^{-1}$ so that $\Psih_{23}$ also satisfy the three properties of Definition \ref{def:calibration}. Hence, $\Phi_2$, $\Phi_3$  are calibrated relative to $\Psi_{23}$ as desired.
\end{proof}

 
  \subsection{Comparison of $\ell=1$ modes between two spheres}
  
  
  Consider, as in Theorem  \ref{Thm:unifor-twometrics},   two almost round spheres   $(S_1, g^{S_1} )$ and $(S_2, g^{S_2} )$  and 
 a  smooth map $\Psi:S_1\to S_2$ such that, as in   \eqref{eq:unifor-twometricsB},     that the metrics $g^{S_1} $ and $\Psi^\#( g^{S_2} )$  are close to each other  in  $S^1$. Assume that   $(\Phi_1, u_1)$ , $(\Phi_2, u_2)$ are  effective  uniformization
 maps of $S_1$ and $S_2$,    calibrated  as in  definition \ref{def:calibration}.
We define 
  \beaa
  J^i=J^{S_i}= J^{\SSS^2}\circ \Phi_i^{-1}, \qquad i=1,2,
  \eeaa
to   be the $\ell=1$     canonical modes of  
   $S_1, S_2 $  according to  Definition  \ref{definition:ell=1mpdesonS}. 
   We  want to compare $J^1$ with $ J^2\circ \Psi^{-1} $. We prove the following result.
   \begin{proposition}
   \lab{Prop:comparison.canJ}
   Under the assumptions of  Corollary  \ref{Cor:unifor-twometrics-calibrated} we have
   \bea
  \Big| J^1- J^2\circ \Psi  \Big| &\les &\de.
   \eea
   \end{proposition}
   
   \begin{proof}
   Indeed, using that $\Psi\circ\Phi_1=\Phi_2\circ\Psih$,
   \beaa
   J^2\circ \Psi= J^{\SSS^2}\circ \Phi_2^{-1} \circ \Psi =J^{\SSS^2}\circ\Psih\circ \Phi_1^{-1}. 
   \eeaa
   Hence,
   \beaa
    J^1- J^2\circ \Psi= J^{\SSS^2}\circ \Phi_1^{-1} -J^{\SSS^2}\circ\Psih\circ \Phi_1^{-1}. 
   \eeaa
   This implies, together with \eqref{eq:unifor-twometrics1-calibr},
   \beaa
   \Big| J^1- J^2\circ \Psi\Big| \les \sup_{\SSS^2} \Big|  I -\Psih\Big| \les \de
   \eeaa
    as stated.
   \end{proof}


\section{Review of the  geometric set-up in \cite{KS-Kerr1}}



\subsection{Background spacetime}
\lab{sec:backgroundspacetime}


As in   \cite{KS-Kerr1},  we  consider given a  vacuum spacetime $\RR$   with metric $\g$  
endowed with an outgoing geodesic foliation  by  spheres $S(u, s)$   of fixed $(u, s) $,     where $u$  is an outgoing optical function\footnote{That is  $\g^{\a\b}\pr_\a u\pr_\b u  =0$.}  with $L=-\g^{\a\b} \pr_\b u \pr_\a  $ its  null geodesic generator and $s$ chosen such that
\beaa
L(s) =\frac{1}{\vsi}, \qquad L(\vsi)=0.
\eeaa
Let $e_4=\vsi L $ and $e_3$  the unique    null vectorfield orthogonal to  $S(u,s)$   and such that $\g(e_3, e_4)=-2$. We  then  let $(e_1, e_2)$  an orthogonal basis of the tangent space of $S(u, s)$.   
The  corresponding  connection coefficients   relative to the null frame $(e_3, e_4, e_1, e_3)$  are denoted by $\chi, \chib, \xi, \xib, \om, \omb, \eta, \etab, \ze$ and  the null components of the curvature tensor by $\a, \aa, \b, \bb, \rho, \rhod$.    For the convenience of the reader  we recall their definition below.
 \beaa
 \begin{split}
\chib_{ab}&=g(\D_ae_3, e_b),\qquad \,\,\chi_{ab}=g(\D_ae_4, e_b),\\
\xib_a&=\frac 1 2 \g(\D_3 e_3 , e_a),\qquad \xi_a=\frac 1 2 \ g(\D_4 e_4, e_a),\\
\omb&=\frac 1 4 g(\D_3e_3 , e_4),\qquad\,\, \om=\frac 1 4 g(\D_4 e_4, e_3),\qquad \\
\etab_a&=\frac 1 2 (\D_4 e_3, e_a),\qquad \quad \eta_a=\frac 1 2 g(\D_3 e_4, e_a),\qquad\\
 \ze_a&=\frac 1 2 g(\D_{e_a}e_4,  e_3),
 \end{split}
\eeaa
and
\beaa
\a_{ab} &=&\R(e_a, e_4, e_b, e_4) , \qquad \aa_{ab} =\R(e_a, e_3, e_b, e_3), \\
\b_{a} &=&\frac 1 2 \R_(e_a, e_4, e_3, e_4), \qquad \bb_{a} =\frac 1 2 \R_(e_a, e_3, e_3, e_4),\\
\rho&=& \frac 1 4  \R_(e_3, e_4, e_3, e_4), \qquad \rhod =\dual R( e_3, e_4, e_3, e_4) .
\eeaa

  The null second fundamental forms $\chi, \chib$ are further  decomposed in their traces $\ka=\trch$ and $\kab=\trchb$, and traceless parts $\chih$ and $\chibh$.         We recall, see Lemma 2.4 in \cite{KS-Kerr1}, that the geodesic nature of the foliation implies
\beaa
\om=\xi=0,  \qquad  \etab = -\ze, \qquad  \vsi=\frac{2}{e_3(u)}.
\eeaa
The volume radius $ r=r(u, s)$ is defined   such  that the volume of $S$ is given by $4\pi r^2$. The Hawking mass  $m=m(u, s)$ of $S=S(u, s) $ is given by the formula 
 \bea
\frac{2m}{r}=1+\frac{1}{16\pi}\int_{S_{}}\trch \trchb.
\eea
The Gauss curvature of $S$ is denoted by $K$. It verifies the Gauss equation
\bea\lab{eq:Gaussequation}
K=-\rho -\frac{1}{4} \trch \trchb +\frac{1}{2}\chih\c\chibh.
\eea
The mass aspect function $\mu$  is defined by
\bea
\label{def:massaspectfunctions.general}
\mu &:=& - \div \ze -\rho+\frac 1 2  \chih\c \chibh.
\eea
As in \cite{KS-Kerr1}, we define the renormalized quantities 
\beaa
&&\widecheck{\trch} := \trch -\frac{2}{r}, \qquad
\widecheck{\trchb} := \trchb +\frac{2\Up}{r},\qquad 
\widecheck{\omb} := \omb -\frac{m}{r^2},\\
&&\widecheck{K} := K -\frac{1}{r^2},\qquad\,\,\,\,\widecheck{\rho} := \rho +\frac{2m}{r^3},\qquad \quad\,\,\,
\widecheck{\mu} := \mu -\frac{2m}{r^3}, \\
&&\widecheck{\Omb} :=\Omb+\Up, \qquad\quad 
\widecheck{\varsigma} := \varsigma-1,
\eeaa
where 
\beaa
\Up :=1-\frac{2m}{r},
\eeaa
and the sets
\bea
\lab{definition:Ga_gGa_b}
\bsplit
\Ga_g &:= \Bigg\{\widecheck{\trch},\,\, \chih, \,\, \ze, \,\, \widecheck{\trchb},\,\,  r\widecheck{\mu} ,\,\,  r\widecheck{\rho}, \,\, r\dual\rho, \,\, r\b, \,\, r\a, \,\, r\widecheck{K}, \,\, r^{-1} \big(e_4(r)-1\big),\,\, r^{-1}e_4(m)\Bigg\},\\
\Ga_b &:= \Bigg\{\eta, \,\,\chibh, \,\, \ombc, \,\, \xib,\,\,  r\bb, \,\, \aa, \,\, r^{-1}\widecheck{\Omb}, \,\,r^{-1}\widecheck{\varsigma}, \,\, r^{-1}(e_3(r)+\Up\big),  \,\, r^{-1}e_3(m)  \Bigg\}.
\end{split}
\eea


\subsubsection{Adapted coordinates}


Recall, see \cite{KS-Kerr1},   that  a coordinate system $(u, s, y^1, y^2)$ is  said to be adapted to  an outgoing geodesic foliation   as above if $e_4(y^1)= e_4(y^2)=0$.  In that case the spacetime metric can be written in the form
\bea
\lab{spacetimemetric-y-coordinates}
\g &=& - 2\vsi du ds + \vsi^2\Omb  du^2 +g_{ab}\big( dy^a- \vsi \undB^a du\big) \big( dy^b-\vsi \undB^b du\big),
\eea
where
\bea
\Omb=e_3(s), \qquad \undB^a =\frac{1}{2} e_3(y^a), \qquad  g_{ab}=\g(\pr_{y^a}, \pr_{y^b}).
\eea
Relative to these coordinates
\beaa
e_4=\pr_s, \qquad \pr_u = \vsi\left(\frac{1}{2}e_3-\frac{1}{2}\Omb e_4-\undB^a\pr_{y^a}\right), \qquad \pr_{y^a}= \sum_{c=1,2} Y_{(a)}^c e_c, \qquad   a=1,2,
\eeaa
with coefficients  $ Y_{(a)}^b $ verifying
\beaa
g_{ab}=\sum_{c=1, 2} Y_{(a)}^c Y_{(b)}^c.
\eeaa

As in \cite{KS-Kerr1}  we assume  that $\RR $ is covered by two coordinate systems , i.e. $\RR=\RR_N\cup \RR_S$,
such that
\begin{enumerate}
\item The North coordinate chart   $\RR_N$ is given by the coordinates
$(u, s, y_{N}^1, y_{N}^2)$ with    $(y^1_{N})^2+(y^2_{N})^2<2$. 

\item The South coordinate chart  $\RR_S$ is  given by the coordinates
$(u, s, y_{S}^1, y_{S}^2)$  with $(y^1_{S})^2+(y^2_{S})^2<2$. 

 \item   The two coordinate charts   intersect in  the  open equatorial region
 $\RR_{Eq}:=\RR_N\cap \RR_S$ in which both coordinate systems are defined.
 
 \item  In $\RR_{Eq} $   the transition functions  between the two coordinate  systems are given by  the smooth  functions $ \varphi_{SN}$ and $\varphi_{NS}= \varphi_{SN}^{-1} $. 
 \end{enumerate}
The metric coefficients for the two coordinate systems   are given by
 \beaa
\g &=& - 2\vsi du ds + \vsi^2\Omb  du^2 +g^{N}_{ab}\big( dy_N^a- \vsi \undB_{N}^a du\big) \big( dy_N^b-\vsi \undB_N^b du\big),\\
\g &=& - 2\vsi du ds + \vsi^2\Omb  du^2 +g^{S}_{ab}\big( dy_S^a- \vsi \undB_{S}^a du\big) \big( dy_S^b-\vsi \undB_S^b du\big),
\eeaa
where
\beaa
\Omb=e_3(s), \qquad \undB_N^a =\frac{1}{2} e_3(y_N^a), \qquad \undB_S^a =\frac{1}{2} e_3(y_S^a).
\eeaa
On  a 2-sphere $S(u,s)$ and   $f\in \SS_p(S)$, $p=0,1,2$,    we consider   the following   norms, 
  \bea
  \label{Norms-spacetimefoliation-GSMS}
  \bsplit
  \| f\|_{\infty} :&=\| f\|_{L^\infty(S)}, \qquad  \| f\|_{2} :=\| f\|_{L^2(S)}, \\
  \|f\|_{\infty,k} &= \sum_{i=0}^k \|\dk^i f\|_{\infty },  \qquad 
\|f\|_{2,k}=\sum_{i=0}^k \|\dk^i f\|_{2},
\end{split}
  \eea
  where $\dk^i$ stands for any   combination  of length $i$ of operators  of the form 
   $e_3, r e_4,  r\nab $.


\subsubsection{Reduced region  $\RR$}
\lab{subsubsect:regionRR1}


In what follows we restrict our attention to  smaller region   of the original $\RR$,  still denoted  $\RR$ defined as follows.
 \begin{definition} 
\label{defintion:regionRRovr}
Let $m_0>0$ a constant.   Let $\epg >0$   two sufficiently   small   constants, and let  $(\ug, \sg, \rg)$ three real numbers with $\rg$ sufficiently large so that
\bea\lab{eq:rangeofrgandepsilon}
\epg\ll m_0, \qquad\qquad  \rg\gg m_0.
\eea
We define  $\RR=\RR$  to be the region
\bea
\lab{definition:RR(dg,epg)}
\RR:=\left\{|u-\ug|\leq\epg,\quad |s-\sg|\leq  \epg \right\},
\eea
such that  assumptions {\bf A1-A4} below  with constant $\epg$  on  the background foliation of $\RR$,   are verified. 
\end{definition}


\subsubsection{Main assumptions for  $\RR$}
\lab{subsubsect:regionRR2}


Given an integer $s_{max}\geq 3$, we assume\footnote{In  view of \eqref{eq:assumtioninRRforGagandGabofbackgroundfoliation}, we will often replace $\Ga_g$ by $r^{-1} \Ga_b$.} the following.
\begin{enumerate}
\item[\bf A1.]
For  $k\le s_{max}$
\bea\lab{eq:assumtioninRRforGagandGabofbackgroundfoliation}
\bsplit
\| \Ga_g\|_{k, \infty}&\leq  \epg  r^{-2},\\
\| \Ga_b\|_{k, \infty}&\leq  \epg  r^{-1}.
\end{split}
\eea

\item[\bf A2.]  The Hawking mass $m=m(u,s)$ of  $S(u, s)$ verifies 
\bea\lab{eq:assumtionsonthegivenusfoliationforGCMprocedure:Hawkingmass} 
\sup_{\RR}\left|\frac{m}{m_0}-1\right| &\leq& \epg.
\eea

\item[\bf A3.] 
In the  region of their respective validity\footnote{That is  the quantities on the left verify the  same estimates as those for $\Ga_b$, respectively $\Ga_g$.}   we have
\bea
 \undB_N^a,\,\, \undB_S^a \in r^{-1}\Ga_b, \qquad  Z_N^a,\,\, Z_S^a \in \Ga_b,\qquad r^{-2} \widecheck{g}^{N}_{ab},  \,\, r^{-2} \widecheck{g}^{S}_{ab} \in r\Ga_g
 \eea
 where
 \beaa
 \widecheck{g} ^{N}\!_{ab} &=&   g^N_{ab}   -   \frac{4r^2}{1+(y^1_{N})^2+(y^2_{N})^2) }\de_{ab},\\
 \widecheck{g}^{S}\!_{ab} &=&   g^S_{ab}   -   \frac{4r^2}{(1+(y^1_{S})^2+(y^2_{S})^2) } \de_{ab}.
 \eeaa

\item[\bf  A4.] We assume  the existence of a   smooth family of  scalar functions\footnote{The property of the scalar functions $\Jp$ above is motivated by the fact that the $\ell=1$ spherical harmonics  on the standard sphere  $\SSS^2$,  given by  $J^{(0, \SSS^2)}=\cos\th, \, J^{(+, \SSS^2)}=\sin\th\cos\vphi, \,  J^{(-, \SSS^2)}=\sin\th\sin\vphi$, 
satisfy  \eqref{eq:Jpsphericalharmonics}  with $\epg=0$. Note also  that on $\SSS^2$,
\beaa
\int_{\mathbb{S}^2}(\cos\th)^2=\int_{\mathbb{S}^2}(\sin\th\cos\vphi)^2=\int_{\mathbb{S}^2}(\sin\th\sin\vphi)^2=\frac{4\pi}{3}, \qquad |\SSS^2|=4\pi.
\eeaa,} 
$\Jp:\RR\to\RRR$, for $p=0,+,-$,   verifying the following properties, for all surfaces $S$ of the background foliation
 \bea
 \lab{eq:Jpsphericalharmonics}
\bsplit
  \Big(r^2\lap+2\Big) \Jp  &= O(\epg),\qquad p=0,+,-,\\
\frac{1}{|S|} \int_{S}  \Jp J^{(q)} &=  \frac{1}{3}\de_{pq} +O(\epg),\qquad p,q=0,+,-,\\
\frac{1}{|S|}  \int_{S}   \Jp   &=O(\epg),\qquad p=0,+,-.
\end{split}
\eea
\end{enumerate}

\begin{remark}
 We note that the assumptions {\bf A1}, {\bf A2}, {\bf A3}, {\bf A4},  are expected to be valid   in the far regions, i.e. $r$ large, of a  perturbed Kerr. 
\end{remark} 

Given a scalar function $f$   defined on a sphere $S$ we define the $\ell=1$ modes of $f$ by   the triplet
\beaa
(f)_{\ell=1} :=\left\{ \int_S f \Jp, \qquad   p=0,+,-\right\}.
\eeaa


\subsubsection{$O(\epg)$-spheres}


Given a compact    $2$-surface $\S\subset \RR$, not necessarily   a leaf  $S(u,s)$ of the  background geodesic foliation of $\RR$, we  denote by  $\chi^\S$, $\chib^\S$, $\ze^\S$,..., the corresponding Ricci coefficients and by $\a^\S$, $\b^\S$, $\rho^\S$, ..., the corresponding curvature coefficients. We also denote  by $r^\S$, $m^\S$, $K^\S$ and $\mu^\S$ respectively the corresponding area radius, Hawking mass, Gauss curvature  and mass aspect function. 
Finally, we  denote by   $\nab^\S$ the corresponding  covariant derivative  on $\S$.

Also,   we introduce the following norm on $\S$
   \bea
   \lab{definition:spaceH^k(boldS)}
\| f\|_{\hk_s(\S)}:&=&\sum_{i=0}^s \|( \dkb^\S )^i f\|_{L^2(\S)}, \qquad   \dkb^\S =r^\S \nab^\S.
\eea


\subsubsection{General  frame transformations}


\begin{lemma}
\lab{Lemma:Generalframetransf}
Given a null frame $(e_3, e_4, e_1, e_2)$, a general null transformation     from  the null frame  $(e_3, e_4, e_1, e_2)$  to another null frame   $(e_3', e_4', e_1', e_2')$         can be written in   the form,
 \bea
 \lab{eq:Generalframetransf}
 \bsplit
  e_4'&=\la\left(e_4 + f^b  e_b +\frac 1 4 |f|^2  e_3\right),\\
  e_a'&= \left(\de_{ab} +\frac{1}{2}\fb_af_b\right) e_b +\frac 1 2  \fb_a  e_4 +\left(\frac 1 2 f_a +\frac{1}{8}|f|^2\fb_a\right)   e_3,\qquad a=1,2,\\
 e_3'&=\la^{-1}\left( \left(1+\frac{1}{2}f\c\fb  +\frac{1}{16} |f|^2  |\fb|^2\right) e_3 + \left(\fb^b+\frac 1 4 |\fb|^2f^b\right) e_b  + \frac 1 4 |\fb|^2 e_4 \right),
 \end{split}
 \eea
  where $\la$ is a scalar, $f$ and $\fb$ are horizontal 1-forms. The dot product and magnitude  $|\c |$ are taken with respect to the standard euclidian norm of $\RRR^2$.  We call $(f, \fb, \la)$  the transition coefficients of the change of frame.   We denote $F:=(f, \fb, \ovla)$ where $\ovla=\la-1$.
  \end{lemma}
  Relative to the primed frame    the  connection   coefficients     $\Ga' $  can be related to  the connection coefficients 
  $\Ga$ of the background frame by  specific formulas  involving  the transition coefficients and their derivatives. 
   Similarly the     curvature coefficients    $R'$ of the primed frame  are connected to   $R$  by formulas involving
   the transition coefficients.       The precise formulas are given  in Proposition 3.3 in \cite{KS-Kerr1}.   In this paper we  only make use of the transformation formulas for $\ka=\trch$, $\kab =\trchb$ and $\b$. The transformation  formulas for the first two  have the form, see Proposition 3.4 in \cite{KS-Kerr1}, 
   \bea
   \lab{eq:transfforbandkab1}
   \bsplit
\ka' &= \ka+ \ka \ovla+\div'f   +F\c  \Ga_b+ F\c\nab' F+r^{-1} F^2,\\
\kab' &= \kab- \kab \ovla+\div'\fb   +F\c  \Ga_b+ F\c\nab' F+r^{-1} F^2.
\end{split}
\eea
     The transformation formula for $\b$   has the form, see Proposition 3.3. in \cite{KS-Kerr1},
   \bea
   \lab{eq:transfforbandkab2}
     \b'&=&  \b +\frac 3 2\big(  f \rho+\dual  f  \rhod\big)+ r^{-1} \Ga_g \c F.
     \eea

   
\subsection{Deformation of surfaces in $\RR$}

 
 \begin{definition}
 \label{definition:Deformations}
 We say that    $\S$ is an $O(\dg)$\,  deformation of $ \ovS$ if there exist  smooth  scalar functions $U, S$ defined on $\ovS$ and a  map 
  a map $\Psi:\ovS\to \S $  verifying, on  any  coordinate  chart  $(y^1, y^2) $ of $\ovS$,  
   \bea
 \Psi(\ovu, \ovs,  y^1, y^2)=\left( \ovu+ U(y^1, y^2 ), \, \ovs+S(y^1, y^2 ), y^1, y^2  \right)
 \eea
 with $(U, S)$ smooth functions on $\ovS$  of size $\dg$.
 \end{definition}

\begin{definition}
Given a deformation $\Psi:\ovS\to \S$ we  say that 
 a new frame   $(e_3', e_4',  e_1', e_2')$ on $\S$, obtained from the standard frame $(e_3, e_4, e_1, e_2)$  via the transformation  \eqref{eq:Generalframetransf},  is  $\S$-adapted  if   the horizontal  vectorfields $e'_1, e'_2$   are tangent to $\S$ or, equivalently $e_3', e_4' $ are perpendicular to  $S$.
\end{definition}

We will need the following lemmas.
\begin{lemma}
\lab{Le:Transportcomparison}
Let     $\Psi:\ovS\longrightarrow \S $   be  a  deformation in $\RR$ as in Definition \ref{definition:Deformations},    $F$ a scalar function on $\RR$ and $F^\# =F\circ \Psi $  its pull back to $\ovS$ by $\Psi$. We have
\bea
\big\| F^\#-F\big\|_{L^\infty(\ovS)}&\les&  \big \| ( U, S) \big\|_{L^\infty(\ovS)}  \sup_{\RR}\Big(\big|  e_3  F \big|+ |e_4 F|+| \nab F |\Big).
\eea
\end{lemma}

\begin{proof}
See Lemma 5.7 in \cite{KS-Kerr1}.
\end{proof}

 \begin{lemma}\lab{lemma:comparison-gaS-ga}
 Let $\ovS \subset \RR$.    Let  $\Psi:\ovS\longrightarrow \S $  be   a  deformation generated by the  functions $(U, S)$ as in Definition \ref{definition:Deformations} and denote by $g^{\S,\#}$  the  pull back of the metric $g^\S$ to $\ovS$. Assume the bound, for $s\le   s_{max}+1$.
 \bea
 \label{assumption-UV-dg}
   \| (U, S)\|_{L^\infty(\ovS)} +r ^{-1} \big\|(U, S)\big\|_{\hk_{s}(\ovS)}   &\les&  \dg.
 \eea
  Then 
   \bea
 \frac{r^\S}{\ovr}= 1 + O(r ^{-1}  \dg )
 \eea
 where $r^\S$ is the area radius of $\S$ and $\ovr$ that of $\ovS$.  
 
 Also,
  \bea
  \big\|  g^{\S, \#} -\ovg\big\|_{L^\infty} +r^{-1} \big\|  g^{\S, \#} -\ovg\big\|_{\hk_{s}(\ovS)}\les \dg  r.
  \eea
\end{lemma}

\begin{proof}
See Lemma  5.8 in \cite{KS-Kerr1}.
\end{proof} 

As a corollary, we deduce the following.
\begin{lemma}
\lab{Lemma:coparison-forintegrals}
        Under the same assumptions as in Lemma \ref{lemma:comparison-gaS-ga},  the following estimate holds  for   a scalar function $F$ defined on $\RR$.           
\beaa
\left|\int_\S F -\int_{\ovS} F\right| &\les& \dg   r   \sup_{\RR} \Big( |F| +  r \big(|e_3 F|+ |e_4 F|+|\nab F| \big)\Big).
\eeaa
\end{lemma}

\begin{proof}
We have, 
\beaa
\int_\S F -\int_{\ovS} F&=&\int_{\ovS}  F\circ \Psi \frac{\sqrt{\det g^{\S,\#}}}{\sqrt{ \det \ovg}} da_{\ovg}-\int_{\ovS} F da_{\ovg} \\
&=&
 \int_{\ovS}  F\circ\Psi \left(\frac{\sqrt{\det g^{\S,\#}}}{\sqrt{\det \ovg}}-1\right) + \int_{\ovS} \big|F\circ\Psi-F\big|
\eeaa
where $g^{\S,\#}$ is the pull back of the metric $g^\S$ to $\ovS$ by the map $\Psi$. 
In view of Lemma \ref{lemma:comparison-gaS-ga}
 \beaa
   \big\| g^{\S, \#} -\ovg \big\|_{L^\infty(\ovS)} &\les r\dg.
 \eeaa 
Hence, using also  Lemma \ref{Le:Transportcomparison},
\beaa
\left|\int_\S F -\int_{\ovS} F\right|&\les&\dg  \rg \sup_{\ovS} | F|+ 
  \big(\ovr\big)^2 \sup_{\ovS}             \big |F\circ\Psi-h\big|\\
  &\les& \dg  \rg \sup_{\ovS} | F|+ \dg ( \ovr )^2 \Big(|e_3 F|+ |e_4 F|+|\nab F| \Big)
\eeaa
as stated.
\end{proof}


\subsection{Adapted non canonical  $\ell=1$ modes}


Consider a deformation $\Psi:\ovS \to \S$ and recall the existence of the family of scalar functions $\Jp$, $p\in\big\{0, +, -\big\}$, on $\RR$  introduced in assumption {\bf A4}, see  \eqref{eq:Jpsphericalharmonics}, which form a basis of the $\ell=1$ modes on the spheres $S(u,s)$ of $\RR$, and hence in particular on $\ovS$.  We   denote by 
$\Jpov $ the restriction of the family $\Jp$ to $\ovS$.

 \begin{definition} 
 \lab{def:ell=1sphharmonicsonS}
 We define the basis of adapted $\ell=1$ modes  $\JpS$   on $\S$ by
 \beaa
\JpS = \Jpov\circ\Psi^{-1}, \qquad p\in\big\{ -, 0, +\big\}.
 \eeaa
 \end{definition}
 
 \begin{proposition}
 Assume  that the parameters  $(U, S)$ of the   deformation  verifies the bounds 
 \bea\lab{eq:boundonUSin h_snorms}
   \|( U, S)\|_{\hk_{s_{max}+1}(\ovS)}  &\les& r\dg.
 \eea
 Assume  $\Jp$ is an admissible triplet of $\ell=1$ modes on $\RR$ (and hence on $\ovS$), i.e. satisfying \eqref{eq:Jpsphericalharmonics}.
 
  Then  $\JpS$ is  an admissible triplet of $\ell=1$ modes on $\S$, i.e.
\bea
\lab{eq:admissibleJpS}
\bsplit
 \Big((r^\S)^2\lap^\S +2\Big) \JpS  &= O(\epg),\\
 \frac{1}{|\S|} \int_{\S}  \JpS J^{( q, \S)}  &=  \frac{1}{3}\de_{pq} +O(\epg),\\
 \frac{1}{|\S|}  \int_{\S}   \JpS   &=O(\epg).
\end{split}
\eea

Moreover at all point  of $\S$  we have
\bea
\lab{eq:admissibleJpS2}
\Big|\JpS-\Jp\Big|&\les \dg.
\eea
 \end{proposition} 
 
\begin{proof}
See Proposition 5.12 in  \cite{KS-Kerr1}.
\end{proof}


\section{GCM spheres with canonical $\ell=1$ modes}



\subsection{GCM spheres with non-canonical $\ell=1$ modes in \cite{KS-Kerr1}}


We review below  the  result   on existence  and uniqueness of GCM spheres  proved in \cite{KS-Kerr1}. The result was proved  in the context of  an arbitrary choice of a $\ell=1$ modes on $\ovS$, denoted $\Jp$  which verify 
 the assumptions {\bf A4}.

\begin{theorem}[GCM spheres with non-canonical $\ell=1$ modes \cite{KS-Kerr1}]
\lab{Theorem:ExistenceGCMS1}
 Let  $\ovS=S(\ovu, \ovs)$   be  a fixed    sphere  of the background geodesic  foliation\footnote{ Verifying the  assumptions ${\bf A1-A4}$.  Note  in particular that assumption {\bf A3}  was made w.r.t.  a general version of the  ${\ell=1} $ modes of  $\S$, see Definition \ref{def:ell=1sphharmonicsonS}.} of $\RR$ with  $\rg$ and $\mg$ denoting  its area radius and  Hawking mass.
 Assume that  the GCM quantities   $\ka, \kab, \mu$  of the background foliation verify          the  following:
\bea
\bsplit
\ka&=\frac{2}{r}+\dot{\ka},\\
\kab&=-\frac{2\Up}{r} +  \Cb_0+\sum_p \Cbp \Jp+\dot{\kab},\\
\mu&= \frac{2m}{r^3} + M_0+\sum _p\Mp \Jp+\dot{\mu},
\end{split}
\eea
where
\bea
\lab{eq:GCM-improved estimate1-again}
|\Cb_0, \Cbp| &\les& r^{-2} \epg, \qquad  |M_0, \Mp| \les r^{-3} \epg,
\eea
and
\bea
\lab{eq:GCM-improved estimate2-again}
\big\| \kadot, \kabdot\|_{\hk_{s_{max} }(\S) }&\les& r^{-1}\dg,\qquad 
\big\|\mudot\| _{\hk_{s_{max} }(\S) }\les r^{-2}\dg.
\eea
 Then
for any fixed triplets   $\La, \Lab \in \RRR^3$  verifying
\bea\lab{eq:assumptionsonLambdaabdLambdabforGCMexistence}
|\La|,\,  |\Lab|  &\les & \dg,
\eea
 there 
exists a unique  GCM sphere $\S=\S^{(\La, \Lab)}$, which is a deformation of $\ovS$, 
such that  there exist constants $\Cb^\S_0$, $\CbpS$, $ M^\S_0$, $\MpS$, $p\in\{-,0, +\}$  for which the   following GCM conditions   are verified
 \bea
\lab{def:GCMC}
\bsplit
\ka^\S&=\frac{2}{r^\S},\\
\kab^\S &=-\frac{2}{r^\S}\Up^\S+  \Cb^\S_0+\sum_p \CbpS \JpS,\\
\mu^\S&= \frac{2m^\S}{(r^\S)^3} +   M^\S_0+\sum _p\MpS \JpS,
\end{split}
\eea
where 
\bea
\lab{eq:non-canonicalJpS}
\JpS=\Jpov\circ\Psi^{-1}.
\eea 
Relative to these modes we also have
 \bea
\lab{GCMS:l=1modesforffb}
(\div^\S f)_{\ell=1}&=\La, \qquad   (\div^\S\fb)_{\ell=1}=\Lab.
\eea

The resulting  deformation has the following additional properties:
\begin{enumerate}
\item The triplet $F=(f,\fb,\ovla)$ verifies, for $s\le  s_{max}+1$,
\bea
\lab{eq:ThmGCMS1}
\|\ovla\|_{\hk_{s+1}(\S)}+\|(f,\fb)\|_{\hk_{s }(\S)} &\les &\dg. 
\eea
\item The  GCM constants  $\Cb^\S_0,\,  \CbpS$, \, $ M^\S_0$, \, $\MpS, \, p\in\{-,0, +\}$  verify
\bea
\lab{eq:ThmGCMS2}
\bsplit
\big| \Cb^\S_0-\Cb_0\big|+\big| \CbpS-\Cbp\big|&\les r^{-2}\dg,\\
\big| M^\S_0-M_0\big|+\big| \MpS-\Mp\big|&\les r^{-3}\dg.
\end{split}
\eea

\item The volume radius $r^\S$  verifies
\bea
\lab{eq:ThmGCMS3}
\left|\frac{r^\S}{\rg}-1\right|\les  r^{-1} \dg.
\eea
\item  The parameter  functions $U, S$  of the deformation verify, for $s\le s_{max}+1$,
\bea
\lab{eq:ThmGCMS4}
 \|( U, S)\|_{\hk_{s}(\ovS)}  &\les&  r \dg.
 \eea

\item The Hawking mass  $m^\S$  of $\S$ verifies the estimate
\bea
\lab{eq:ThmGCMS5}
 \big|m^\S-\ovm\big|&\les &\dg. 
 \eea
 \item The well defined\footnote{Note  that  while  the Ricci coefficients $\ka^\S, \kab^\S,  \chih^\S, \chibh^\S, \ze^\S$ as well as all curvature  components  and   mass aspect function $\mu^\S$     are well defined on $\S$, this in not the case  of $\eta^\S, \etab^\S, \xi^\S, \xib^\S, \om^\S, \omb^\S$ which require  the  derivatives of the frame in the $e_3^\S$ and $e_4^\S$ directions.}
  Ricci and curvature coefficients of $\S$  verify,
   \bea
   \lab{eq:ThmGCMS6}
\bsplit
\| \Ga^\S_g\|_{\hk_{s_{max} }(\S) }&\les  \epg  r^{-1},\\
\| \Ga^\S_b\|_{\hk_{s_{max} }(\S) }&\les  \epg.
\end{split}
\eea
\item  The   transition parameters $f, \fb, \ovla$ are continuously differentiable  with respect to $\La, \Lab $ and
\bea\lab{eq:property7oftheoldGCMtheoremnoncanonicalell=1modes}
\bsplit
\frac{\pr f }{\pr \La}&=O\big( r^{-1}\big), \quad  \frac{ \pr f }{\pr \Lab}=O\big(\dg r^{-1} \big), \quad 
\frac{\pr\fb }{\pr \La}&=O\big(\dg  r^{-1}\big), \quad  \frac{\pr\fb}{\pr \Lab}=O\big( r^{-1} \big),\\
\frac{\pr\ovla  }{\pr \La}&=O\big(\dg  r^{-1}\big), \quad \frac{\pr\ovla  }{\pr \Lab}=O\big(\dg  r^{-1}\big).
\end{split}
\eea

\item The parameter functions $(U, S)$ of the deformation are  are continuously differentiable  with respect to $\La, \Lab $
 and 
 \bea
 \frac{\pr U}{\pr \La}=O(1), \quad \frac{\pr U}{\pr \Lab}=O(1),\quad  \frac{\pr S}{\pr \La}=O(1) , \quad  \frac{\pr S}{\pr \Lab}=O(1). 
 \eea

\item Relative to the  coordinate system induced by $\Psi$ the metric $g^{\La, \Lab}$  of $\S=\S^{(\La, \Lab)}$  is continuous with respect to the parameters $\La, \Lab$ and verifies 
\bea\lab{eq:property9oftheoldGCMtheoremnoncanonicalell=1modes}
\big\| \pr_\La g^\S, \,  \pr_{\Lab} g^\S\|_{L^\infty(\S)} &\les O( r^2).
\eea
\item The Deformations  $\S^{(\La, \Lab)}$ also depend continuously on the   choice  the $\ell=1$ basis  of  $\ovS$ verifying  condition {\bf A4.}
\end{enumerate}
 \end{theorem}
Note that the $\ell=1$ modes on $\S$ given by \eqref{eq:non-canonicalJpS}   are not canonical on $\S$, i.e. they do not correspond to the  modes  defined using the  effective uniformization map for $\S$.  In the next  section we show that we can modify the result of Theorem \ref{Theorem:ExistenceGCMS1}  so that the same statements hold  when we replace  $\JpS$    by    canonical modes.   In the  proof of   this new   theorem,   we  make use of Theorem \ref{Theorem:ExistenceGCMS1}  for the following  class of $\ell=1$ modes $\Jp$ for  $\ovS$.
\begin{definition}
\lab{definition:J[S]modes}
Given an arbitrary $O(\dg)$  deformation of $\ovS$,  $\Psi: \ovS\to  S $ and $ J^S=\big\{J^{(p,S)}, \,\, p=-,0,+\big\} $  a canonical $\ell=1$ 
 basis  for $S$ we denote by 
 \beaa
 \ovJ[S]=\big\{\ovJ^{(p)}[S], \quad  p=-,0,+\big\}  
 \eeaa
  the $\ell=1$ basis of $\ovS$ given by 
 \bea
\ovJ^{(p)}[S]= J^{(p,S)}\circ\Psi, \qquad \text{ or  simply }\quad  \ovJ[S]= J^S\circ\Psi.
 \eea
\end{definition}
 Note  that $\ovJ[S]$ defined  above satisfies the  assumptions ${\bf A4}$.

    
    \subsection{Construction of GCM spheres with canonical $\ell=1$ modes}
    
    
    We state and prove  a new version of Theorem  \ref{Theorem:ExistenceGCMS1}  stated relative to  a  canonical
   definition of   ${\ell=1} $ modes of $\S$.  In other words we replace the  adapted $\JpS $ modes on $\S$  defined  by \eqref{eq:non-canonicalJpS}
    \beaa
    \JpS&=&\Jpov \circ \Psi^{-1} 
    \eeaa
    with  the following definition of      canonical modes   for deformed surfaces  $\S$.
    
    \begin{definition}
    \lab{definition:canmodesS}
    Given a  deformation map $\Psi:\ovS\to \S$   and a fixed effective uniformization map
       $(\ovPhi, \ovphi)$ for $\ovS$  we let $(\Phi, \phi) $ be the  unique effective  uniformization map of  $\S$ calibrated 
       with $(\ovPhi, \ovphi)$ relative to the map $\Psi$, in the sense of Definition \ref{def:calibration}. With this choice we define the canonical $\ell=1$  modes of $\S$  by the formula
       \bea
    \lab{def:JpSS}
    \JpSS&=&J^{(p, \SSS^2)} \circ \Phi^{-1} 
    \eea        
  with $J^{(p, \SSS^2)}$          the $\ell=1 $ basis on $\SSS^2$  given by Definition \ref{definition:ell=1mpdesonS}.
  \end{definition}

    \begin{remark}
    \lab{remark:calibration-of-deformations}
    Assume given  two  deformations $\Psi_i: \ovS\to \S_i$, $ i=1,2$  with  canonical effective  uniformization maps   $(\Phi_i, u_i) $,     $\Phi_i:\SSS^2 \to \S_i$     $ i=1,2$,   both  calibrated  through the respective maps  with the  fixed  effective uniformization map $(\ovPhi, \ovphi)$ for $\ovS$.  Then, in view of Lemma \ref{Le:Transitivitycalibrations},    $(\Phi_1,  u_1) $,  $(\Phi_2, u_2) $  are also calibrated  with respect to the    map $\Psi_{12}: \S_1 \to \S_2$, with $\Psi_{12}=\Psi_2\circ\Psi_1^{-1}$. 
    \end{remark}

     \begin{theorem}[GCM spheres with canonical $\ell=1$ modes]
\lab{Theorem:ExistenceGCMS1-2}
Under the same assumptions as in Theorem  \ref{Theorem:ExistenceGCMS1}, 
for any fixed triplets   $\La, \Lab \in \RRR^3$  verifying
\bea\lab{eq:assumptionsonLambdaabdLambdabforGCMexistence:bis}
|\La|,\,  |\Lab|  &\les & \dg,
\eea
 there 
exists a unique  GCM sphere $\S=\S^{(\La, \Lab)}$, which is a deformation of $\ovS$, 
such that  there exist constants $\Cb^\S_0$,\, \,  $\CbpS$, \, $ M^\S_0$, \, $\MpS, \, p\in\{-,0, +\}$  for which the   following GCM conditions   are verified
 \bea
\lab{def:GCMCversion2}
\bsplit
\ka^\S&=\frac{2}{r^\S},\\
\kab^\S &=-\frac{2}{r^\S}\Up^\S+  \Cb^\S_0+\sum_p \CbpS \J^{(p, \S)} ,\\
\mu^\S&= \frac{2m^\S}{(r^\S)^3} +   M^\S_0+\sum _p\MpS \J^{(p, \S)}.
\end{split}
\eea
 Moreover, 
 \bea
\lab{GCMS:l=1modesforffb.version2}
(\div^\S f)_{\ell=1}&=\La, \qquad   (\div^\S\fb)_{\ell=1}=\Lab.
\eea
where $\J^{(p, \S)}$ form the basis of canonical $\ell=1$ modes of  $\, \S$ in the sense of Definition \ref{definition:canmodesS}  and  $(\div^\S f)_{\ell=1}$, $(\div^\S \fb)_{\ell=1}$ are  defined  
 with respect to  them.  Moreover
 \begin{enumerate}
 \item The resulting  deformation  verifies all   the properties  1-9  as in Theorem  \ref{Theorem:ExistenceGCMS1}.  
\item The  canonical $\ell=1$ modes  $\J^{(p, \S)}$ verify,    
 \bea
  \lab{eq:PropertiesofJ^S-strong2}
 \bsplit
 \lap_S \J^{(p,\S)}  &=\left(-\frac{2}{(r^\S)^2} +O\left(\frac{\epg}{(r^\S)^3}\right) \right)\J^{(p,\S)}, \\
 \int_S \J^{(p,\S)}\J^{(q,\S)} da_g &= \frac{4\pi}{3}(r^\S)^2\de_{pq}+ O(\epg r^\S),\\
 \int_S \J^{(p,\S)} da_g&=0.
 \end{split}
 \eea
 \end{enumerate}
 \end{theorem}
 
 \begin{remark}
 Note that  \eqref{eq:PropertiesofJ^S-strong2} follows from the   fact that  $K^\S-\frac{1}{(r^{\S})^2}\in  r^{-1} \Ga_g$, i.e.
 \beaa
\left| K^\S-\frac{1}{(r^{\S})^2}\right| &\les& \frac{\epg}{(r^\S)^{3}}
 \eeaa
 and  the result of Lemma \ref{lemma:basicpropertiesofJforcanonicalell=1basis} 
 applied to the case when $\ep=r^{-1} \epg$.
 \end{remark}


  \subsection{Proof of Theorem \ref{Theorem:ExistenceGCMS1-2}}


We proceed by an iteration argument   in which, in order to make comparison between the
$\ell =1$ modes   of different deformation spheres,  we shall need the following  lemma.

\begin{lemma}
  \lab{lemma:comparisonJ2}
   Assume given two almost   round  spheres\footnote{Such that   \eqref{eq:Almostround-S-k} is verified.}  $(\S_1, g^{\S_1} )$ and $(\S_2, g^{\S_2} )$,  their respective canonical   uniformization maps 
     $(\Phi_1, \phi_1)$, $(\Phi_2, \phi_2)$, and    $\J^{(p, \S_1)}, \, \J^{(p, \S_2)}  $ be the  corresponding  canonical  basis  of $\ell=1$ modes of $\S_1$  and $\S_2$ calibrated  according to  Definition \ref{def:calibration}.
   Consider also a  given smooth  map\,  $ \Psi :\S_1\to \S_2$ and assume that     the two metric are comparable  relative to a   small parameter $\de_{12 }$, i.e. 
      \bea
   \lab{eq:unifor-twometricsB1}
  \| g^{\S_1}-  \Psi^\#( g^{\S_2} )\|_{\hk^{2}(\S^1)} \le \de_{12}  (r^{\S_1})^2.
  \eea
Then  the following estimate holds true.
     \bea
    \max_p  \sup_{\S_1}\Big|\J^{(p, \S_1)}-   \J^{(p, \S_2)} \circ   \Psi   \Big| \les (r^{\S_1}) ^{-1} \de_{12}.
     \eea   
\end{lemma}

\begin{proof}
This follows immediately from the proof of Proposition \ref{Prop:comparison.canJ} 
with   $\de = (r^{\S_1}) ^{-1}  \de_{12} $.
\end{proof}

We rephrase  the statement of the lemma above by introducing the following  notation.
\begin{definition}
Given two  neighboring, almost   round  spheres, $(\S_1, g^{\S_1} )$ and $(\S_2, g^{\S_2} )$
 and a map $ \Psi: \S_1\longrightarrow  \S_2$ we define  their distance relative to $\Psi$, measured in $\S_1$,   to be
\bea
\lab{distance:g_1g_2}
 d_{\S_1}( g^{\S_1}, g^{\S_2}) &:=& (r^{\S_1})^{-2}  \| g^{\S_1}-  \Psi^\#( g^{\S_2} )\|_{\hk^{2}(\S^1)}.
\eea
\end{definition} 

\begin{corollary}\lab{lemma:comparisonJ2:corollary}
Under the same situation as in Lemma \ref{lemma:comparisonJ2} we have,
\bea
\sup_{\S_1} \Big|\J^{\S_1}-   \J^{\S_2} \circ   \Psi   \Big| \les (r^{\S_1}) ^{-1} d_{\S_1}( g^{\S_1}, g^{\S_2}). 
\eea
\end{corollary}


\subsubsection{Iteration procedure} 


We define iteratively  a sequence of maps $\Psi(n) :\ovS\longrightarrow \S(n)$, for fixed $\La, \Lab$,  as follows.

\begin{enumerate}
\item We start   with the  deformation  map  $\Psi(1) :\ovS\longrightarrow \S(1)$  constructed  by Theorem    \ref{Theorem:ExistenceGCMS1}   with respect to   $\ovJ(1):=\ovJ_{can} $, the canonical $\ell=1$ basis on $\ovS$.  We define $J^1 $ to be  the adapted   $\ell=1$ basis on $S(1)$ i.e.
\beaa
J^1=J(1)=\ovJ_{can} \circ \Psi(1)^{-1}.
\eeaa
We also  denote by $\J^1_{can}=\J(1)_{can}  $  the canonical $\ell=1$ basis of $\S(1)$.

\item We   define $\Psi(2)  :\ovS\longrightarrow \S(2)  $ to be the deformation map 
constructed by Theorem  \ref{Theorem:ExistenceGCMS1} relative to the  $\ell=1$ basis on $\ovS$  given by 
\bea
\ovJ(2) := \J_{can}(1)\circ\Psi(1).
\eea
We note that $ \ovJ(2)$ verifies assumption {\bf A4} and thus  the assumptions of Theorem 
 \ref{Theorem:ExistenceGCMS1} apply.
The adapted   $\ell=1$ basis of $\S(2)$ is  given  by
\bea
J^2=J(2):=\ovJ(2)\circ \Psi(2)^{-1}= \J^1_{can}\circ\Psi(1) \circ \Psi(2)^{-1}.
\eea

\item Assume that  the sequence of  deformation maps  
$\Psi(n) :\ovS\longrightarrow \S(n)$ has already been constructed by Theorem \ref{Theorem:ExistenceGCMS1}  with the help of  the sequence $\ovJ(n) $ of $\ell =1$  bases 
on $\ovS$ given recursively  by
\bea
\ovJ(n)&=& \J_{can}^{n-1}  \circ \Psi(n-1)
\eea
 where  $\J_{can}^{n-1}=\J_{can}(n-1)$ is  the canonical    $\ell =1$ basis of $\S(n-1)$.
 The adapted sequence of $\ell =1$ bases of $S(n)$ is given by 
 \bea
 J^n=J(n) := \ovJ(n)\circ\Psi(n)^{-1}. 
 \eea

 We then   define   $\Psi(n+1) :\ovS\longrightarrow \S(n+1)$   to be the new  deformation map 
constructed by Theorem  \ref{Theorem:ExistenceGCMS1} relative to the  $\ell=1$ basis on $\ovS$  given by  
\bea
\ovJ(n+1)&=& \J^n_{can}  \circ \Psi(n),
\eea
where  $\J_{can} ^n=\J_{can}(n)$ is  the canonical    $\ell =1$ basis of $\S(n)$.

 We  define  the adapted   $\ell=1$ basis on $\S(n+1)$  by 
 \beaa
J^{n+1}= J(n+1)= \ovJ(n+1)\circ\Psi(n+1)^{-1} =  \J_{can}^n  \circ \Psi(n)\circ\Psi(n+1)^{-1}.
 \eeaa
 Introducing  the maps 
 \bea
 \bsplit
&  \Psi_{n+1, n}:\S(n+1) \longrightarrow \S(n), \qquad \Psi_{n+1, n}:= \Psi(n)\circ \Psi(n+1)^{-1}, \\
&\Psi_{n, n+1}:\S(n) \longrightarrow \S(n+1), \qquad \Psi_{n, n+1}:=\Psi(n+1)\circ \Psi(n) ^{-1},
\end{split}
 \eea
 we deduce the formula 
 \bea
 \lab{eq:identityJ^nJ^{n+1}}
 J^{n+1}&=  \J^n_{can}  \circ \Psi_{n+1, n}.
 \eea
\end{enumerate}


\subsubsection{Boundedness and contraction}  

    
 Consider the   quintets
   \beaa
     \QQ^{(n)}= \Big(U^{(n)}, S^{(n)}, \ovla^{(n)},  f^{(n)},  \fb^{(n)}\Big)=
      \Big(U^{(n)}, S^{(n)},  F^{(n)}\Big)
  \eeaa
    where $ U^{(n)}, S^{(n)} $ are the defining  functions of the map $\Psi(n)$ and $F^{(n)}$
   the corresponding transition function  defined in Theorem \ref{Theorem:ExistenceGCMS1}.   Consider also   the corresponding ninetets
   \beaa
\NN^{(n)}:=\Big(U^{(n)}, S^{(n)}, \ovla^{(n)},  f^{(n)},  \fb^{(n)}; \, \Cbn_0, \Mn_0, \Cbpn, \Mpn\Big).
\eeaa
We define the norms,
 \bea
\label{quintet-norm}
\bsplit
\| Q^{(n)}\|_k: &= r^{-1}   \Big\| \big(U^{(n)}, S^{(n)}\big)\Big \|_{\hk_{k}(\ovS)}  +\Big \|\big( f^{(n)},  \fb^{(n)}, \ovla^{(n)}\big) \Big\|_{\hk_{k}(\S)}
\end{split}
\eea
and,
\beaa
\big\| \NN^{(n)} \big\|_k= \| \QQ^{(n)}\|_k+ r^2 \left(\big| \Cbndot_0\big|+\sum_p\big| \Cbpndot\big| +r \big| \Mndot_0\big|+
r \sum_p\big| \Mpndot\big|\right)
\eeaa
where, 
\beaa
\Cbndot_0= \Cbn_0-\Cb_0, \,\,\,\, \Cbpndot= \Cbpn-\Cbp, \,\,\,\, \Mndot_0= \Mn_0-M_0, \,\,\,\, \Mpndot=\Mpn-\Mp.
\eeaa

    The equations  \eqref{def:GCMC} verified by each ninetet $\NN^{(n)}$ are given by
    \bea
\lab{def:GCMC-n}
\bsplit
\ka^{\S(n)}&=\frac{2}{r^{\S(n)}},\\
\kab^{\S(n)} &=-\frac{2\Up^{\S(n)}}{r^{\S(n)}}+  \Cb^{(n)}_0+\sum_p\Cb^{(n),p}  J^{n,p},\\
\mu^{\S(n)}&= \frac{2m^{\S(n)}}{(r^{\S(n)})^3} +   M^{(n)}_0+\sum_p M^{(n),p}  J^{n,p}.
\end{split}
\eea
Relative to the  modes $J^n $ we also have
 \bea
\lab{GCMS:l=1modesforffb-n}
\int_{\S(n)} (\div^{\S(n)} \fn)  J^n &=\La, \qquad \int_{\S(n)}   (\div^{\S(n)}\fbn) J^n =\Lab.
\eea
The maps $\Psi(n)=\big(\Un, \Sn \big) $  and the  triplet  $\Fn =\big( \fn, \fbn, \ovlan\big)$ verify  a  coupled system of equations which we  write schematically as follows\footnote{See the system (6.16)-(6.19) in \cite{KS-Kerr1} for a full description.} .
\bea
\bsplit
\lapzero \Sn&= \divzero\Big( \SS(\fn, \fbn, \Ga)_b Y^b_{(a)}(n) \Big)^{\#(n)}, \\
\lapzero \Un&= \divzero\Big( \UU(\fn, \fbn, \Ga)_b Y^b_{(a)}(n) \Big)^{\#(n)},\\
\DDn \Fn&=\GG^{(n)} + \HH(\Fn, \Ga).
\end{split}
\eea
Here $\#_n$  denotes the pull-back by  the map $\Psi(n)$ and $Y_{(a)}(n)$ the push-forward by $\Psi(n)$ of the coordinate vectorfields $\pr_{y^a} $, $a=1,2$. The expressions  
  $(\UU, \SS)$  
 are       linear   plus higher order  terms in $(f, \fb) $  and   $\HH$  is linear plus higher order in $F$. The  term\footnote{It also depends on  the  quantities  appearing in \eqref{eq:GCM-improved estimate1-again}, \eqref{eq:GCM-improved estimate2-again}.  } $ \GG^{(n)}$ is linear  with respect to the constants  $\Cbndot_0, \, \Cbpndot ,\, \Mndot_0, \,\,\Mpndot$.

  {\bf Step 1.}  \textit{Boundedness of the iterates.}  The  ninetets  $\NN^{(n)}$ are  uniformly bounded
  \bea
\lab{eq:nintetbound}
\big\| \NN^{(n)} \big\|_{s_{max}+1} &\les&\dg,
\eea
  where \eqref{eq:nintetbound}   is an immediate consequence of the estimates  \eqref{eq:ThmGCMS1}, \eqref{eq:ThmGCMS2} and \eqref{eq:ThmGCMS4}  in Theorem   \ref{Theorem:ExistenceGCMS1}.
  
 {\bf Step 2.}   \textit{Contraction of the iterates.}   Since we cannot  compare  directly
 the ninetets  $\NN^{(n)}$   we compare instead  their pull-back by $\Psi(n)$ to $\ovS$ i.e.
 \bea
 \NN^{n,\#}:=\Big(U^{(n)}, S^{(n)}, \ovla^{n,\#},  f^{n,\#},  \fb^{n,\#}; \, \Cbn_0, \Mn_0, \Cbpn, \Mpn\Big)
 \eea
 where $\ovla^{n,\#},  f^{n,\#},  \fb^{n,\#}$ are the pull-backs by the map $\Psi(n)$   of  the triplet $\ovla^{(n)}, f^{(n)}, \fb^{(n)} $
   defined on the sphere  $\S(n)$.
We also  introduce  the modified norms
\bea
\bsplit
\big\| \NN^{n,\#} \big\|_{k, \ovS} :&= r^{-1}   \Big\| \big(U^{(n)}, S^{(n)}\big)\Big \|_{\hk_{k}(\ovS)} +
\Big\|\big( f^{n,\#},  \fb^{n,\#}, \ovla^{n,\#}\big)\Big\| _{\hk_{k}(\ovS)}\\
&+ r^{2} \left(\big| \Cbndot_0\big|+\sum_p\big| \Cbpndot\big| +r \big| \Mndot_0\big|+
r \sum_p\big| \Mpndot\big|\right).
\end{split}
\eea
In view of the  Sobolev norm  comparison  of Proposition 5.9 in \cite{KS-Kerr1}  we  deduce from
\eqref{eq:nintetbound}
 \bea
 \lab{eq:nintetboundmodified}
\big\| \NN^{n,\#} \big\|_{s_{max}+1, \ovS} &\les&\dg.
\eea
Contraction in this modified norms is established in the following.
\begin{proposition}
\lab{Prop:contractionforNN}
The following estimate holds true.
\bea
\|\NN^{n+1,\#}-\NN^{n,\#}\|_{3,\ovS} &\les& \epg \|\NN^{n,\#}-\NN^{n-1,\#}\|_{3,\ovS}.
\eea
\end{proposition}

The proof  of Proposition  \ref{Prop:contractionforNN}  is very similar  to that of Proposition 6.5  in  \cite{KS-Kerr1}.  The  difference is in  the way  our       basis of adapted $\ell=1$ modes $J^n$  were defined  above, which in turn requires to compare $J^{n+1}$ and $J^n$, see \eqref{eq:intermadiaryestimateforproofcontractionschemeintrinsicGCM} below.  To deal with this issue, we rely on  the following lemma.
 \begin{lemma}\lab{lemma:comparisionofJnplus1andJnforcontractionscheme}
 We have,
 \beaa
\Big|  J^{n+1} \circ \Psi_{n,n+1} - J^n\Big|&\les  r^{-1} d_n \big( \gn, g^{(n-1)}  \big)
 \eeaa
 with $\gn$ the metric of $\S(n)$. 
 \end{lemma}
 
 \begin{proof}
 We write,
 \beaa
  J^{n+1} \circ \Psi_{n, n+1} - J^n&=& \J_{can}^n   \circ \Psi_{n+1, n} \circ \Psi_{n, n+1} - J^n= \J_{can}(n)-J(n)\\
  &=& \J_{can}^n   -  \J_{can}^{n-1}   \circ \Psi_{n, n-1}. 
 \eeaa
 Thus, in view of Corollary  \ref{lemma:comparisonJ2:corollary}  applied to  the pairs  $\big(S(n), \gn\big)$,
 $\big(S(n-1), g^{(n-1)} \big)$ and map $\Psi=\Psi_{n, n-1}$,
 \beaa
 \Big|  J^{n+1} \circ \Psi_{n, n+1} - J^n\Big| =  \Big| \J_{can}^n  -  \J_{can}^{n-1}  \circ \Psi_{n, n-1} \Big| &\les&   r^{-1} d_n \big( \gn, g^{(n-1)}  \big)
 \eeaa
 as stated.
 \end{proof}
 
\begin{proof}[Proof of  Proposition  \ref{Prop:contractionforNN}]
Proposition 6.5 in   \cite{KS-Kerr1} yields
\bea\lab{eq:intermadiaryestimateforproofcontractionschemeintrinsicGCM}
\|\NN^{n+1,\#}-\NN^{n,\#}\|_{3,\ovS} &\les& \epg \|\NN^{n,\#}-\NN^{n-1,\#}\|_{3,\ovS} +\dg\Big|  J^{n+1} \circ \Psi_{n, n+1} - J^n\Big|
\eea
where the last term in the RHS comes from comparing the $\ell=1$ modes of $\fnn$, $\fbnn$ with the ones of $\fn$, $\fbn$ using \eqref{GCMS:l=1modesforffb-n}. Using Lemma \ref{lemma:comparisionofJnplus1andJnforcontractionscheme}, we infer
\beaa
\|\NN^{n+1,\#}-\NN^{n,\#}\|_{3,\ovS} &\les& \epg \|\NN^{n,\#}-\NN^{n-1,\#}\|_{3,\ovS} +\dg d_n \big( \gn, g^{(n-1)}  \big).
\eeaa

Making use of the expressions of the metrics $\gn$, $g^{(n-1)}$ with respect to  the maps $\Psi(n)$,  $\Psi(n-1)$, see  Lemma 5.2 in \cite{KS-Kerr1},  and using the definition of $\Psi(n)$,  $\Psi(n-1)$ respectively in terms of the scalar functions $(U^{(n)}, S^{(n)})$ and $(U^{(n-1)}, S^{(n-1)})$, see Definition \ref{definition:Deformations}, we can  then   estimate
$d_n \big( \gn, g^{(n-1)}  \big)$, similarly to the Proposition 6.5  in  \cite{KS-Kerr1},  as follows
\beaa
d_n \big( \gn, g^{(n-1)}  \big) &\les& r^{-1}   \Big\| \big(U^{(n)}-U^{(n-1)}, S^{(n)}- S^{(n-1)}\big)\Big \|_{\hk_{3}(\ovS)}.
\eeaa
We deduce, using also $\dg\leq\epg$ and the definition of $\NN^{n,\#}$, 
\beaa
\|\NN^{n+1,\#}-\NN^{n,\#}\|_{3,\ovS} &\les& \epg \|\NN^{n,\#}-\NN^{n-1,\#}\|_{3,\ovS} +\dg d_n \big( \gn, g^{(n-1)}  \big)\\
&\les& \epg \Bigg[\|\NN^{n,\#}-\NN^{n-1,\#}\|_{3,\ovS}\\
&&+r^{-1}   \Big\| \big(U^{(n)}-U^{(n-1)}, S^{(n)}- S^{(n-1)}\big)\Big \|_{\hk_{3}(\ovS)}\Bigg]\\
&\les& \epg \|\NN^{n,\#}-\NN^{n-1,\#}\|_{3,\ovS}
\eeaa
as desired.
\end{proof}

   
     \subsubsection{Convergence}
     
     
     In view of \eqref{eq:nintetbound} and Proposition \ref{Prop:contractionforNN}, it follows that  the sequence    $\Psi(n)$ converges to a deformation  map 
$\Psi:\ovS \longrightarrow \S$, 
 $\S(n)\to \S$,  $\Psi_{n+1, n}\to I_\S$,  $  \J^n_{can} \to \J^\S_{can} $ and $J^n\to J$.  In view of  \eqref{eq:identityJ^nJ^{n+1}}, 
\beaa J^{n+1}&=  \J^n_{can}  \circ \Psi_{n+1, n}.
 \eeaa
 We deduce $J =\J^\S_{can}$. 
 The limiting map $\Psi$  thus  verifies  the desired properties of Theorem 
\ref{Theorem:ExistenceGCMS1-2} with respect  to   a canonical $\ell=1$ basis  of $\S$ as desired. This concludes the proof of Theorem \ref{Theorem:ExistenceGCMS1-2}.

    
\section{Construction of intrinsic GCM spheres}


In  what follows we prove an important corollary of  Theorem \ref{theorem:ExistenceGCMS2}  which  makes use of the arbitrariness of $\La, \Lab$ to  ensure the vanishing of the $\ell=1$ modes of  $\b$ and  $\trchbc$.  The result requires  stronger  assumptions  than  those   made  in {\bf A1},  see \eqref{eq:assumtioninRRforGagandGabofbackgroundfoliation}, namely we assume  that  $\Ga_b$  has the same behavior as  $\Ga_g$, i.e. 

{\bf A1-Strong.}  For $k\le s_{max}$,
\bea
\lab{eq:assumtioninRRforGagandGabofbackgroundfoliation'}
\Big\|(\Ga_g, \Ga_b) \Big\|_{k, \infty}  \les \epg r^{-2}.
\eea

\begin{remark}\lab{rmk:remarkonimprovedrdecaybytradingudecay}
In applications to  our envisioned  proof of  the nonlinear  stability of Kerr,   we  plan to   apply   the result below  in   spacetime regions where  we have  $r\sim u$.    Such  regions 
 appear naturally  when we extend a  given  GCM admissible   spacetime $\MM$, see  Theorem  M7 in \cite{KS-Schw}, in a neighborhood  of  its top     GCM sphere. The improvement in $r$  for $\Ga_b$  is a consequence of the fact that we can exchange, in such regions,  the  decay factors in $u$ for decay in $r$. 
 \end{remark}

We also need to make stronger assumptions concerning  the   basis of      $\ell=1$  modes for the background foliation. These assumptions are also consistent with the     what we expect in applications to   the    nonlinear stability of Kerr.

{\bf  A4-Strong.}  We assume  the existence of a   smooth family of  $\ell=1$ modes  $\Jp:\RR\to\RRR$ verifying  the    verifying the following properties.

\begin{enumerate}
\item  On  each sphere    $S$ of the background foliation $\RR$, $\Jp$ is  $O(r^{-1}\dg)$ close  to a canonical 
basis of $\ell=1$ modes $\Jp_{can}$, i.e.
\bea
\Big| \Jp - \Jp_{can}\Big| \les \dg r^{-1}.
\eea

\item  The $\Jp$ family verifies 
 \bea
 \lab{eq:Jpsphericalharmonics-strong}
\bsplit
  \Big(r^2\lap+2\Big) \Jp  &= O(\epg r^{-1} ),\qquad p=0,+,-,\\
\frac{1}{|S|} \int_{S}  \Jp J^{(q)} &=  \frac{1}{3}\de_{pq} +O(\epg r^{-1} ),\qquad p,q=0,+,-,\\
\frac{1}{|S|}  \int_{S}   \Jp   &=O(\epg),\qquad p=0,+,-.
\end{split}
\eea

\item  The  $\Jp$ family also verifies
\bea
\lab{assumptionA4-strong}
\max_p\sup_{\RR} \Big(|e_3(\Jp)|+|e_4(\Jp)|+|\nab \Jp|\Big) \les \epg r^{-1}.
\eea
\end{enumerate}

\begin{lemma}
\lab{Lemma:ComparisonJ-strong}
Under the assumptions {\bf A1-Strong}, {\bf A2, A3} and {\bf A4-Strong},  
consider  a  $O(\dg)$-deformation  $\Psi: \ovS\longrightarrow \S$, with $(U,S)$ as in Definition \ref{definition:Deformations},  
verifying the assumption, see \eqref{eq:ThmGCMS4}, 
 \bea
 \lab{eq::ComparisonJ-strong1}
 \| (U, S)\|_{L^\infty(\ovS)} +  r^{-1} \big\|(U, S)\big\|_{\hk_{s_{max}+1}(\ovS)}  &\les  \dg.
 \eea
 Assume also  that the canonical  $\ell=1$ modes  of $\ovS$ and $\S$ are calibrated.   
Then:
 \begin{enumerate}
\item The following comparison estimate  between    $\JJpS$ and $\JpS$ holds true,
\bea
\max_p\sup_\S \Big|\JJpS-\JpS\Big| \les \dg (r^\S)^{-1}. 
\eea

\item Also
\bea
\max_p\sup_\S\Big| \JpS-\Jp \Big| \les \epg\dg (r^\S)^{-1}.
\eea
\end{enumerate}
\end{lemma} 

\begin{proof}
Recall that, see \eqref{eq:non-canonicalJpS},  $ \JpS=\Jpov \circ \Psi^{-1}$.
We write,
\beaa
\JJpS-\JpS&=&  \JJpS-  \Jpov \circ \Psi^{-1} = \JJpS-  \Jpov_{can}  \circ \Psi^{-1} +\big( \Jpov_{can}   -\Jpov\big) \circ \Psi^{-1}.
\eeaa
Hence, in view of  the assumption {\bf A4-strong} on $\ovS$, 
\beaa
\Big|\JJpS-\JpS\Big|&\les& \Big| \JJpS-  \Jpov_{can}  \circ \Psi^{-1} \Big|+ \Big| \big( \Jpov_{can}   -\Jpov\big) \circ \Psi^{-1}\Big|\\
&\les&  \Big| \JJpS-  \Jpov_{can}  \circ \Psi^{-1} \Big|+\dg r^{-1}.
\eeaa
It remains to estimate  $| \JJpS-  \Jpov_{can}  \circ \Psi^{-1}|$.  Making use  of  the assumption \eqref{eq::ComparisonJ-strong1} and  Lemma  \ref{lemma:comparison-gaS-ga}  we have
\beaa
 \big\| \ovg- \Psi^\# (g^\S)\big\|_{L^\infty(\ovS)} +( \ovr)^{-1}  \big\| \ovg- \Psi^\# (g^\S)\big\|_{\hk^2(\ovS)}\les \dg  \ovr. 
\eeaa
We are thus in a position to apply  the comparison Lemma \ref{lemma:comparisonJ2} to deduce,
\beaa
\max_p \sup_{\ovS} \Big|\JJpS\circ \Psi- \Jpov_{can}\Big|&\les \dg( \ovr)^{-1}
\eeaa
or, equivalently,
\beaa
 \max_p \sup_{\S}\Big| \JJpS-  \Jpov_{can}  \circ \Psi^{-1} \Big|&\les&\dg  (r^\S)^{-1}. 
\eeaa
Hence, 
\beaa
\Big|\JJpS-\JpS\Big| &\les&  \Big| \JJpS-  \Jpov_{can}  \circ \Psi^{-1} \Big|+\dg r^{-1}\les\dg  (r^\S)^{-1}
\eeaa
as stated in the first part of the lemma.

To prove the second part, we appeal to  Lemma \ref{Le:Transportcomparison}
applied to  the scalar function $F=\Jp$ which yields 
\beaa
\|\Jp\circ\Psi -\Jp\|_{L^\infty(\ovS)} &\les& \big \| ( U, S) \big\|_{L^\infty(\ovS)}  \sup_{\RR}\Big(|e_3(\Jp)|+|e_4(\Jp)|+|\nab \Jp|\Big)\\
&\les& r^{-1}\epg\dg 
\eeaa
where we used \eqref{eq::ComparisonJ-strong1} and  \eqref{assumptionA4-strong} in the last inequality. Recalling that $r^\S$, $\ovr$ and   $r|_\S$ are comparable,  we deduce
$$
\max_p\sup_\S\Big| \JpS-\Jp \Big| = \max_p\sup_\S\Big| \Jp\circ\Psi^{-1}-\Jp \Big|=\max_p\sup_{\ovS}\Big| \Jp\circ\Psi -\Jp \Big|\les (r^\S)^{-1}\epg\dg
$$
as stated.
\end{proof}

We are ready to state and prove  our intrinsic GCM Theorem
\ref{theorem:ExistenceGCMS2}. 
\begin{theorem}[Existence of intrinsic GCM spheres]
\lab{theorem:ExistenceGCMS2} 
Assume that the spacetime region  $\RR$ verifies the assumptions {\bf A1-Strong, A2, A3}, {\bf A4-Strong}.
We further  assume that, relative to the   $\ell=1$ modes of  the background foliation,
\bea
\lab{Assumptions:theorem-ExistenceGCMS2}
(\div \b )_{\ell=1}=O( \dg r^{-3}), \qquad        (\widecheck{\trch})_{\ell=1}=O(\dg r^{-1}),   \qquad (\widecheck{\trchb})_{\ell=1}=O(\dg r^{-1}).
\eea
Then,  there  exist unique constants  $ M^\S_0, \, \MpS$, \,  $p\in\{-,0, +\}$  such that
 \bea
\lab{def:GCMC2}
\bsplit
\ka^\S&=\frac{2}{r^\S},\\
\kab^\S &=-\frac{2}{r^\S}\Up^\S,\\
\mu^\S&= \frac{2m^\S}{(r^\S)^3} +   M^\S_0+\sum _p\MpS \J^{(p, \S)}, 
\end{split}
\eea
and
\bea
\lab{def:GCMC2-b}
\int_\S \div^\S \b^\S  \J^{(p, \S)}  =0,
\eea
where $ \J^{(p, \S)}  $ is a  canonical  $\ell=1$ basis for $\S$ calibrated,  relative  by $\Psi $, with the canonical    $\ell=1$ basis  of $\ovS$.  Moreover the deformation  verifies the properties  \eqref{eq:ThmGCMS3}, \eqref{eq:ThmGCMS4},  \eqref{eq:ThmGCMS5}, \eqref{eq:ThmGCMS6} stated in Theorem \ref{Theorem:ExistenceGCMS1}.
\end{theorem}

\begin{remark}
\lab{Remark:theorem-ExistenceGCMS2}
The assumptions \eqref{Assumptions:theorem-ExistenceGCMS2}   for $(\div \b)_{\ell=1}$ and $(\widecheck{\trch})_{\ell=1}$ holds true in general under the assumptions {\bf A1-A3}.  The corresponding assumption  for  $(\trchbc)_{\ell=1}$
 holds true  in regions, discussed in Remark \ref{rmk:remarkonimprovedrdecaybytradingudecay}, where $r\sim u$, see  the proof  of Theorem M4 in \cite{KS-Schw}.
\end{remark}

The proof of Theorem \ref{theorem:ExistenceGCMS2}  relies  on  the following lemma.
\begin{lemma}
\lab{Lemma:GCMS2}
Assume that the spacetime region  $\RR$ verifies the assumptions {\bf A1-Strong, A2, A3}, {\bf A4-Strong}. Let $\S=\S^{(\La, \Lab)} $ be a deformation of $\ovS$ as constructed by Theorem \ref{Theorem:ExistenceGCMS1-2} with
\beaa
\int_\S  \div^\S f\, \J^{(p, \S)}=\La, \qquad \int_\S  \div^\S \fb\, \J^{(p, \S)}=\Lab.
\eeaa
The following identities hold true.
\bea
\lab{eq:ExistenceGCMS2.1}
\bsplit
\La&=\frac{(r^\S)^3}{3 m^\S}\Big[- (\div^\S \b^\S)_{\ell=1}     + (\div \b)_{\ell=1} \Big]        + F_1(\La, \Lab),\\
\Lab &= \Up^\S \La +\frac{r^\S}{ 3m^\S} \Big[( \kabc^\S)_{\ell=1} +   \Up^\S ( \kac)_{\ell=1} -  (\kabc)_{\ell=1} \Big]+ F_2(\La, \Lab),
\end{split}
\eea
where $F_1, F_2$ are  continuously differentiable   functions of $\La, \Lab$  verifying
\bea
\big|F_1, F_2\big| \les \epg \dg, \qquad \big|\pr_\La (F_1, F_2) ,  \pr_{\Lab} (F_1, F_2)\big|\les \epg.  
\eea
\end{lemma} 

We postpone the proof of Lemma \ref{Lemma:GCMS2} to section \ref{sec:proofofLemma:GCMS2}.  

\begin{proof}[Proof of  Theorem \ref{theorem:ExistenceGCMS2}]
We note that under the assumptions of the theorem the system
\beaa
\bsplit
\La&=\frac{(r^\S)^3}{3 m^\S} (\div \b)_{\ell=1}         + F_1(\La, \Lab),\\
\Lab &= \Up^\S \La +\frac{r^\S }{3 m^\S} \Big(  \Up^\S ( \kac)_{\ell=1} -  (\kabc)_{\ell=1} \Big)+ F_2(\La, \Lab),
\end{split}
\eeaa
has a unique solution  $\La_0, \Lab_0$    verifying the estimate 
\beaa
\big|\La_0\big|+\big| \Lab_0\big| \les \dg. 
\eeaa
Therefore, taking $\La=\La_0, \Lab=\Lab_0$ in \eqref{eq:ExistenceGCMS2.1},
 we deduce
\beaa
(\div^\S \b^\S)_{\ell=1}=0, \qquad  (\kabcS)_{\ell=1}=0.
\eeaa

 It remains to check \eqref{def:GCMC2}. According to Theorem 
 \ref{Theorem:ExistenceGCMS1-2}, 
  there exist constants $\Cb^\S_0$,\, \,  $\CbpS$, \, $ M^\S_0$, \, $\MpS, \, p\in\{-,0, +\}$  such that
 \beaa
\bsplit
\ka^\S&=\frac{2}{r^\S},\\
\kab^\S &=-\frac{2}{r^\S}\Up^\S+  \Cb^\S_0+\sum_p \CbpS \J^{(p, \S)} ,\\
\mu^\S&= \frac{2m^\S}{(r^\S)^3} +   M^\S_0+\sum _p\MpS \J^{(p, \S)}.
\end{split}
\eeaa
Using the two first equations above, the definition of the Hawking mass,  and using the fact that $\int_\S  \J^{(p, \S)}=0$,   we infer
 \beaa
\Up^\S&=& 1 -\frac{2m^\S}{r^\S} =-\frac{1}{16 \pi} \int_{\S } \ka^\S \kab^\S=-\frac{1}{8 \pi r^\S} \int_{\S }\left(-\frac{2}{r^\S}\Up^\S+  \Cb^\S_0+\sum_p \CbpS \J^{(p, \S)}\right)\\
&=& \Up^\S -\frac{r^\S}{2}\Cb^\S_0
\eeaa
and hence,
\beaa
  \Cb^\S_0&=& 0.
\eeaa
Therefore,
\beaa
\kabc^\S&=& \sum_p \CbpS \J^{(p, \S)}.
\eeaa
Projecting  on the  canonical $\ell=1$ basis   and using the condition $(\kabc^\S)_{\ell=1}=0$ 
and the property $ \int_S \J^{(p,\S)}\J^{(q,\S)} da_g = \frac{4\pi}{3}(r^\S)^2\de_{pq}+ O(\epg r^\S)$ in \eqref{eq:PropertiesofJ^S-strong2}
  we then conclude that  $ \CbpS=0$. Hence, we have finally obtained
   \beaa
\bsplit
\ka^\S&=\frac{2}{r^\S},\\
\kab^\S &=-\frac{2}{r^\S}\Up^\S,\\
\mu^\S&= \frac{2m^\S}{(r^\S)^3} +   M^\S_0+\sum _p\MpS \J^{(p, \S)}, 
\end{split}
\eeaa
and
\beaa
\int_\S \div^\S \b^\S  \J^{(p, \S)}  =0
\eeaa
   as desired.
\end{proof}


\subsection{Proof of Lemma \ref{Lemma:GCMS2}}
\lab{sec:proofofLemma:GCMS2}


Note first that we are in position to apply the comparison  Lemma  \ref{Lemma:ComparisonJ-strong} according to which $\max_p\sup_\S \Big|\JJpS-\JpS\Big| \les \dg (r^\S)^{-1}$ and $\max_p\sup_\S\Big| \JpS-\Jp \Big| \les \epg\dg (r^\S)^{-1}$, 
and therefore
\bea
\lab{eq:comparsonsJs}
\max_p\sup_\S \Big|\JJpS-\Jp\Big| \les \dg(r^\S)^{-1}.
\eea

Let  $(f, \fb, \la)$ denote the transition coefficients between the background  frame of $\RR$  and 
 the frame $(e_3^\S, e_4^\S, e_1^\S, e_2^\S)$ adapted to the deformation $\S$. 
 
{\bf Step 1.} Consider a GCM sphere $\S^{(\La, \Lab)}$ as in Theorem \ref{Theorem:ExistenceGCMS1-2}.
Since   $\rhod\in r^{-1}\Ga_g$ we rewrite   the transformation    \eqref{eq:transfforbandkab2}         for $\b$ in the form
\beaa
   \b^\S&=&\b +\frac 3 2  f \rho +  r^{-1} \Ga_g \c F.
\eeaa 

\begin{remark}
Note that  all quantities   defined  on $\S$, such as $r^\S$, $m^\S$, $(f, \fb, \ovla)$  and $\J^{(p, \S)}$   depend  in fact  continuously on the parameters $\La, \Lab$ even though this dependence  is  not  made explicit.
\end{remark}

Since  $|m^\S-m|\les \dg$,  $|r^\S-r|\les \dg$, and
  \beaa
   \rho&=&-\frac{2m}{r^3} +\rhoc= -\frac{2m^\S}{(r^\S)^3} +\rhoc + \left(   \frac{2m^\S}{(r^\S)^3}   -   \frac{2m}{r^3}\right) \\
         &=& -\frac{2m^\S}{(r^\S)^3}+ r^{-1}\Ga_g + O(r^{-3} \dg),
   \eeaa
we deduce
\beaa
\b^\S+\frac{3m^\S}{(r^\S)^3}  f &=&\b +      r^{-1}\Ga_g\c F.
\eeaa
Taking the divergence,         we  infer
\beaa
\div^\S \b^\S+\frac{3m^\S}{(r^\S)^3} \div^\S  f         &=&\div^\S \b+  \div^\S \big(r^{-1}\Ga_g\c F\big)\\
&=&\div \b + r^{-2} (\dk \Ga_g) \c F +  \div^\S \big(r^{-1}\Ga_g\c F\big)\\
&=&\div\b+E_1(\La, \Lab),
\eeaa
with error term of the form
\bea
E_1(\La, \Lab)&:=& r^{-2} (\dk \Ga_g) \c F +  \div^\S \big(r^{-1}\Ga_g\c F\big).
\eea
We deduce,
\bea
\lab{identity:badmode-b}
\frac{3m^\S}{(r^\S)^3}\int_\S ( \div^\S  f)  \J^{(p, \S)}&=&-\int_\S (\div^\S \b^\S) \J^{(p, \S)}+ \int_\S (\div  \b ) \J^{(p, \S)}
+\int_\S E_1\J^{(p, \S)}.
\eea
In view of   \eqref{eq:comparsonsJs} we have $\big|\J^{(p, \S)} - \Jp\big| \les   \dg r^{-1}$. 
Therefore, since  $\b\in r^{-1}\Ga_g$,     
\beaa
 \int_\S (\div  \b ) \J^{(p, \S)}&=&  \int_\S (\div  \b ) \Jp +   \int_\S (\div  \b )\big(\J^{(p, \S)}- \Jp\big)\\
 &=& \int_\S (\div  \b ) \Jp  +O( r^{-3} \epg\dg)\\
 &=& \int_{\ovS } (\div  \b ) \Jp+\left( \int_\S (\div  \b ) \Jp - \int_{\ovS} (\div  \b ) \Jp\right) +O( r^{-3} \epg\dg).
\eeaa
To estimate the term
$\int_\S (\div  \b ) \Jp - \int_{\ovS} (\div  \b ) \Jp$ we   appeal to Lemma \ref {Lemma:coparison-forintegrals} applied to $F=(\div  \b ) \Jp$. Thus          
\beaa
\left|\int_\S F -\int_{\ovS} F\right| &\les& \dg   r   \sup_{\RR} \Big( |F| +  r \big(|e_3 F|+ |e_4 F|+|\nab F| \big)\Big).
\eeaa
Note that in view of the fact that $\div \b = O(\epg r^{-4} )$,  $e_3(\div \b) = O(\epg r^{-5})$ and $e_3\Jp =O( \epg r^{-1}) $, we have
\beaa
\big|e_3 F\big|&\les &  \big|e_3(\div \b) \big|+ \big|\div \b\big| \big|e_3 \Jp \big|\les \epg  r^{-5}. 
\eeaa
The other terms are treated similarly, and  we infer that
\beaa
\left|\int_\S F -\int_{\ovS} F\right| &\les& \epg \dg r^{-3}.
\eeaa
We deduce
\beaa
 \int_\S (\div  \b ) \J^{(p, \S)}&=& \int_{\ovS} (\div  \b ) \Jp  +O( \epg\dg r^{-3} )=(\div \b)_{\ell=1} +O( \epg\dg r^{-3} ).
 \eeaa
Since $\int_\S  ( \div^\S  f)  \J^{(p, \S)}=\La$,  after gathering  all error terms in $F_1=F_1(\La, \Lab)$, we deduce from \ref{identity:badmode-b}
\bea
\La&=& \frac{(r^\S)^3}{3 m^\S}\Big[- (\div^\S \b^\S)_{\ell=1}     + (\div \b)_{\ell=1}  \Big]       + F_1(\La, \Lab)
\eea
where 
\beaa
F_1(\La, \Lab)&=& \frac{(r^\S)^3}{3 m^\S}  \int_{\S^{\La, \Lab} } E_1(\La, \Lab)\JpS+\frac{(r^\S)^3}{3 m^\S}\left(\int_\S (\div  \b ) \J^{(p, \S)} -(\div\b)_{\ell=1} \right).
\eeaa
In view of the above, we  easily check that 
\beaa
\Big| F_1(\La, \Lab)\Big|\les   \epg \dg 
\eeaa
as stated.

{\bf Step 2.} We next consider the  equations \eqref{eq:transfforbandkab1}
 \beaa
   \bsplit
\ka^\S &= \ka+ \ka \ovla+\div^\S f   +F\c  \Ga_b+ F\c\nab^\S F+r^{-1} F^2,\\
\kab^\S &= \kab- \kab \ovla+\div^\S\fb   +F\c  \Ga_b+ F\c\nab^\S F+r^{-1} F^2.
\end{split}
\eeaa
     Since  $| F| \les r^{-1} \dg $, and since we have $\big|\Ga_b\big| \les  r^{-2}\epg$  in view of {\bf A1-strong}, and using the GCM condition $\ka^\S=2/r^\S$, we  deduce  
 \beaa
   \bsplit
\div^\S f+ \ka \ovla &= \frac{2}{r^\S} - \ka    +O\big( \epg \dg r^{-3} \big),\\
\div^\S\fb - \kab \ovla &= \kab^\S - \kab    +O\big( \epg \dg r^{-3} \big).
\end{split}
\eeaa     
Differentiating w.r.t. $\Delta^\S$, and  using {\bf A1-strong} for $\ka$ and $\kab$, we infer
 \beaa
   \bsplit
\Delta^\S\div^\S f+ \ka \Delta^\S\ovla &=  - \Delta^\S\ka    +O\big( \epg \dg r^{-5} \big),\\
\Delta^\S\div^\S\fb - \kab \Delta^\S\ovla &= \Delta^\S\kab^\S - \Delta^\S\kab    +O\big( \epg \dg r^{-5} \big).
\end{split}
\eeaa 
This yields
\beaa
\kab\Delta^\S\div^\S f+\ka\Delta^\S\div^\S\fb &=& \ka\Delta^\S\kab^\S -\kab\Delta^\S\ka -\ka\Delta^\S\kab +O\big( \epg \dg r^{-6} \big).
\eeaa
Hence, using the control of $\ka$ and $\kab$ provided by {\bf A1-Strong}, $|m-m^\S|\les\dg$, and $|r-r^\S|\les \dg$, we deduce
\beaa
-\frac{2\Up^\S}{r^\S}\Delta^\S\div^\S f+\frac{2}{r^\S}\Delta^\S\div^\S\fb &=& \frac{2}{r^\S}\Delta^\S\kab^\S +\frac{2\Up^\S}{r^\S}\Delta^\S\ka -\frac{2}{r^\S}\Delta^\S\kab +O\big( \epg \dg r^{-6} \big),
\eeaa
or,
\beaa
\Delta^\S\big(\div^\S\fb -\Up^\S\div^\S f\big) &=& \Delta^\S\kab^\S +\Up^\S\Delta^\S\ka -\Delta^\S\kab +O\big( \epg \dg r^{-5} \big).
\eeaa

Next, we focus on $\Delta^\S\ka$ and $\Delta^\S\kab$. Recall from \eqref{eq:Generalframetransf} that 
\beaa
 e_a^\S &= \left(\de_{ab} +\frac{1}{2}\fb_af_b\right) e_b +\frac 1 2  \fb_a  e_4 +\left(\frac 1 2 f_a +\frac{1}{8}|f|^2\fb_a\right)   e_3,\qquad a=1,2.
\eeaa
Recalling the definition of $\Ga_g$ and $\Ga_b$ in \eqref{definition:Ga_gGa_b}, we have $\ka-\frac 2  r =\kac  \in \Ga_g$, $\kab+\frac{2\Up}{r}=\kabc  \in \Ga_g$ and 
\beaa
 e_4(r)-1 \in r\Ga_g,\quad   e_3(r)+\Up \in r\Ga_b, \quad 
e_4(m)\in  r\Ga_g, \quad e_3(m)\in r\Ga_b.
\eeaa
Thus  using   assumption  {\bf A1-Strong}  and  $| F| \les r^{-1} \dg $, we infer that
\beaa
\nab^\S\ka &=& \nab\ka+ \frac{\Up}{r^2}f -\frac{1}{r^2}\fb +O(r^{-4}\epg\de),\\
\nab^\S\kab &=& \nab\kab - \frac{\Up\left(1-\frac{4m}{r}\right)}{r^2}f +\frac{1-\frac{4m}{r}}{r^2}\fb +O(r^{-4}\epg\dg).
\eeaa
Taking the divergence, and using {\bf A1-strong} for $\ka$ and $\kab$,  $|m-m^\S|\les\dg$ and $|r-r^\S|\les \dg$, we deduce
\beaa
\Delta^\S\ka &=& \Delta\ka+ \frac{\Up^\S}{(r^\S)^2}\div^\S f -\frac{1}{(r^\S)^2}\div^\S\fb +O(r^{-5}\epg\dg),\\
\Delta^\S\kab &=& \Delta\kab - \frac{\Up^\S\left(1-\frac{4m^\S}{r^\S}\right)}{(r^\S)^2}\div^\S f +\frac{1-\frac{4m^\S}{r^\S}}{(r^\S)^2}\div^\S\fb +O(r^{-5}\epg\dg).
\eeaa

We introduce the notations
\beaa
h &:=&  \fb-\Up^\S f.
\eeaa
and rewrite the above equations in the form
\beaa
\Delta^\S\big(\div^\S h\big) &=& \Delta^\S\kab^\S +\Up^\S\Delta^\S\ka -\Delta^\S\kab +O\big( \epg \dg r^{-5} \big),
\eeaa
and
\beaa
\Delta^\S\ka &=& \Delta\ka -\frac{1}{(r^\S)^2}\div^\S h +O(r^{-5}\epg\dg),\\
\Delta^\S\kab &=& \Delta\kab +\frac{1-\frac{4m^\S}{r^\S}}{(r^\S)^2}\div^\S h +O(r^{-5}\epg\dg).
\eeaa
We infer that
\beaa
\left(\Delta^\S + \frac{2}{(r^\S)^2}\right)\div^\S h - \frac{6m^\S}{(r^\S)^3}\div^\S h  &=& \Delta^\S\kab^\S  +\Up^\S\Delta\ka   -\Delta\kab    +O\big( \epg \dg r^{-5} \big).
\eeaa
Projecting over  the  basis of  canonical   $\ell=1$  modes $\J^{(p, \S)} $,  integrating by parts and using, see  \eqref{eq:PropertiesofJ^S-strong2},
\beaa
\left(\lap^\S +\frac{2}{(r^\S)^2 } \right)  \J^{(p, \S)}= O\left(\frac{\epg}{(r^\S)^3 } \right)\J^{(p, \S)},
\eeaa
we deduce,
\beaa
 -\frac{ 6 m^\S}{(r^\S)^3} \int_\S  \div^\S h \J^{(p, \S)}+ \frac{O(\epg)}{(r^\S)^3 }\int_\S \div^\S h \J^{(p, \S)} = \int_\S \Big(\Delta^\S\kab^\S  +\Up^\S\Delta\ka   -\Delta\kab\Big)  \J^{(p, \S)} + O( \epg \dg r^{-3})
\eeaa
i.e.
\beaa
  \frac{ 6 m^\S }{(r^\S)^3} \int_\S  \div^\S h \J^{(p, \S)}&=& -\int_\S \Big(\Delta^\S\kab^\S  +\Up^\S\Delta\ka   -\Delta\kab\Big)  \J^{(p, \S)} +  O( \epg \dg r^{-3})
\eeaa 
or, 
\bea
\lab{eq:l-1mode-ofdivh}
 \int_\S  \div^\S h \J^{(p, \S)}= -\frac{(r^\S)^3 }{ 6m^\S}   \int_\S \Big(\Delta^\S\kab^\S  +\Up^\S\Delta\ka   -\Delta\kab\Big) \J^{(p, \S)}+O( \epg \dg).
\eea

Next, we focus on the RHS of \eqref{eq:l-1mode-ofdivh}. We have
\beaa
\int_\S\Delta\ka \J^{(p, \S)} &=& \int_\S\Delta\ka \Jp +\int_\S\Delta\ka (\J^{(p, \S)} -\Jp)\\
&=& \int_{\ovS}\Delta\ka \Jp +\left(\int_\S \Delta\ka \Jp - \int_{\ovS}\Delta\ka \Jp\right) +\int_\S\Delta\ka (\J^{(p, \S)} -\Jp).
\eeaa
Thus, using \eqref{eq:comparsonsJs}, {\bf A1-Strong} for $\ka$, and Lemma \ref {Lemma:coparison-forintegrals} applied to $F=(\Delta\ka ) \Jp$, we infer, proceeding as in Step 1, 
\beaa
\int_\S\Delta\ka \J^{(p, \S)} &=& \int_{\ovS}\Delta\ka \Jp +O(r^{-3}\epg\dg) = -\frac{2}{r^2}(\kac)_{\ell=1}  +O(r^{-3}\epg\dg).
\eeaa
Similarly, we obtain 
\beaa
\int_\S\Delta\kab \J^{(p, \S)} &=& -\frac{2}{r^2}(\kabc)_{\ell=1}  +O(r^{-3}\epg\dg).
\eeaa
This yields
\beaa
\int_\S \Big(\Delta^\S\kab^\S  +\Up^\S\Delta\ka   -\Delta\kab\Big) \J^{(p, \S)} &=& \int_\S \Delta^\S\kab^\S \J^{(p, \S)} +\Up^\S\int_\S \Delta\ka \J^{(p, \S)} - \int_\S\Delta\kab \J^{(p, \S)}\\
&=& -\frac{2}{(r^\S)^2}\int_\S \kabc^\S  \J^{(p, \S)} -\frac{2\Up^\S}{r^2}(\kac)_{\ell=1} +\frac{2}{r^2}(\kabc)_{\ell=1} +O(r^{-3}\epg\dg).
\eeaa
Since $|r-r^\S|\les \dg$, we infer
\beaa
\int_\S \Big(\Delta^\S\kab^\S  +\Up^\S\Delta\ka   -\Delta\kab\Big) \J^{(p, \S)} &=& -\frac{2}{(r^\S)^2}(\kabc^\S)_{\ell=1} -\frac{2\Up^\S}{(r^\S)^2}(\kac)_{\ell=1} +\frac{2}{(r^\S)^2}(\kabc)_{\ell=1} +O(r^{-3}\epg\dg).
\eeaa
Together with \eqref{eq:l-1mode-ofdivh}, we deduce
\beaa
 \int_\S  \div^\S h \J^{(p, \S)}= \frac{r^\S }{ 3m^\S}\Big((\kabc^\S)_{\ell=1} +\Up^\S(\kac)_{\ell=1} -(\kabc)_{\ell=1} \Big)+O( \epg \dg).
\eeaa
Recalling $h=  \fb-\Up^\S f$ and the definition of $\La, \Lab$, we deduce,
\beaa
\Lab-\Up^\S \La &=&\frac{ r^\S }{3m^\S} \Big( ( \kabc^\S)_{\ell=1} +   \Up^\S ( \kac)_{\ell=1} -  (\kabc)_{\ell=1} \Big)+ F_2(\La, \Lab)
\eeaa
with 
\beaa
|F_2(\La, \Lab)| &\les& \epg \dg
\eeaa
as stated. This concludes the proof of Lemma \ref{Lemma:GCMS2}.


  \subsection{Definition of angular momentum}
  

  The result  of Theorem \ref{theorem:ExistenceGCMS2}  is unique up to a rotation 
  of $\SSS^2$ in the definition of the canonical  $\ell=1$ modes on $\S$, see Definition \ref{definition:canmodesS}.  We can remove this final  ambiguity  
  by     choosing  that rotation    such that the $p=\pm$ component of $(\curl^\S \b^\S)_{\ell=1}$ vanish.   We state this result in the following  corollary
  of  Theorem  \ref{theorem:ExistenceGCMS2}.

  \begin{corollary}
  \lab{Corr:ExistenceGCMS2}
  Under the same assumptions as in Theorem \ref{theorem:ExistenceGCMS2}      we have,  in addition to   \eqref{def:GCMC2} and   \eqref{def:GCMC2-b}, 
  \begin{itemize}
  \item   either, for any choice of a canonical ${\ell=1}$ basis of $\S$, 
  \beaa
  (\curl^\S \b^\S)_{\ell=1}=0,
  \eeaa

  \item   or there exists a unique   $\ell=1$ basis of  canonical modes of $\S$ such that
   \bea
   \lab{angular-momentum}
   \int_\S  \curl^\S \b^\S\, \J^{(\pm, \S)}=0, \qquad  \int_\S  \curl^\S \b^\S\, \J^{(0, \S)}\neq 0.
   \eea 
  \end{itemize} 
  We then define   the angular   parameter $a^\S$ on $\S$ by the  formula\footnote{Note that in a Kerr space  $\KK(a, m)$, relative to  a geodesic foliation normalized on $\II^+$,   we have  $  \int_\S  \curl^\S  \b^\S  \J^{\pm, S}=0$ and    $  \int_\S  \curl^\S  \b^\S \J^{0, S} =\frac{8\pi a m}{  (r^\S)^3}+O(\frac{ma^2}{  (r^\S)^4}) $. } 
   \bea
 a^\S:=\frac{(r^\S)^3}{8\pi m^\S}   \int_\S  \curl^\S  \b^\S \J^{(0, \S)}.  
   \eea  
With this definition, we have $a^\S=0$ in the first case, while $a^\S\neq 0$ in the second case.
  \end{corollary}
  
  \begin{proof}  
  Recall    that the basis of canonical $\ell=1$ modes on $\S$ are given by  the formula
       \beaa
    \JpSS&=&J^{(p, \SSS^2)} \circ \Phi^{-1} 
    \eeaa 
   with $J^{(p, \SSS^2)}$          the $\ell=1 $ basis on $\SSS^2$  given by Definition \ref{definition:ell=1mpdesonS}, and $(\Phi, \phi) $ an effective  uniformization map of  $\S$. In particular, we have by change of variable 
   \beaa
   \int_\S \curl^\S \b^\S\, \J^{(p, \S)} &=& \int_{\SSS^2}e^{2\phi}\,\curl^\S \b^\S\circ\Phi\, J^{(p, \SSS^2)}
   \eeaa
and hence
   \beaa
   (\curl^\S \b^\S)_{\ell=1} &=& \int_{\SSS^2}e^{2\phi}\,\curl^\S \b^\S\circ\Phi\, x
   \eeaa
  where $x$ is the position vector on $\SSS^2$. Changing the  definition of $\ell=1$ modes by a rotation of $\SSS^2$ consist in changing $(\Phi, \phi) $ to $(\Phi\circ O, \phi\circ O)$ with $O\in O(3)$ which yields
 \beaa
 \int_{\SSS^2}e^{2\phi\circ O}\,\curl^\S \b^\S\circ\Phi\circ O\, x &=& \int_{\SSS^2}e^{2\phi}\,\curl^\S \b^\S\circ\Phi\, O^{-1}x\\
 &=& O^{-1}\left(\int_{\SSS^2}e^{2\phi}\,\curl^\S \b^\S\circ\Phi\, x\right)\\
 &=& O^{-1}(\curl^\S \b^\S)_{\ell=1}.
 \eeaa 
  In particular, $(\curl^\S \b^\S)_{\ell=1}$ is identified with a vector $v$ in $\RRR^3$, and changing the definition of $\ell=1$ modes by a rotation of $\SSS^2$ amounts to apply a rotation to $v$.  The proof  follows then  from the  fact that,  given  a   vector  $v$ in $\RRR^3$ with $v\neq 0$,  there exists a unique  rotation of   the $(x^1, x^2, x^3)$ coordinates   of  $\RRR^3$  such that    $v$   points in the direction of the $x^3$ axis.
  \end{proof}

  
  \subsection{Intrinsic GCM spheres in Kerr}
  

Consider  $\KK(a_0, m_0)$, $|a_0|< m_0$, a sub-extremal  Kerr spacetime
 endowed with an outgoing optical function $u$ normalized at  $\II^+$. We denote by 
 $S(u, s) $ the   spheres       of the induced  geodesic foliation, with $s$ the affine parameter, and  $r$ the area radius, normalized on $ \II^+$  such that $\lim_{r\to \infty} \frac{s}{r} =1$.    Let  $(e_4, e_3, e_2, e_2)$   be the associated null  frame with      $e_4 =-\g^{\a\b}\pr_\b u\pr_\a$.  Define also the  corresponding angular coordinates $(\th, \vphi)$,  properly normalized  at infinity,  with $e_4(\th)=e_4(\vphi)=0$,   and the corresponding   $\Jp$  defined by  them, i.e.
   \bea
   J^{(0)}=\cos \th, \qquad J^{(+)}=\sin\th \cos\vphi,\qquad J^{(-)}= \sin\th \sin\vphi.
   \eea
    Finally we consider  the spacetime region 
    \beaa
    \RR(r_0)  = \Big\{  r\ge  r_0\Big\} \subset \KK(a_0, m_0). 
 \eeaa 
   
\begin{lemma}
\lab{lemma:controlfarspacetimeregionKerrassumptionRR}
 If $r_0=r_0(m_0)$ is sufficiently large,  the region  $\RR(r_0)$, endowed with the geodesic foliation described above, verifies the assumptions {\bf  A1-Strong}, {\bf A2, A3} and {\bf A4-Strong},  as well as \eqref{Assumptions:theorem-ExistenceGCMS2},  with the smallness constants 
\beaa
\epg= \dg= \frac{a_0m_0}{r_0}.
\eeaa
\end{lemma}

\begin{proof}
{\bf  A1-Strong}, {\bf A2, A3} and {\bf A4-Strong}, as well as the estimate for $(\div\b)_{\ell=1}$ in \eqref{Assumptions:theorem-ExistenceGCMS2}, follow immediately from Lemma 2.10  in \cite{KS-Kerr1}. It then remains to prove  \eqref{Assumptions:theorem-ExistenceGCMS2} for $(\kac)_{\ell=1}$ and $(\kabc)_{\ell=1}$. Lemma 2.10  in \cite{KS-Kerr1} yields $\kac=O(a_0^2r^{-3})$ and $\kabc=O(a_0^2r^{-3})$. One can easily push the asymptotic to the next order to obtain 
\beaa
\kac = \frac{a_0^2}{r^3}\Big(c_1+c_2(\cos\th)^2\Big)+O\left(\frac{m_0a_0^2}{r^4}\right), \qquad \kabc = \frac{a_0^2}{r^3}\Big(\underline{c}_1+\underline{c}_2(\cos\th)^2\Big)+O\left(\frac{m_0a_0^2}{r^4}\right),
\eeaa
for some universal constants $c_1$, $c_2$, $\underline{c}_1$ and $\underline{c}_2$. Since 
\beaa
\Big(c_1+c_2(\cos\th)^2\Big)_{\ell=1}=0, \qquad \Big(\underline{c}_1+\underline{c}_2(\cos\th)^2\Big)_{\ell=1}=0,
\eeaa
we infer 
\beaa
(\kac)_{\ell=1}=O\left(\frac{m_0a_0^2}{r^2}\right), \qquad (\kabc)_{\ell=1}=O\left(\frac{m_0a_0^2}{r^2}\right),
\eeaa
which concludes the proof of \eqref{Assumptions:theorem-ExistenceGCMS2}.
\end{proof}

\begin{corollary}[Existence of intrinsic  GCM spheres in Kerr]
\lab{cor:ExistenceGCMS1inKerr} 
 If  $r_0\gg m_0$  is sufficiently large,  then any sphere $\ovS\subset \RR(r_0)$  admits a  unique deformation $\Psi:\ovS\longrightarrow \S$  verifying \eqref{def:GCMC2} and \eqref{def:GCMC2-b}. Moreover, the deformation  verifies the properties  \eqref{eq:ThmGCMS3}, \eqref{eq:ThmGCMS4},  \eqref{eq:ThmGCMS5} and \eqref{eq:ThmGCMS6} stated in Theorem \ref{Theorem:ExistenceGCMS1}.
\end{corollary} 

\begin{proof}
In view of Lemma \ref{lemma:controlfarspacetimeregionKerrassumptionRR}, the assumptions {\bf  A1-Strong}, {\bf A2, A3} and {\bf A4-Strong},  as well as \eqref{Assumptions:theorem-ExistenceGCMS2},  are satisfied by the spacetime region  $\RR(r_0)$ provided that $r_0=r_0(m_0)$  is sufficiently large, with the smallness constants $\epg$ and $\dg$ given by 
\beaa
\epg= \dg= \frac{a_0m_0}{r_0}.
\eeaa
Thus, Theorem \ref{theorem:ExistenceGCMS2} applies which concludes the proof of the corollary.
\end{proof}


\end{document}